\documentclass{amsart}
\usepackage{amsmath,amsthm,latexsym,amssymb,hyperref}
\usepackage{eucal}
\usepackage{mathrsfs}
\usepackage[all]{xy}
\allowdisplaybreaks
\numberwithin{equation}{section}

\newtheorem{thm}{Theorem}[section]
\newtheorem{lem}[thm]{Lemma}
\newtheorem{prop}[thm]{Proposition}
\newtheorem{cor}[thm]{Corollary}

\theoremstyle{definition}
\newtheorem{defn}[thm]{Definition}
\theoremstyle{remark}
\newtheorem{rmk}[thm]{Remark}
\newtheorem{ex}[thm]{Example}

\newtheorem{notn}[thm]{Notation}

\newdir{ >}{{}*!/-10pt/@{>}}
\newcommand\fib{\ar @{->>} [r]} 
\newcommand\cof{\ar @{ >->}[r]}

\newcommand \Om{\Omega}
\newcommand \n{\nabla}
\newcommand \del{\partial}
\newcommand \vp{\varphi}
\newcommand \ve{\varepsilon}
\newcommand{\sdr}[2]{\overset{#1}{\underset{#2}{\rightleftharpoons}}}
\newcommand \im{\operatorname{Im}}

\newcommand \sn[1]{(-1)\sp{#1}}
\newcommand\op{\mathcal}
\newcommand\cat{\mathbf}
\newcommand{\ob}{\operatorname{Ob}}

\newcommand{\rth}{$r^{\text{th}}$}

\newcommand{\hoch}{\mathscr H}

\renewcommand{\Bar}{\mathscr B}
\newcommand{\cohoch}{\widehat{\mathscr H}}

\newcommand{\si}{s^{-1}}

\newcommand{\tot}{\operatorname{Tot}}
\newcommand{\G}{\mathsf G}
\newcommand{\Z}{\mathsf Z}
\newcommand{\B}{\mathsf B}
\newcommand{\C}{\mathsf C}
\newcommand{\W}{\mathsf W}
\newcommand{\E}{\mathsf E}
\newcommand{\Eu}{\mathsf E^{\mathsf u}}
\renewcommand{\H}{\mathsf H}
\renewcommand{\P}{\mathsf P}
\newcommand{\Pu}{\mathsf P^{\mathsf u}}
\newcommand{\Q}{\mathsf Q}
\newcommand{\coH}{\widehat{\mathsf H}}
\newcommand{\zz}{\mathbb Z}
\newcommand{\Id}{\operatorname{Id}}
\renewcommand{\S}{\mathsf S^{2}}
\newcommand{\A}{\mathsf A}
\newcommand\M{\mathfrak M}

\begin{document}
\title [Power maps in algebra and topology]{Power maps in algebra and topology}
\author {Kathryn Hess}

\address{MATHGEOM \\
    \'Ecole Polytechnique F\'ed\'erale de Lausanne \\
    CH-1015 Lausanne \\
    Switzerland}
    \email{kathryn.hess@epfl.ch}

\date {\today }
 \begin{abstract} Given any twisting cochain $t:C\to A$, where $C$ is a connected, coaugmented chain coalgebra and $A$ is an augmented chain algebra over an arbitrary PID $R$, we construct a twisted extension of chain complexes 
 $$\xymatrix@1{A\cof&\hoch (t)\fib & C.}$$ 
We show that  both the well-known Hochschild complex of an associative algebra and the coHochschild complex of a coassociative coalgebra \cite{hps-cohoch} are special cases of $\hoch (t)$, which we therefore call the Hochschild complex of $t$.  We explore the extent of the naturality of the Hochschild complex construction and apply the results  of this exploration to determining conditions under which $\hoch(t)$ admits multiplicative or comultiplicative structure.   In particular, we show that the Hochschild complex on a chain Hopf algebra always admits a natural comultiplication.

Furthermore, when $A$ is a chain Hopf algebra, we determine conditions under which $\hoch(t)$ admits an $r^{\text{th}}$-power map extending the usual $r^{\text{th}}$-power map on $A$ and lifting the identity on $C$.  As special cases, we obtain that both the Hochschild complex of any cocommutative Hopf algebra and  the coHochschild complex of the normalized chain complex of a double suspension admit power maps.  We show moreover that if $K$ is a double suspension, then the power map on the coHochschild complex of the normalized chain complex of $K$ is a model of the topological power map on the space of free loops on $K$, illustrating the topological relevance of our algebraic construction.
 \end{abstract}
 
 \maketitle

\tableofcontents

\section {Introduction}\label{sec:introduction}

Let $X$ be a topological space, and let $\op L X$ denote its free loop space, i.e., the space of unbased maps from the circle into $X$.   Let $e:\op LX\to X$ denote the map sending a loop to its basepoint, i.e., $e(\ell)=\ell (z_{0})$, where $z_{0}$ is a fixed basepoint for the circle. 

For any $r\in \mathbb N$,  the free loop space $\op LX$ admits an \rth-power map given by iterated concatenation of loops with the same basepoint.  More precisely, let $\op L^r X$ denote the pullback in the following diagram.
$$\xymatrix{ \op L^rX\ar[d]^{e^{(r)}} \ar [r]&(\op LX)^{\times r}\ar [d]^{e^{\times r}}\\ X\ar [r]^{\Delta ^{(r)}}&X^{\times r}}$$
Let  $\Delta_{top}^{(r)}:\op LX\to \op L ^rX: \ell\mapsto (\ell,...,\ell)$.

Elements of $\op L^rX$ are sequences of $r$ loops in $X$, all with the same basepoint  in $X$.  Fixing an order of concatenation, we can define a continuous map $\widetilde\mu_{top} ^{(r)}:\op L^rX\to \op LX$.  The \rth-power map $\widetilde\lambda_{top}^{(r)}$ on $\op LX$ is equal to the composite
\begin{equation}\label{eqn:top-nth}
\op LX \xrightarrow {\Delta_{top} ^{(r)}} \op L^rX \xrightarrow {\widetilde\mu_{top} ^{(r)}} \op LX.
\end{equation}
In fact there is a commuting diagram of maps of fibration sequences
$$\xymatrix{\Om X\ar [d] \ar [r]^(0.45){\Delta_{top}^{(r)}}&(\Om X)^{\times r}\ar [d]\ar[r]^(0.55){\mu_{top} ^{(r)}}&\Om X\ar [d]\\
		\op LX\ar [d]^{e} \ar [r]^(0.47){\Delta_{top}^{(r)}}&\op L^rX\ar [d]^{e^{(r)}}\ar[r]^(0.53){\widetilde\mu_{top} ^{(r)}}&\op LX\ar [d]^{e}\\
		X\ar@{=} [r]&X\ar@{=}[r]&X,}$$
where $\mu_{top}^{(r)}:(\Om X)^{\times r} \to \Om X$ is the iterated multiplication map, given by concatenation of based loops  in the same order as that used in the definition of $\widetilde\mu_{top}^{(r)}$.  Letting $\lambda^{(r)}_{top}=\mu_{top} ^{(r)}\Delta_{top}^{(r)}$, we therefore have a commuting diagram of fibration sequences
\begin{equation}\label{eqn:fibseq}
\xymatrix{\Om X\ar[d]_{\lambda ^{(r)}_{top}}\cof &\op L X \ar[d]_{\widetilde\lambda ^{(r)}_{top}}\fib &X\ar@{=}[d]\\
		\Om X\cof &\op L X \fib &X.}
\end{equation}

Power maps on free loop spaces play an important role in algebraic topology, most notably in the definition of the topological cyclic homology, $TC(X;p)$,  of a space $X$, where $p$ is any prime \cite{bhm}.  Our goal in this paper is to describe a reasonably simple algebraic model of the \rth-power map on $\op LX$ (Theorem \ref{thm:topo-alg}), which is an essential building block in our construction of a chain model of $TC(X;p)$ in \cite{hess-rognes}.  In particular, for any  simply connected double suspension $X$, we define a small chain complex of abelian groups $fls_{*}(X)$ and an endomorphism $\widetilde\lambda_{r}$ of $fls_{*}(X)$ such that $H_{*}(fls_{*}(X))\cong H_{*}(\op LX)$ and such that $H_{*}\widetilde\lambda_{r}$ corresponds to $H_{*}\widetilde\lambda_{top}^{(r)}$ under this isomorphism.

\begin{rmk}  Power maps and the maps they induce in (co)homology are also sometimes called $\lambda$-operations, Adams operations or Frobenius operations.
\end{rmk}

\subsection{History}
It has been known for many years that the Hochschild homology of a (differential graded) commutative algebra also admits natural power operations \cite{gerstenhaber-schack:jpaa}, \cite{loday:inventiones}, which are closely related to the power operations on free loop spaces \cite{burgehelea-fiedorowicz-gajda}, \cite{vigue}.  These power operations in Hochschild homology can be constructed as follows.

If $H$ is a Hopf algebra over a ring $R$, then the $R$-module $\operatorname{End}(H)$ of endomorphisms of $H$ admits a multiplication $*$, called the \emph{convolution product}.   For all $f,g\in End(H)$, the convolution product of $f$ and $g$ is the composite
$$H\xrightarrow \delta H\otimes H\xrightarrow {f\otimes g} H\otimes H\xrightarrow \mu H,$$
where $\delta$ and $\mu$ are the comultiplication and multiplication on $H$, respectively.  The \emph{\rth-power map}  on $H$ is then just 
\begin{equation}\label{eqn:convolution}
\lambda_{r}=\Id_{H}^{*r}=\mu ^{(r)}\delta^{(r)},
\end{equation}
where $\delta^{(r)}:H\to H^{\otimes r}$ and $\mu^{(r)}:H^{\otimes r}\to H$ denote the iterated comultiplication and multiplication.

Recall that the Hochschild complex of an augmented algebra $A$, which we denote $\hoch(A)$, can be seen as a twisted extension of $A$ by $\Bar (A)$, the bar construction, i.e., there is a twisted tensor extension
\begin{equation}\label{eqn:hochbar}
\xymatrix{A\cof&\hoch (A) \fib & \Bar (A)}.
\end{equation}
If $A$ is commutative, then $\Bar(A)$ is naturally a commutative Hopf algebra, where the multiplication is the shuffle product.  The multiplication on $\Bar (A)$ lifts to $\hoch (A)$, so that (\ref{eqn:hochbar}) becomes a sequence of algebra maps. Moreover, the \rth-power map on $\Bar (A)$ also lifts to a linear map $\widetilde \lambda_{r}:\hoch(A) \to \hoch(A)$ such that
\begin{equation}\label{eqn:twistedext}
\xymatrix{ A\ar @{=}[d] \cof &\hoch(A)\ar [d]_{\widetilde \lambda_{r}}\fib&\Bar(A)\ar[d]_{\lambda_{r}}\\
		A \cof &\hoch(A)\fib&\Bar(A)}
\end{equation}
commutes.  The map induced by $\widetilde \lambda_{r}$ in Hochschild homology is the \rth-power map of \cite{gerstenhaber-schack:jpaa}, \cite{loday:inventiones}.

Compare the diagram (\ref{eqn:twistedext}) to the diagram below, obtained by applying the singular cochains functor $S^*$ to diagram (\ref{eqn:fibseq}).
\begin{equation}\label{eqn:cochains-fibseq}
\xymatrix{ S^*(X)\ar @{=}[d] \cof &S^*(\op LX)\ar [d]_{S^*\widetilde \lambda_{top}^{(r)}}\fib&S^*(\Om X)\ar[d]_{S^*\lambda_{top}^{(r)}}\\
		S^*(X)\cof &S^*(\op LX)\fib&S^*(\Om X)}
\end{equation}
Burghelea, Fiedorowicz and Gajda \cite{burgehelea-fiedorowicz-gajda} and Vigu\'e \cite{vigue}, building on work of Jones \cite{jones}, showed that the analogy between diagrams (\ref{eqn:twistedext}) and (\ref{eqn:cochains-fibseq}) is not purely formal, at least rationally, when they proved the following result.  

Let $X$ be a simply connected simplicial complex, and let $\Bbbk$ be a field of characteristic zero.  Let $\op A_{PL}(X)$ denote the commutative cochain algebra of piecewise linear differential forms on $X$ with coefficients in $\Bbbk$.  Let $\widetilde {HH}_{*}\big(\op A_{PL}(X)\big)$ denote the reduced homology of $\hoch\big(\op A_{PL}(X)\big)$. There is an isomorphism of graded $\Bbbk$-vector spaces
$$a:\widetilde {HH}_{-*}\big(\op A_{PL}(X)\big) \xrightarrow \cong \widetilde H^*(\op LX;\Bbbk)$$ such that 
$$\xymatrix{\widetilde {HH}_{-*}\big(\op A_{PL}(X)\big)\ar [d]_{H_{-*}\widetilde\lambda _{r}}\ar [r]^{a}_{\cong}&\widetilde H^*(\op LX;\Bbbk)\ar[d]^{H^*\widetilde \lambda_{top}^{(r)}}\\
 \widetilde {HH}_{-*}\big(\op A_{PL}(X)\big)\ar [r]^{a}_{\cong}&\widetilde H^*(\op LX;\Bbbk)}$$
 commutes.
 
Our construction in this paper of the integral chain model $fls_{*}(X)$ of the free loop space and of the power map on $fls_{*}(X)$ is inspired by this rational result.  In particular, as explained below, we also work with Hochschild-type complexes.  

\subsection{Summary of results}
We begin this paper by laying the theoretical foundations of our algebraic model of the free loop space and its power operations.

Let $\Om$ and $\Bar$ denote the cobar and bar construction functors, respectively. For any twisting cochain $t:C\to A$, we define a chain complex $\hoch (t)$ that is a  twisted tensor extension of $A$ by $C$.  The complex $\hoch (t)$ generalizes both the well-known Hochschild complex $\hoch (A)$ of a chain algebra $A$ and the coHochschild complex $\cohoch (C)$ of a chain coalgebra $C$ \cite{hps-cohoch}.  
\begin{itemize}
\item If $t_{\Bar}:\Bar A \to A$ is the couniversal twisting cochain associated to $A$, then $\hoch (t_{\Bar})= \hoch (A)$.
\item If $t_{\Om }:C\to \Om C$ is the universal twisting cochain associated to $C$, then $\hoch (t_{\Om})=\cohoch (C)$.
\end{itemize}
We first show that the construction $\hoch (t)$ is natural with respect to the most obvious notion of morphisms of twisting cochains:  pairs $(f,g)$, where $f$ is a map of coalgebras and $g$ is a map of algebras, commuting with the twisting cochains.  In particular, any twisting cochain $t:C\to A$ induces a chain map $\cohoch(C) \to \hoch(A)$.

We then prove that the Hochschild complex construction admits an extended naturality, with respect to pairs of maps $(f,g)$, where either $f$ is a map of coalgebras up to strong homotopy (Theorem \ref{thm:hoch-extnatl}) or $g$ is a map of algebras up to strong homotopy (Theorem \ref{thm:hoch-extnatl-dual}).  We point out that the natural section $A\to \Om \Bar A$ of the counit $\Om \Bar A\to A$ of the bar/cobar adjunction is a map of algebras up to strong homotopy, while the natural retration $\Bar \Om C\to C$ of the unit map $C\to \Bar \Om C$ is a map of coalgebras up to strong homotopy.  Consequently,  the natural chain map $\cohoch(C)\to \hoch(\Om C)$ admits a retraction, while $\cohoch(\Bar A)\to \hoch(A)$ admits a section (Corollaries \ref{cor:dcsh-unit} and \ref{cor:dash-counit}).

We are interested in determining conditions that guarantee the existence of operations and cooperations on the Hochschild complex, which motivates us to study certain types of (co)algebras with additional structure, known as Alexander-Whitney (co)algebras.  A chain coalgebra $C$ is an Alexander-Whitney coalgebra if its comultiplication map is a map of coalgebras up to strong homotopy, and the higher homotopies induce a coassociative comultiplication on $\Om C$.  Every Alexander-Whitney coalgebra $C$ is therefore a Hirsch coalgebra \cite{kadeishvili}, but, as we show (Example \ref{ex:nonrealAW}), not all Hirsch coalgebras are Alexander-Whitney coalgebras.  Dually, a chain algebra $A$ is an Alexander-Whitney algebra if its multiplication map is a map of algebras up to strong homotopy, and the higher homotopies induce an associative multiplication on $\Bar A$.  Every Alexander-Whitney algebra $A$ is therefore a Hirsch algebra.  Alexander-Whitney (co)algebras are special types of $B_{\infty}$-(co)algebras \cite{getzler-jones}, \cite {baues2}.

We prove that if $H$ is a chain Hopf algebra, then $\Bar H$ is an Alexander-Whitney coalgebra and $\Om H$ is an Alexander-Whitney algebra (Theorem \ref{thm:bar-hopf}), improving a result of Kadeishvili \cite{kadeishvili}.  In the course of the proof, we establish a result that is interesting in and of itself (Theorem \ref{thm:aw-bar-dcsh}):  for every pair of chain algebras $A$ and $A'$, the natural ``Alexander-Whitney'' map $\Bar (A\otimes A')\to \Bar A \otimes \Bar A'$ is map of coalgebras up to strong homotopy.

As a consequence of the extended naturality of the Hochschild complex construction, we obtain that if $t:C\to H$ is a twisting cochain such that $C$ is an Alexander-Whitney coalgebra and $H$ is a chain Hopf algebra, then $\hoch(t)$ admits  a comultiplication extending that on $H$ and lifting that on $C$ (Theorem \ref{thm:hoch-comult}).  Dually, if $t:H\to A$ is a twisting cochain such that $A$ is an Alexander-Whitney algebra and $H$ is a chain Hopf algebra, then $\hoch (t)$ admits a multiplication extending that on $A$ and lifting that on $H$ (Theorem \ref{thm:hoch-mult}).  In particular, if $C$ is an Alexander-Whitney coalgebra, then $\cohoch(C)$ admits a comultiplication, while if $A$ is an Alexander-Whitney algebra, then $\hoch (A)$ admits a multiplication.

The algebraic heart of this article concerns the existence of power maps on the Hochschild complex of a twisting cochain.  We show that, under certain cocommutativity conditions, if $t:C\to H$ is a twisting cochain, where $C$ is a Hirsch coalgebra and $H$ is a chain Hopf algebra, then $\hoch(t)$ admits an \rth-power map $\widetilde\lambda_{r}$ extending the usual \rth-power map on $H$ and lifting the identity on $C$ (Theorem \ref{thm:exists-nthpower}). In particular, if $H$ is a cocommutative Hopf algebra, then $\hoch (H)$ admits an \rth-power map extending the usual \rth-power map on $H$ and lifting the identity map on $\Bar H$ (Corollary \ref{cor:nth-hoch}).  Dually, if $C$ is a Hirsch coalgebra such that associated comultiplication on $\Om C$ is cocommutative, then $\cohoch(C)$ admits an \rth-power map extending the usual \rth-power map on $\Om C$ and lifting the identity on $C$ (Corollary \ref{cor:nth-cohoch}).  We also show that the natural map $\cohoch(C) \to \hoch (H)$ induced by the twisting cochain $t$ commutes with the \rth-power maps. 

Once we have laid these algebraic foundations, we proceed to apply the Hochschild complex construction to building an integral algebraic model of the free loop space and its power operations.   We begin by recalling the cyclic bar construction from \cite{waldhausen}, which is a simplicial model of the free loop space on the classifying space of a topological group, and the B\"okstedt-Hsiang-Madsen model of the power maps in this model (Proposition \ref{prop:bhm})\cite{bhm}.  The model given in \cite{bhm} for the \rth-power map is not simplicial, but we show that it is homotopic to the realization of a purely simplicial map that mimics loop concatenation (Proposition \ref{prop:john}).  As a consequence, we obtain a particularly simple formula for the \rth-power map in the B\"okstedt-Hsiang-Madsen model (Corollary \ref{cor:simpl-rth}). 

Applying Artin-Mazur totalization \cite{artin-mazur}, we pass from the cyclic bar construction to a ``simplicial Hochschild'' model of the free loop space on the classifying space of a simplicial group (Theorem \ref{thm:tot-zg}), which is to the cyclic bar construction as the simplicial classifying space $\overline\W$ functor is to the simplicial bar construction functor $\B$. The induced \rth-power map on the simplicial Hochschild model also has a relatively simple expression (Remark \ref{rmk:simpl-rth}).

Considering only those simplicial groups that are obtained by applying Kan's loop group functor $\G$ to a reduced simplicial set $K$, we can then simplify even further the simplicial free loop space model, constructing a ``simplicial coHochschild'' model, denoted $\widehat \H K$, of the free loop space on the classifying space of $\G K$ (Theorem \ref{thm:simpl-cohoch-to-hoch}).  The expression for the induced \rth-power map  on $\widehat \H K$ is extremely simple (Remark \ref{rmk:rth-simpl-cohoch}).

We then prove that $C_{*}\widehat \H K$ is weakly equivalent to the singular chain complex of the free loop space on the classifying space of the realization of $\G K$ (Theorem \ref{thm:small-to-big}).  Moreover, the weak equivalence commutes with the chain maps induced by the \rth-power maps. 

It follows from a theorem in \cite{hps-cohoch} that if $K$ is a $1$-reduced simplicial set, then $\cohoch(C_{*}K)$ is weakly equivalent to  $C_{*}\widehat \H K$ (Theorem \ref{thm:cohoch-ez}).  We show here that this weak equivalence respects the \rth-power maps in $\cohoch (C_{*}K)$ and in $C_{*}\widehat \H K$ (Theorem \ref{thm:compat}), at least when $K$ is a simplicial double suspension. 

Finally, we conclude that if $X$ is a simply connected topological space with the homotopy type of the realization of a simplicial double suspension $K$, then  $\big (\cohoch(C_{*}K), \widetilde\lambda _{r}\big)$ is an algebraic model of $\big(\op LX, \widetilde \lambda _{top}^{(r)}\big)$ (Theorem \ref{thm:topo-alg}).  We apply this result to computing \rth-power maps in $H_{*}(\op L S^n)$ for $n\geq 2$. 

\begin{rmk} The \rth-power maps on the Hochschild complex of a cocommutative Hopf algebra should induce a Hodge-type decomposition of its Hochschild homology, at least in characteristic zero. It would be interesting to study this decomposition.
\end{rmk}

\begin{rmk}  We suspect that it is possible to generalize the constructions here to higher order Hochschild complexes, in the sense of Pirashvili and Ginot \cite{pirashvili}, \cite{ginot}.
\end{rmk}

\subsection{Notation and conventions}
\begin{itemize}

\item Throughout this paper we are working over a principal ideal domain $R$.  We denote  the category of (non-negative) chain complexes over $R$ by $\cat{Ch}_R$, the category of augmented chain algebras over $R$ by $\cat {Alg}_{R}$, the category of coaugmented, connected chain coalgebras by $\cat{Coalg}_{R}$ and the category of connected chain Hopf algebras by $\cat {Hopf}_{R}$.  The underlying graded modules of all chain complexes are assumed to be $R$-free.  
The degree of an element $v$ of a chain complex $V$ is denoted $|v|$. 

Given chain complexes $(V,d)$ and $(W,d)$, the notation
$f:(V,d)\xrightarrow{\simeq}(W,d)$ indicates that $f$ induces an isomorphism in homology. 
In this case we refer to $f$ as a \emph {quasi-isomorphism}.

Let $f,g:A\to A'$ be morphisms of chain algebras.  A \emph{derivation homotopy} from $f$ to $g$ consists of a chain homotopy $H:A\to A'$ from $f$ to $g$ such that $H(ab)=H(a)f(b)+(-1)^{|a|}g(a) H(b)$ for all $a,b\in A$.
\item Throughout this article we apply the Koszul sign convention for commuting elements  of a graded module or for commuting a morphism of graded modules past an element of the source module.  For example,  if $V$ and $W$ are graded algebras and $v\otimes w, v'\otimes w'\in V\otimes W$, then $$(v\otimes w)\cdot (v'\otimes w')=(-1)^{|w|\cdot |v'|}vv'\otimes ww'.$$ Furthermore, if $f:V\to V'$ and $g:W\to W'$ are morphisms of graded modules, then for all $v\otimes w\in V\otimes W$, 
$$(f\otimes g)(v\otimes w)=(-1)^{|g|\cdot |v|} f(v)\otimes g(w).$$
\item The \emph {suspension} endofunctor $s$ on the category of graded modules is defined on objects $V=\bigoplus _{i\in \mathbb Z} V_ i$ by
$(sV)_ i \cong V_ {i-1}$.  Given a homogeneous element $v$ in
$V$, we write $sv$ for the corresponding element of $sV$. The suspension $s$ admits an obvious inverse, which we denote $\si$.

\item Let $T$ denote the endofunctor on the category of free graded $R$-modules given by
$$TV=\oplus _{n\geq 0}V^{\otimes n},$$
where $V^{\otimes 0}=R$.  An element of the summand $V^{\otimes n}$ of $TV$ is a sum of terms denoted $v_{1}|\cdots |v_{n}$, where $v_{i}\in V$ for all $i$. 
\item The \emph{bar construction} functor $\Bar: \cat{Alg}_{R}\to \cat {Coalg}_{R}$ is defined by 
$$\Bar A=\left(T (s\overline A), d_{\Bar}\right)$$
where $\overline A$ denotes the augmentation ideal of $A$, and if $d$ is the differential on $A$, then
\begin{align*}
d_{\Bar}(sa_{1}|\cdots|sa_{n})=&\sum _{1\leq j\leq n}\pm sa_{1}|\cdots |s(da_{j})|\cdots |sa_{n}\\ 
&+\sum _{1\leq j<n}\pm sa_{1}|...|s(a_{j}a_{j+1})|\cdots |sa_{n}.
\end{align*} 
The graded $R$-module underlying $\Bar A$ is naturally a cofree coassociative coalgebra, with comultiplication given by splitting of words. 
\item The \emph{cobar construction} functor $\Om :\cat {Coalg}_{R}\to \cat {Alg}_{R}$ is defined by 
$$\Om C= \left(T (\si \overline C), d_{\Om}\right)$$
where $\overline C$ denotes the coaugmentation coideal of $C$, and if $d$ denotes the differential on $C$, then
\begin{align*}
d_{\Om}(\si c_{1}|\cdots|\si c_{n})=&\sum _{1\leq j\leq n}\pm \si c_{1}|\cdots |\si (dc_{j})|\cdots |\si c_{n}\\ 
&+\sum _{1\leq j\leq n}\pm \si c_{1}|...|\si c_{ji}|\si c_{j}{}^{i}|\cdots |\si c_{n},
\end{align*}
with signs determined by the Koszul rule, where the reduced comultiplication applied to $c_{j}$ is $c_{ji}\otimes c_{j}{}^{i}$ (using Einstein implicit summation notation).  The graded $R$-module underlying $\Om C$ is naturally a free associative algebra, with multiplication given by concatenation.
\item The category of simplicial sets is denoted $\cat {sSet}$ in this article.  Its full subcategory of reduced simplicial sets (i.e., simplicial sets with a unique 0-simplex) is denoted $\cat {sSet}_{0}$, while the category of pointed simplicial sets and basepoint-preserving simplicial maps is denoted $\cat {sSet}_{*}$.  Observe that $\cat {sSet}_{0}$ can naturally be viewed as a full subcategory of $\cat {sSet}_{*}$.  Objects in $\cat {sSet}$, $\cat {sSet}_{*}$, and $\cat {sSet}_{0}$ are usually denoted $M$, $L$ and $K$, respectively, in this paper.
\item The normalized chains functor from simplicial sets to chain complexes is denoted $C_{*}$, while the singular simplices functor from topological spaces to simplicial sets is denoted $\mathcal S_{\bullet}$.  Their composite, $C_{*}\circ \mathcal S_{\bullet}$, is denoted $S_{*}$.  The left adjoint to $\mathcal S_{\bullet}$, i.e.,  geometric realization, is denoted $|-|$.

\end{itemize}

\subsection*{Acknowledgment}  The author would like to express her heartfelt appreciation and deep gratitude to John Rognes for his participation in extensive discussions of this project over the course of the past several years.  His contributions were particularly crucial to  the completion of Section \ref{sec:geometry}.


\section{The Hochschild complex of a twisting cochain}

We begin this section by recalling the definition and certain well-known examples of twisting cochains.  We observe that twisting cochains are the objects of a category $\cat {Twist}$, which admits an interesting monoidal structure. We then define the Hochschild complex functor $\hoch: \cat {Twist}\to \cat {Ch}_{R}$, which turns out to be strongly monoidal, and consider certain important special cases.  

Weakening the definition of morphisms in $\cat {Twist}$ somewhat, we then consider two faithful, wide embeddings (i.e.,  injective on morphisms and bijective on objects) $\cat {Twist}\hookrightarrow \cat {Twist^{sh}}$ and  $\cat {Twist}\hookrightarrow \cat {Twist_{sh}}$ and show that the Hochschild complex functor extends over both $\cat {Twist^{sh}}$ and $\cat {Twist_{sh}}$.  The extended naturality of the Hochschild construction that we obtain in this manner plays an essential role in the later sections of the paper.

\subsection{Twisting cochains: definition and examples}

Seen as functors from coalgebras to algebras and vice-versa, the cobar and bar constructions form an adjoint pair $\Om \dashv \Bar$.  Let $\eta: \Id\to \Bar \Om$ denote the unit of this adjunction.  It is well known that for all connected, coaugmented chain coalgebras $C$, the counit map
\begin{equation}\label{eqn:unit-barcobar}
\eta_{C}:C\xrightarrow\simeq\Bar\Om C
\end{equation}
is a quasi-isomorphism of chain coalgebras.   

Dually, let $\ve: \Om \Bar \to \Id$ denote the counit of this adjunction.  For all augmented chain algebras $A$, the unit map
\begin{equation}\label{eqn:counit-barcobar}
\ve_{A}:\Om \Bar A\xrightarrow\simeq A
\end{equation}
is a quasi-isomorphism of chain algebras.

\begin{defn}
A \emph{twisting cochain} from a connected, coaugmented chain coalgebra $(C,d)$ with comultiplication $\Delta$ to an augmented chain algebra $(A,d)$ with multiplication $m$ consists of a linear map $t:C\to A$ of degree $-1$ such that
$$dt+td=m (t\otimes t)\Delta.$$
\end{defn}

\begin{rmk}\label{rmk:twisting-induced}
A twisting cochain $t:C\to A$ induces both a chain algebra map
$$\alpha _{t}:\Om C\to A$$
specified by $\alpha _{t}(\si c)=t(c)$ and a chain coalgebra map
$$\beta _{t}:C\to \Bar A, $$
 satisfying
 $$\alpha_{t}=\ve_{A}\circ \Om \beta_{t}\quad\text{and}\quad \beta _{t}=\Bar\alpha_{t}\circ \eta_{C}.$$
It follows that $\alpha_{t}$ is a quasi-isomorphism if and only if $\beta_{t}$ is a quasi-isomorphism.  
\end{rmk}
 
 \begin{ex} Let $C$ be a connected, coaugmented chain coalgebra. The \emph{universal  twisting cochain}
$$t_{\Om}:C\to \Om C$$
is defined by $t_{\Om }(c)=s^{-1} c$ for all $c\in C$, where $\si c$ is defined to be $0$ if $|c|=0$.  Note that  $\alpha_{t_{\Om}}=\Id_{\Om C}$, so that $\beta _{t_{\Om}}=\eta_{C}$.  Moreover, $t_{\Om}$ truly is universal, as all twisting cochains $t:C\to A$ factor through $t_{\Om}$, since the diagram
 $$\xymatrix{C\ar[r] ^{t_{\Om}}\ar [dr]_{t}&\Om C\ar [d]^{\alpha_{t}}\\ &A}$$
 always commutes.
 \end{ex}
 
  \begin{ex} Let $A$ be an augmented chain algebra. The \emph{couniversal  twisting cochain}
$$t_{\Bar}:\Bar A\to A$$
is defined by $t_{\Bar  }(sa)=a$ for all $a\in A$ and $t_{\Bar}(sa_{1}|\cdots |sa_{n})=0$ for all $n\not= 1$.  Note that  $\beta_{t_{\Bar}}=\Id_{\Bar A}$, so that $\alpha _{t_{\Bar}}=\ve_{A}$.  Moreover, $t_{\Bar}$ truly is couniversal, as all twisting cochains $t:C\to A$ factor through $t_{\Bar}$, since the diagram
 $$\xymatrix{&\Bar A\ar [d]^{t_{\Bar}}\\ C\ar[ur]^{\beta_{t}}\ar [r]_{t}&A}$$
 always commutes.
 \end{ex}

 \begin{ex} \label{ex:szczarba} Let $K$ be a reduced simplicial set, and let $\G K$ denote its Kan loop group. In 1961 \cite {szczarba}, Szczarba gave an explicit formula for a twisting cochain
$$t_{K}:C_{*}K\to C_{*}\G K,$$
natural in $K$ that induces a chain algebra map
\begin{equation}\label{eqn:szczarba}
\alpha _{K}:=\alpha_{t_{K}}:\Om C_{*}K\to C_{*}\G K.
\end{equation}
As shown in \cite{hess-tonks}, $\alpha_{K}$ factors naturally through an ``extended cobar construction,'' $\widehat \Om C_{*}K$, of which the usual cobar construction is a chain subalgebra, i.e., there is a commuting diagram of chain algebra maps
$$\xymatrix{\Om C_{*}K \ar [rr]^{\alpha_{K}}\ar@{>->}[dr]&& C_{*}\G K.\\ &\widehat \Om C_{*}K \ar [ur]_{\widehat \alpha_{K}}}$$
Moreover, $\widehat \alpha_{K}$ admits a natural retraction $\rho_{K}$ such that $\widehat \alpha_{K}\rho_{K}$ is chain homotopic to the identity on $C_{*}\G K$.  In particular, $\widehat\alpha_{K}$ is a quasi-isomorphism for all reduced simplicial sets $K$.  Since $\widehat \Om C_{*}K =\Om C_{*}K$ if $K$ is actually 1-reduced, it follows that
that $\alpha_{K}$ itself is a quasi-isomorphism if $K$ is 1-reduced.
\end{ex}

\begin{rmk}  If $t:C\to A$ is a twisting cochain, $f:C'\to C$ is a chain coalgebra map and $g:A\to A'$ is a chain algebra map, then $gtf:C'\to A'$ is also a twisting cochain.
\end{rmk}

\begin{notn}\label{notn:twist}  Let $\cat{Twist}$ denote the category such that
\begin{itemize}
\item $\ob\cat{Twist}=\{t:C\to A\mid t \text{ twisting cochain}\}$, and
\item  if $t:C\to A$ and $t':C'\to A'$ are twisting cochains, then
$$\cat{Twist}(t,t')=\{ (f,g)\in \cat{Coalg}_{R}(C,C')\times \cat {Alg}_{R}(A,A')\mid g\circ t=t'\circ f\}.$$
\end{itemize}
Composition of morphisms in $\cat {Twist}$ is defined componentwise.
\end{notn}

\begin{rmk} Note that $(f,g)\in \cat{Twist}(t,t')$ if and only if $\Bar g \circ \beta _{t}=\beta _{t'}\circ f$, which is true if and only if $g\circ \alpha _{t}=\alpha _{t'}\circ \Om f$ .
\end{rmk}

Later in this paper we are led to consider the following variant of $\cat {Twist}$.  Below, and elsewhere in this paper, if $(H,\delta)$ denotes a chain Hopf algebra, then $H$ is the underlying chain algebra, and $\delta$ is the comultiplication.

\begin{notn}\label{notn:twist-hopf}  Let $\cat{Twist}_{\text{Hopf}}$ denote the category with 
\begin{itemize}
\item $\ob\cat{Twist}_{\text{Hopf}} =\big\{ \big(C\xrightarrow t H, (H, \delta)\big)\mid t \in \ob \cat {Twist}, (H,\delta) \in\ob \cat{Hopf}_{R}\}$, and
\item if $\big(t, (H, \delta)\big)$ and $\big(t', (H', \delta')\big)$ are objects in $\cat{Twist}_{\text{Hopf}}$, then
$$\cat{Twist}_{\text{Hopf}}(t,t')=\{(f,g)\in \cat {Twist}(t,t')\mid (g\otimes g)\delta =\delta' g\}.$$
\end{itemize}
\end{notn}

The proposition below gives a categorical formulation of the universality of $t_{\Om}$ and of the couniversality of $t_{\Bar}$.

\begin{prop} If $S:\cat{Twist}\to \cat {Coalg}_{R}$ is the functor  that projects onto the source of a twisting cochain and $U:\cat{Coalg}_{R}\to \cat {Twist}$ is the ``universal twisting cochain functor,'' specified by $U(C)=t_{\Om}:C\to \Om C$ and $U(f)=(f,\Om f)$, then $U$ is left adjoint to $S$.  

Similarly, if $T:\cat{Twist}\to \cat {Alg}_{R}$ is the functor  that projects onto the target of a twisting cochain and $V:\cat{Alg}_{R}\to \cat {Twist}$ is the ``couniversal twisting cochain functor,'' specified by $V(A)=t_{\Bar}:\Bar A\to A$ and $V(g)=(\Bar g,g)$, then $V$ is right adjoint to $T$.  
\end{prop}

\begin{proof} Note that $(f,g)\in \cat{Twist}\big( U(C'), t\big)$ implies that 
$$\xymatrix{C'\ar[d]_{t_{\Om}}\ar [r]^f &C\ar[d]^t\\ \Om C'\ar [r]^g&A}$$
commutes.  It follows that $g=\alpha_{t}\circ \Om f$, since the graded algebra underlying $\Om C'$ is free, and $g$ is therefore determined by its values on the generators $\si \overline {C'}$.

The natural isomorphism 
$$\zeta:\cat {Coalg}_{R}\big(C', S(t)\big)\xrightarrow \cong \cat{Twist}\big( U(C'), t\big)$$ 
for $C'\in \ob \cat {Coalg}_{R}$ and a twisting cochain $t:C\to A$ is thus defined by $\zeta (f)=(f,\alpha_{t}\circ \Om f)$, with inverse $\zeta^{-1}$ defined by $\zeta^{-1}(f,g)=f$.

The proof that $V$ is right adjoint to $T$ is similar.
\end{proof}

There is an important binary operation on the set of twisting cochains, defined as follows.

\begin{defn}\label{defn:cartesian} Let $t:C\to A$ and $t':C'\to A'$ be twisting 
cochains.  Let $\ve: C\to R$ and $\ve ':C'\to R$ be the counits (augmentations), 
and let $\eta :R\to A$ and $\eta ':R\to A'$ be the units (coaugmentations). Set 
$$t*t'=t\otimes \eta '\ve '+ \eta\ve\otimes t':C\otimes C'\to A\otimes A'$$
Then $t*t'$ is a twisting cochain, called the \emph{cartesian product} 
of $t$ and $t'$.  
\end{defn}

\begin{ex}\label{ex:milgram}  An important special case of the cartesian product of twisting cochains is
$$t_{\Om}*t_{\Om}:C\otimes C'\to \Om C\otimes \Om C',$$
for $C,C'\in \cat {Coalg}_{R}$.  To simplify notation, we write
\begin{equation}\label{eqn:milgram}
q=\alpha _{t_{\Om}*t_{\Om}}:\Om (C\otimes C')\to \Om C\otimes \Om C'.
\end{equation}
Milgram proved in \cite{milgram} that the chain algebra map $q$ was a quasi-isomorphism if $C$ and $C'$ were simply connected, i.e., connected and $C_{1}=C'_{1}=0$.  In \cite{hps2} it was shown that $q$ is in fact a chain homotopy equivalence, for all $C,C'\in \cat {Coalg}_{R}$. 

Dually, for all $A,A'\in \cat{Alg}_{R}$, there is a  map of chain coalgebras that is a chain homotopy equivalence 
\begin{equation}\label{eqn:milgram-dual}
\nabla=\beta _{t_{\Bar}*t_{\Bar}}:\Bar A\otimes \Bar A'\to \Bar(A\otimes A').
\end{equation}
The equivalence $\nabla$ is often called the ``Eilenberg-Zilber'' equivalence, by analogy with the Eilenberg-Zilber equivalence of algebraic topology.   We refer the reader to Appendix \ref{sec:bar-hopf} for further details of this equivalence. 
\end{ex}

\begin{rmk}\label{rmk:cartesian-twist}  Let $t:C\to A$ and $t':C'\to A'$ be  twisting cochains.  Note that 
$$\alpha_{t*t'}=(\alpha_{t}\otimes \alpha_{t'})\circ q:\Om (C\otimes C')\to A\otimes A'$$
and that 
$$\beta_{t*t'}=\nabla \circ(\beta_{t}\otimes \beta _{t'}):C\otimes C'\to \Bar (A\otimes A').$$
\end{rmk}

\begin{rmk} Endowed with the cartesian product of twisting cochains, the category $\cat {Twist}$ is clearly monoidal, where the unit object is the zero map $0:R\to R$.
\end{rmk}

\subsection{Definition of the Hochschild complex}

We can now define the Hochschild construction functor $\hoch: \cat{Twist} \to \cat {Ch}_{R}$ that we study and apply throughout the remainder of this paper.  We begin by introducing a somewhat more general construction, which also englobes the other familiar constructions of chain complexes built from twisting cochains. 

\begin{defn}\label{defn:Hochschild-cx}  Let $t:C\to A$ be a twisting cochain.  Let $N$ be a $C$-bicomodule with left $C$-coaction $\lambda$ and right $C$-coaction $\rho$, and let $M$ be an $A$-bimodule.   Let $d$ denote the differentials on both $M$ and $N$. The \emph{Hochschild complex} of $t$ with coefficients in $N$ and $M$, denoted $\hoch_{t} (N,M)$, is the chain complex with underlying graded $R$-module $N\otimes M$ and with differential
$d_{t}$, defined on $y\otimes x\in N\otimes M$ by
\begin{align*}
d_{t}(y\otimes x)=&\; dy\otimes x +  \sn {|y|} y\otimes dx\\
& -\sn {|y_{j}|} y_{j}\otimes t(c^{j})\cdot x +\sn{(|c_{i}|-1)(|y^{i}| +|x|)} y^{i}\otimes x\cdot t(c_{i}),
\end{align*}
where $\cdot$ denotes the left and right actions of $A$ on $M$, $\lambda (y)= c_{i}\otimes y^{i}$ and $\rho(y)=y_{j}\otimes c^j$ (using Einstein summation notation).
\end{defn}

It is a matter of straightforward computation, using the definition of twisting cochains, to show that $d_{t}^2=0$, i.e., that $\hoch _{t}(N,M)$ really is a chain complex.  Furthermore, if $M$ is augmented over $R$ and $N$ is coaugmented over $R$, there is a twisted extension of chain complexes
$$\xymatrix@1{M\cof &\hoch_{t}(N,M) \fib & N}.$$

\begin{rmk} The signs in the definition of $d_{t}$ are simply those given by the Koszul rule (cf.~Introduction).
\end{rmk}

\begin{notn}  When $N=C$, seen as a bicomodule over itself via its comultiplication, and $M=A$, seen as a bimodule over itself via its multiplication, then we write
$$\hoch(t):=\hoch_{t}(C,A)=(C\otimes A, d_{t}).$$
\end{notn}

\begin{ex}\label{ex:hoch} If $A\in \cat{Alg}_{R}$ and $M$ is an $A$-bimodule, then $\hoch_{t_{\Bar}}(\Bar A, M)$ is exactly the usual Hochschild complex on $A$ with coefficients in $M$.  In particular, 
$$\hoch(t_{\Bar})=\hoch(A)$$ 
is the usual Hochschild complex  on $A$.
\end{ex} 

\begin{ex}\label{ex:cohoch} If $C\in \cat{Coalg}_{R}$ and $N$ is a $C$-bicomodule, then $\hoch_{t_{\Om}}(N, \Om C)$ is exactly the coHochschild complex on $C$ with coefficients in $N$, as defined in \cite{hps-cohoch}.  In particular, 
$$\hoch(t_{\Om})=\cohoch (C)$$
is the coHochschild complex on $C$.
\end{ex} 

\begin{ex} The usual twisted extension over a twisting cochain $t:C\to A$ of a right $C$-module $N$ by a left $A$-module  $M$ is a special case of the Hochschild complex construction defined above.  It suffices to  consider $N$ as a $C$-bicomodule with trivial left $C$-coaction and $M$ as an $A$-bimodule with trivial right $A$-action.
\end{ex}

\begin{prop} \label{prop:hoch-natl} The Hochschild complex construction extends to a functor
$$\hoch: \cat {Twist} \to \cat {Ch}_{R}.$$
\end{prop}

\begin{proof} Let $(f,g):t\to t'$ be a morphism in $\cat {Twist}$ from $t:C\to A$ to $t':C'\to A'$.  Define
$$\hoch(f,g):\hoch (t)\to \hoch(t')$$
by $\hoch (f,g)(c\otimes a)=f(c)\otimes g(a)$.  An easy calculation shows that $\hoch(f,g)$ is then a chain map.  Moreover, it is obvious that $\hoch (\Id_{C},\Id_{A})=\Id_{\hoch (t)}$ and that $\hoch$ respects composition.
\end{proof}

Thanks to this proposition, we obtain a new proof of Theorem 4.1 in \cite{hps-cohoch}.

\begin{cor} \label{cor:cohoch-hoch} A twisting cochain $t:C\to A$ induces a chain map
$$\beta_{t}\otimes \alpha_{t}: \cohoch (C)\to \hoch(A),$$
which is a quasi-isomorphism if $\alpha _{t}$ and $\beta_{t}$ are quasi-isomorphisms.
\end{cor}

\begin{proof} Applying Proposition \ref{prop:hoch-natl} to the morphism of twisting cochains
$$\xymatrix{C\ar [d]_{t_{\Om}}\ar [r]^{\beta_{t}}&\Bar A\ar [d]^{t_{\Bar}}\\
		    \Om C\ar [r]^{\alpha_{t}}&A,}$$
we obtain a morphism of chain complexes
$$\hoch(\beta_{t},\alpha_{t})=\beta_{t}\otimes \alpha_{t}:\hoch (t_{\Om })\to \hoch(t_{\Bar}).$$
Since (cf.~Examples \ref{ex:hoch} and \ref{ex:cohoch}) $\cohoch(C)=\hoch(t_{\Om})$ and $\hoch(A)=\hoch(t_{\Bar})$, we can conclude.

The quasi-isomorphism result follows from an easy argument based on the Zeeman comparison theorem. 
\end{proof}

\begin{rmk}  Applied to the universal twisting cochain $t_{\Om}:C\to \Om C$, Corollary \ref{cor:cohoch-hoch} implies the existence of a quasi-isomorphism
\begin{equation}\label{eqn:hoch-universal}
\eta_{C}\otimes \Id_{\Om C}:\cohoch(C) \xrightarrow\simeq \hoch (\Om C).
\end{equation}
Dually, for the couniversal twisting cochain $t_{\Bar }:\Bar A\to A$, we obtain a quasi-isomorphism
\begin{equation}\label{eqn:hoch-couniversal}
\Id_{\Bar A}\otimes \ve_{A}:\cohoch(\Bar A)  \xrightarrow\simeq \hoch (A).
\end{equation}
\end{rmk}

As stated more precisely in the proposition below, the Hochschild construction functor $\hoch$ also behaves well with respect to the cartesian product of twisting cochains.   The proof of this proposition is again a simple calculation, using the natural symmetry of the tensor product of chain complexes. 

\begin{prop}\label{prop:hoch-monoidal}  For all $t, t'\in \ob \cat {Twist}$,
$$\hoch(t*t')\cong \hoch(t)\otimes \hoch(t'),$$
and the isomorphism is natural in $t$ and $t'$.
\end{prop}

With just a bit more work, one can show that  the functor $\hoch$ is in fact strongly monoidal.

Proposition \ref{prop:hoch-monoidal} enables us to formulate conditions on a twisting cochain $t$ that easily imply the existence of extra algebraic structure on $\hoch (t)$. In section \ref{sec:hoch-comult} we explain how to weaken these conditions, to cover a much larger class of examples.

\begin{cor}\label{cor:hoch-comult} Let $H$ be a Hopf algebra, and let $C$ be a connected, coaugmented chain coalgebra.  Let $\delta:H\to H\otimes H$ and $\Delta: C\to C\otimes C$ denote the comultiplications on $H$ and $C$.  If $t:C\to H$ is a twisting cochain such that $(\Delta, \delta)$ is a morphism in $\cat {Twist}$ from $t$ to $t*t$, then $\hoch(t)$ admits a natural coassociative comultiplication.
\end{cor}
 
 Note that $(\Delta, \delta)$ is a morphism in $\cat {Twist}$ only if $\Delta$ is a morphism of coalgebras, which is true if and only if $C$ is cocommutative.

\begin{proof} The desired comultiplication is given by the composite
$$\hoch (t)\xrightarrow {\hoch (\Delta, \delta)} \hoch(t*t) \xrightarrow {\cong} \hoch (t)\otimes \hoch (t).$$
The coassociativity of $\Delta $ and of $\delta$ implies that of the composite above, since $\hoch$ is strongly monoidal.   
\end{proof}

\begin{ex}  Let $C$ be a connected, coaugmented chain coalgebra with cocommutative comultiplication $\Delta$.  Since $\Delta$ is cocommutative,  it induces a chain algebra map $\Om \Delta :\Om C\to \Om (C\otimes C)$ and therefore a coassociative comultiplication $\psi:\Om C\to \Om C\otimes \Om C$, defined to be the composite
$$\Om C\xrightarrow {\Om \Delta}\Om (C\otimes C)\xrightarrow q \Om C\otimes \Om C,$$
where $q$ is Milgram's equivalence (Example \ref{ex:milgram}).  (In fact, we have simply endowed $\Om C$ with the comultiplication such that all generators are primitive, but we emphasize this construction in terms of $q$, since we generalize it in section \ref{sec:hoch-comult}.) Moreover, the diagram
$$\xymatrix{C\ar[d]_{t_{\Om }}\ar [r]^(0.4)\Delta &C\otimes C\ar [d]^{t_{\Om }*t_{\Om}}\\
		         \Om C\ar[r]^(0.4)\psi&\Om C\otimes \Om C}$$
commutes, i.e., $(\Delta, \psi)$ is a morphism of twisting cochains.  Applying Corollary \ref{cor:hoch-comult}, we obtain a coassociative comultiplication on $\hoch(t_{\Om})$.
\end{ex}

\subsection{Extended naturality of the Hochschild construction}

We show in this section that the Hochschild construction on a twisting cochain is actually natural with respect to a weaker notion of morphism than that adopted in the definition of $\cat{Twist}$.

\begin{defn}\cite{gugenheim-munkholm}  
Given $C,C'\in \ob \cat {Coalg}_{R}$, a chain map $f:C\to C'$ is called a \emph{DCSH (Differential Coalgebra with Strong Homotopy) map} or a \emph{strongly homotopy-comultiplicative map} if there is a chain algebra map $\vp: \Om C\to \Om C'$  such that
$$\xymatrix{\overline C\ar [d]_{t_{\Om}}\ar [rr]^f&&\overline C'\ar [d]^{s^{-1}}\\ \Om C\ar [r]^\vp&\Om C'\ar [r]^{\text{proj}}&\si \overline C'}$$
commutes.  The chain algebra map $\vp$ is said to \emph{realize the strong homotopy structure} of the DCSH map $f$.
\end{defn}

\begin{ex}\label{ex:dcsh-simpl} If $K$ and $L$ are 1-reduced simplicial sets, then the natural Alexander-Whitney map
$$C_{*}(K\times L) \xrightarrow\simeq C_{*}K\otimes C_{*}L$$
is a DCSH map \cite{gugenheim-munkholm}.
\end{ex}

\begin{ex}\label{ex:dcsh-alg} As we prove in Appendix \ref{sec:bar-hopf} (Theorem \ref{thm:aw-bar-dcsh}), if $A$ and $A'$ are augmented chain algebras, then the natural Alexander-Whitney map 
$$\Bar(A\otimes A')\xrightarrow \simeq \Bar A \otimes \Bar A'$$
defined by Eilenberg and Mac Lane in \cite {eilenberg-maclane} is a DCSH map.
\end{ex}

\begin{ex}\label{ex:dcsh-unit} Let $C$ be a connected, coaugmented chain coalgebra.  Let $\rho_{C}:\Bar\Om C\to C$ denote the composite
$$\Bar\Om C \xrightarrow{t_{\Bar }}\Om C \xrightarrow{\text{proj}} \si C \xrightarrow {s} C,$$
i.e., $\rho_{C}\big(s(\si c)\big)=c$ for all $c\in \overline C$ and $\rho _{C}(sa_{1}|\cdots |sa_{n})=0$ otherwise.  It is obvious that $\rho_{C}$ is a retraction of $\eta_{C}$, i.e., 
$$\xymatrix{C\ar[r]^{\eta_{C}}\ar [dr]_{\Id_{C}}&\Bar \Om C\ar [d]^{\rho_{C}}\\ &C}$$
commutes, which implies that $\rho_{C}$ is a quasi-isomorphism.  Moreover, $\rho_{C}$ is a DCSH map, since
$$\xymatrix{\Bar_{+}\Om C\ar [d]_{t_{\Om}}\ar [rr]^{\rho_{C}}&&\overline C\ar [d]^{s^{-1}}\\ \Om \Bar \Om C\ar [r]^{\ve_{\Om C}}&\Om C\ar [r]^{\text{proj}}&\si \overline C}$$ 
commutes, where $\Bar_{+}\Om C$ denotes the $R$-submodule of positively graded elements.  In particular, $\ve_{\Om C}$ realizes the strong homotopy structure of $\rho_{C}$.
\end{ex}

\begin{rmk} \label{rmk:unfold-dcsh} A chain algebra map $\vp:\Om C\to \Om C'$ is determined by a set of $R$-linear maps
$$\{\vp_{k}:\overline C \to (\overline {C'})^{\otimes k}\mid k\geq 1\},$$
where
\begin {enumerate}
\item $\vp _{k}$ is homogeneous of degree $k-1$ for all $k$; and
\item \begin{align*}
\vp_{k}d_{C}+(-1)^kd_{(C')^{\otimes k}}\vp _{k}=&\sum _{i+j=k}(\vp _{i}\otimes \vp_{j})\overline \Delta_{C} \\
&-\sum _{i+j=k-2}(-1)^{i}(\Id_{C'}^{\otimes i}\otimes\Delta_{C'}\otimes \Id_{C'}^{\otimes j})\vp _{k-1} 
\end{align*}
for all $k$.
\end{enumerate} 
The relation between $\vp $ and the family $\{\vp_{k}\}_{k}$ is that 
$$\vp (\si c)=\sum _{k\geq 1}(\si)^{\otimes k}\vp _{k}(c),$$
for all $c\in \overline C$.  It follows that $\vp $ realizes the strong homotopy structure of a DCSH map $f$ if and only if $\vp _{1}=f$.
\end{rmk}

\begin{notn}  Let $\cat{Twist^{sh}}$ denote the category such that
\begin{itemize}
\item $\ob\cat{Twist^{sh}}=\{t:C\to A\mid t \text{ twisting cochain}\}$, and
\item  if $t:C\to A$ and $t':C'\to A'$ are twisting cochains, then
$$\cat{Twist^{sh}}(t,t')=\{ (\vp,g)\in \cat{Alg}_{R}(\Om C,\Om C')\times \cat {Alg}_{R}(A,A')\mid g\circ \alpha_{t}=\alpha _{t'}\circ \vp\}.$$
\end{itemize}
Composition of morphisms in $\cat {Twist^{sh}}$ is defined componentwise.
\end{notn}

\begin{rmk} There is an obvious faithful functor $\cat{Twist}$ to $\cat{Twist^{sh}}$, which is the identity on objects and which sends a morphism $(f,g):t\to t'$ to the morphism $(\Om f, g):t\to t'$.
\end{rmk}

\begin{thm}\label{thm:hoch-extnatl}The Hochschild construction functor extends to a functor 
$$\hoch^{sh}: \cat{Twist^{sh}}\to \cat{Ch}_{R}.$$
In particular, given twisting cochains $t:C\to A$ and $t':C'\to A'$ and a commutative diagram in $\cat{Alg}_{R}$
$$\xymatrix{\Om C\ar[d]_{\vp}\ar [r]^{\alpha _{t}}&A\ar[d]^{g}\\ \Om C'\ar [r]^{\alpha_{t'}}&A',}$$
there is a chain map $\hoch^{sh}(\vp, g):\hoch (t)\to \hoch (t')$ such that 
$$\xymatrix{\Om C\ar[d]_{\vp}\ar [r]^{\alpha _{t}}&A\ar[d]_{g}\cof&\hoch (t)\ar [d]_{\hoch^{sh}(\vp, g)}\fib & C\ar [d]_{f}\\
 \Om C'\ar [r]^{\alpha_{t'}}&A'\cof &\hoch (t')\fib & C'}$$
 commutes, where $f=\vp_{1}$.  Furthermore, if $\vp$ and $g$ are quasi-isomorphisms, then so is $\hoch^{sh}(\vp, g)$.
\end{thm}

\begin{proof} Let $(\vp,g):t\to t'$ be a morphism in $\cat{Twist^{sh}}$ from $t:C\to A$ to $t':C'\to A'$.  Let $\{\vp_{k}\}_{k}$ be the family of $R$-linear maps associated to $\vp$, as in Remark \ref{rmk:unfold-dcsh}. Suppressing summation to simplify notation, for all $c\in C$ and $k\geq 1$, we write
$$\vp_{k}(c)=c'_{k,1}\otimes \cdots\otimes c'_{k,k}.$$
Define $\hoch^{sh}(\vp, g): \hoch(t)\to \hoch(t')$ by
\begin{align}\label{eqn:hsh}
\hoch^{sh} (\vp, g)&(c\otimes a)\\
&= \sum_{ k\geq 1\atop 1\leq i\leq k}\pm c'_{k,i}\otimes t'(c'_{k,i+1})\cdot\, \ldots\, \cdot t'(c'_{k,k})\cdot g(a)\cdot t'(c'_{k,1})\cdot\, \ldots\, \cdot t'(c'_{k,i-1}),
\end{align}
where the signs are determined by the Koszul rule.  

Note that if $\vp=\Om f$, where $f:C\to C'$ is a morphism of coalgebras, then $\vp_{k}=0$ for all $k>1$, so that $\hoch^{sh} (\vp, g)=\hoch(f,g)$, as defined in the proof of Proposition \ref{prop:hoch-natl}.  The functor $\hoch^{sh}: \cat{Twist^{sh}}\to \cat{Ch}_{R}$ that we are defining here is therefore indeed an  extension of the functor $\hoch: \cat {Twist}\to \cat{Ch}_{R}$ of Proposition \ref{prop:hoch-natl}. 

Let $d$ denote the differentials on $A$ and $C$ and $d'$ the differentials on $A'$ and $C'$.  Observe that

\begin{itemize}
\item equation (2) in  Remark  \ref{rmk:unfold-dcsh} implies that for all $k\geq 1$, 
\begin{align*}
&(dc)'_{k,1}\otimes \cdots\otimes (dc)'_{k,k}+\sum _{1\leq i\leq k}\pm c'_{k,1}\otimes \cdots\otimes dc'_{k,i}\otimes \cdots\otimes c'_{k,k}\\
&=\sum _{k'+k''=k}\pm (c_{j})'_{k',1}\otimes \cdots\otimes (c_{j})'_{k',k'}\otimes (c^j)'_{k'',1}\otimes \cdots\otimes (c^j)'_{k'',k''}\\
&\quad+\sum _{1\leq i<k}c'_{k-1,1}\otimes \cdots\otimes c'_{k-1,i,j}\otimes {c'_{k-1,i}}^j\otimes\cdots\otimes c'_{k-1,k-1},
\end{align*}
where $\overline\Delta (c)=c_{j}\otimes c^j$ and $\overline\Delta (c'_{k-1,i})=c'_{k-1,i,j}\otimes {c'_{k-1,i}}^j$ and signs are determined by the Koszul rule;
\item since $\alpha_{t'}\vp =g\alpha_{t}$, for all $c\in C$,
$$g\big(t(c)\big)=\sum _{k\geq 1} t'(c'_{k,1})\cdot \,\ldots\,\cdot t'(c'_{k,k});$$
and
\item since $t'$ is a twisting cochain, for all $c'\in C'$,
$$d't'(c')+t'(d'c')=t'(c'_{j})\cdot t'(c'^j),$$
where $\overline \Delta(c')=c'_{j}\otimes {c'}^j$.
\end{itemize}

Given these observations, it is not difficult to prove that $\hoch^{sh}(\vp, g)$ is a chain map, i.e., that $\hoch^{sh}(\vp, g)d_{t}=d_{t'}\hoch^{sh} (\vp, g)$, by careful expansion of both sides of the equation.

An easy spectral sequence argument shows that $\hoch^{sh}(\vp, g)$ is a quasi-isomorphism if $\vp$ and $g$ are quasi-isomorphisms.
\end{proof}

As a consequence of Theorem \ref{thm:hoch-extnatl}, we obtain a new proof of Theorem 3.3 in \cite{hps-cohoch}.

\begin{cor}\label{cor:cohoch-extnatl}  A DCSH map $f:C\to C'$ with a fixed choice of chain algebra map $\vp:\Om C\to \Om C'$ realizing its strong homotopy structure naturally induces a chain map
$$\widehat \vp:\cohoch(C) \to \cohoch(C')$$
such that 
$$\xymatrix{\Om C\cof \ar[d]^\vp&\cohoch (C) \fib\ar[d]^{\widehat \vp}&C\ar [d]^f\\
\Om C'\cof &\cohoch (C')\fib &C'}$$
commutes. 
\end{cor}

\begin{proof}  Recall that $\cohoch(C)=\hoch(t_{\Om})$ (Example \ref{ex:cohoch}). Let $\widehat\vp=\hoch^{sh}(\vp, \vp)$.
\end{proof}

Theorem \ref{thm:hoch-extnatl} also provides us with a new, more conceptual proof of Theorem B from \cite{jones-mccleary}, which was reformulated as Theorem 4.3 in \cite {hps-cohoch} essentially as follows. Recall the DCSH retraction $\rho_{C}:\Bar \Om C\to C$ from Example \ref{ex:dcsh-unit} and the chain map $\eta_{C}\otimes \Id_{\Om C}:\cohoch(C) \xrightarrow\simeq \hoch (\Om C)$ (\ref{eqn:hoch-universal}).

\begin{cor} \label{cor:dcsh-unit} Let $C$ be a connected, coaugmented chain coalgebra. There is a quasi-isomorphism $\widehat \rho_{C}:\hoch(\Om C)\xrightarrow\simeq \cohoch (C)$, natural in $C$, such that 
$$\xymatrix{\Om C\ar[d]_{\Id_{\Om C}}\cof&\hoch (\Om C)\ar [d]_{\widehat\rho_{C}}\fib & \Bar \Om C\ar [d]_{\rho_{C}}\\
\Om C\cof &\cohoch (C)\fib & C}$$
 commutes and such that $\widehat\rho_{C}\circ( \eta_{C}\otimes \Id_{\Om C})=\Id_{\cohoch(C)}$.
\end{cor}

\begin{proof}  Let $\widehat \rho_{C}=\hoch^{sh}(\ve _{\Om C}, \Id_{\Om C})$. The naturality of $\hoch^{sh}$ then implies that $$\widehat\rho_{C}\circ (\eta_{\Om C}\otimes \Id_{\Om C})=\Id_{\cohoch(C)}.$$
\end{proof}

The definitions and results above can be dualized as follows.

\begin{defn}\cite{gugenheim-munkholm}  
Given $A,A'\in \ob \cat {Alg}_{R}$, a chain map $g:A\to A'$ is called a \emph{DASH (Differential Algebra with Strong Homotopy) map} or a \emph{strongly homotopy-multiplicative map} if there is a chain coalgebra map $\gamma: \Bar A\to \Bar A'$ such that 
$$\xymatrix{s\overline A\ar [d]_{\si}\ar[r]^{\text{incl}}& \Bar A \ar [r]^\gamma&\Bar A'\ar [d]^{t_{\Bar}}\\ \overline A\ar [rr]^g &&\overline A'}$$
commutes.  The chain coalgebra map $\gamma$ is said to \emph{realize the strong homotopy structure} of the DASH map $g$.
\end{defn}

\begin{ex} Dualizing Example \ref{ex:dcsh-simpl}, we see that if $K$ and $L$ are 1-reduced simplicial sets of finite type, then the dual of the Alexander-Whitney map (i.e., the cross product)
$$C^*K\otimes C^*L\xrightarrow \simeq C^*(K\times L)$$
is a DASH map.
\end{ex}

\begin{ex} Dualizing Example \ref{ex:dcsh-alg}, we obtain that if $C$ and $C'$ are connected, coaugmented chain coalgebras of finite type, then the dual
$$\Om C\otimes \Om C' \xrightarrow\simeq \Om (C\otimes C')$$
of the Alexander-Whitney map for the bar construction is a DASH map.
\end{ex}

\begin{ex}\label{ex:dash-counit} Let $A$ be an augmented chain algebra.  Let $\sigma_{A}:A\to \Om \Bar A$ denote the chain map defined by $\sigma _{A}(a)=\si (sa)$.  It is obvious that $\sigma_{A}$ is a section of $\ve_{A}$, i.e., 
$$\xymatrix{A\ar[r]^{\sigma_{A}}\ar [dr]_{\Id_{A}}&\Om\Bar A\ar [d]^{\ve_{A}}\\ &A}$$
commutes, which implies that $\sigma_{A}$ is a quasi-isomorphism.  Moreover, $\sigma_{A}$ is a DASH map, since
$$\xymatrix{s\overline A\ar [d]_{\si}\ar[r]^{\text{incl}}& \Bar A \ar [r]^{\eta_{\Bar A}}&\Bar \Om \Bar A\ar [d]^{t_{\Bar}}\\ \overline A\ar [rr]^{\sigma _{A}} &&\overline {\Om\Bar A}}$$
commutes, where $\overline {\Om\Bar A}$ denotes the augmentation ideal of $\Om\Bar A$.  In particular, $\eta_{\Bar A}$ realizes the strong homotopy structure of $\sigma_{A}$.
\end{ex}

\begin{notn}  Let $\cat{Twist_{sh}}$ denote the category such that
\begin{itemize}
\item $\ob\cat{Twist_{sh}}=\{t:C\to A\mid t \text{ twisting cochain}\}$, and
\item  if $t:C\to A$ and $t':C'\to A'$ are twisting cochains, then
$$\cat{Twist_{sh}}(t,t')=\{ (f,\gamma)\in \cat{Coalg}_{R}(C,C')\times \cat {Coalg}_{R}(\Bar A,\Bar A')\mid \gamma\circ \beta_{t}=\beta _{t'}\circ f\}.$$
\end{itemize}
Composition of morphisms in $\cat {Twist_{sh}}$ is defined componentwise.
\end{notn}

\begin{rmk} There is an obvious faithful functor $\cat{Twist}$ to $\cat{Twist_{sh}}$, which is the identity on objects and which sends a morphism $(f,g):t\to t'$ to the morphism $(f, \Bar g):t\to t'$.
\end{rmk}

To avoid truly nasty explicit formulas, we permit ourselves a slight restriction in dualizing Theorem \ref{thm:hoch-extnatl}.  Let $\cat{Twist_{sh, f}}$ denote the full subcategory of twisting cochains $t:C\to A$ such that both $C$ and $A$ are connected and of finite type, i.e., are finitely generated free $R$-modules in each degree.

\begin{thm}\label{thm:hoch-extnatl-dual}The Hochschild construction functor extends to a functor 
$$\hoch_{sh}: \cat{Twist_{sh,f}}\to \cat{Ch}_{R}.$$
In particular, given twisting cochains $t:C\to A$ and $t':C'\to A'$, where $C$, $A$, $C'$ and $A'$ are connected and of finite type, and a commutative diagram in $\cat{Coalg}_{R}$
$$\xymatrix{C\ar[d]_{f}\ar [r]^{\beta _{t}}&\Bar A\ar[d]^{\gamma}\\ C'\ar [r]^{\beta_{t'}}&\Bar A',}$$
there is a chain map $\hoch_{sh}(f, \gamma):\hoch (t)\to \hoch (t')$ such that 
$$\xymatrix{A\ar[d]_{g}\cof&\hoch (t)\ar [d]_{\hoch_{sh}(f, \gamma)}\fib & C\ar [d]_{f}\ar [r]^{\beta _{t}}&\Bar A\ar[d]_{\gamma}\\
A'\cof &\hoch (t')\fib & C'\ar [r]^{\beta_{t'}}&\Bar A'}$$
 commutes, where $g(a)=t_{\Bar}\gamma (sa)$ for all $a\in A$.
\end{thm}

\begin{proof}  Recall that the $R$-dual of any $R$-coalgebra is an $R$-algebra, while the $R$-dual of any finite-type $R$-algebra is an $R$-coalgebra.  It follows that in order to prove this theorem, we can dualize, then apply Theorem \ref{thm:hoch-extnatl} and finally dualize again to obtain the desired map.
\end{proof}

We can also dualize Corollary \ref{cor:cohoch-extnatl}, obtaining a result not explictly stated in \cite{hps-cohoch}.

\begin{cor}A DASH map $g:A\to A'$ between finite-type, connected chain algebras, with a fixed choice of chain coalgebra map $\gamma:\Bar A\to \Bar A'$ realizing its strong homotopy structure, naturally induces a chain map
$$\widehat \gamma:\hoch(A) \to \hoch(A')$$
such that 
$$\xymatrix{A\ar[d]_{g}\cof&\hoch (A)\ar [d]_{\widehat \gamma}\fib & \Bar A\ar[d]_{\gamma}\\
A'\cof &\hoch (A')\fib &\Bar A'}$$
commutes.
\end{cor}

\begin{proof}  Recall that $\hoch(A)=\hoch(t_{\Bar})$ (Example \ref{ex:hoch}). Let $\widehat\gamma=\hoch_{sh}(\gamma, \gamma)$.
\end{proof}

There is also a result dual to Corollary \ref{cor:dcsh-unit} that holds. Recall the DASH map $\sigma_{A}:A\to \Om \Bar A$ from Example \ref{ex:dash-counit} and the chain map $\Id_{\Bar A}\otimes \ve_{A}:\cohoch(\Bar A)  \xrightarrow\simeq \hoch (A)$ (\ref{eqn:hoch-couniversal}).

\begin{cor}\label{cor:dash-counit}Let $A$ be a connected, augmented chain algebra of finite type. There is a quasi-isomorphism $\widehat \sigma_{A}:\hoch(A)\xrightarrow\simeq \cohoch (\Bar A)$, natural in $A$, such that 
$$\xymatrix{A\ar[d]_{\sigma _{A}}\cof&\hoch (A)\ar [d]_{\widehat\sigma_{A}}\fib & \Bar A\ar [d]_{\Id_{\Bar A}}\\
\Om \Bar A\cof &\cohoch (\Bar A)\fib & \Bar A}$$
commutes and such that $(\Id_{\Bar A}\otimes \ve_{A})\circ \widehat\sigma_{A}=\Id_{\hoch (A)}$.
\end{cor}

\begin{proof}  Let $\widehat \sigma_{A}=\hoch_{{sh}}(\Id_{\Bar A},\eta_{\Bar A} )$.  The naturality of $\hoch_{sh}$ implies then that 
$$(\Id_{\Bar A}\otimes \ve_{A})\circ \widehat\sigma_{A}=\Id_{\hoch (A)}.$$
\end{proof}



\section{Operations and cooperations on the Hochschild complex}\label{sec:hoch-comult}

As mentioned in the Introduction, it is well known that the Hochschild complex of a commutative algebra admits a commutative multiplication.  Moreover, as we pointed out in Corollary \ref{cor:hoch-comult}, if $t:C\to H$ is a twisting cochain from a cocommutative coalgebra to a Hopf algebra that is strictly compatible with the comultiplications, then $\hoch(t)$ admits a natural coassociative comultiplication.  Also, in \cite{hps-cohoch}, the authors gave conditions on a coalgebra $C$ under which $\cohoch(C)$ admits a comultiplication.

Motivated by these special cases, we provide here conditions on a twisting cochain $t:C\to A$, in the spirit of those in \cite{hps-cohoch}, under which $\hoch(t)$ admits a (co)multiplicative structure. 

\subsection{Alexander-Whitney (co)algebras}

In this section we define Alexander-Whitney coalgebras and algebras, which are the type of highly structured coalgebras and algebras for which the Hochshild complex admits a natural comultiplication or a natural multiplication.  We recall well-known examples of Alexander-Whitney coalgebras and algebras coming from topology and introduce new classes of examples, including the bar construction on any chain Hopf algebra.

\begin{defn}\cite{hpst}  A \emph{weak Alexander-Whitney coalgebra} consists of a connected, coaugmented chain coalgebra $C$ such that the comultiplication $\Delta:C\to C\otimes C$ is a DCSH map, together with a choice of chain algebra map $\omega:\Om C\to \Om (C\otimes C)$ that realizes the DCSH structure of $\Delta$.  If the composite
$$\Om C \xrightarrow \omega \Om (C\otimes C) \xrightarrow q \Om C \otimes \Om C$$
is a coassociative comultiplication on $\Om C$, where $q$ denotes the Milgram equivalence (Example \ref{ex:milgram}), then $(C,\omega )$ is an \emph{Alexander-Whitney coalgebra}.   We call $q\omega$ the \emph{associated loop comultiplication}.  Note that $(\Om C, q\omega)$ is a chain Hopf algebra.

An Alexander-Whitney coalgebra $(C,\omega)$ is \emph{balanced} if the associated loop comultiplication is cocommutative.

\emph{Alexander-Whitney algebras}, which we usually denote $(A,\nu)$---and their weak or balanced variants---are defined dually.  If $(A,\nu)$ is an Alexander-Whitney algebra, then the composite
$$\Bar A \otimes \Bar A \xrightarrow \nabla \Bar (A\otimes A) \xrightarrow \nu \Bar A$$
is an associative multiplication on $\Bar A$, where $\nabla $ denotes the Eilenberg-Zilber map (Example \ref{ex:milgram}).  We call $\nu \nabla$ the \emph{associated bar multiplication}.  Note that $(\Bar A, \nu \nabla)$ is a chain Hopf algebra.
 \end{defn}
 
 \begin{rmk} An Alexander-Whitney (co)algebra is a special type of $B_{\infty}$-(co)algebra \cite {baues2}, \cite{getzler-jones}.
 \end{rmk}

If $\Delta:C\to C\otimes C$ is a DCSH map and $\omega:\Om C\to \Om (C\otimes C)$ realizes its DCSH structure, then $\Id_{C}\otimes \Delta$ and $\Delta\otimes \Id_{C}$ are both DCSH maps as well.  In particular, there are chain algebra maps 
$$\Id_{C}\wedge\omega, \omega \wedge \Id_{C}:\Om (C\otimes C)\to \Om (C\otimes C\otimes C)$$ realizing their DCSH structure, where the $k^{\text{th}}$ members of the associated families of $R$-linear maps, $(\Id_{C}\wedge \omega)_{k}$ and $( \omega \wedge \Id_{C})_{k}$, are given by following composites:
$$C\otimes C\xrightarrow{ \Delta ^{(k)}\otimes \omega_{k}} C^{\otimes k}\otimes (C\otimes C)^{\otimes k} \xrightarrow \cong (C\otimes C\otimes C)^{\otimes k}$$
and 

$$C\otimes C\xrightarrow{\omega_{k}\otimes \Delta ^{(k)}} (C\otimes C)^{\otimes k}\otimes C^{\otimes k} \xrightarrow \cong (C\otimes C\otimes C)^{\otimes k},$$
where the second map in each composite is the obvious permutation, and $\omega _{k}$ is the $k^{\text{th}}$ member of the family of $R$-linear maps associated to $\omega$.  For further justification of this construction, we refer the reader to section 1.1 in  \cite{hess2}.

\begin{defn}
A \emph{strict Alexander-Whitney coalgebra}  is a weak Alexander-Whitney coalgebra $(C,\omega)$ such that 
$$(\Id_{C}\wedge\omega) \omega =(\omega\wedge \Id_{C})\omega.$$
A \emph{quasistrict Alexander-Whitney coalgebra}  is a weak Alexander-Whitney coalgebra $(C,\omega)$ such that there is a derivation homotopy from  $(\Id_{C}\wedge\omega) \omega$ to $(\omega\wedge \Id_{C})\omega$.

\emph{Strict Alexander-Whitney algebras} and  \emph{quasistrict Alexander-Whitney algebras} are defined dually.
\end{defn}

\begin{rmk}  The naturality of the Milgram equivalence $q$ implies that any strict Alexander-Whitney coalgebra is an Alexander-Whitney coalgebra.  A similar statement holds for strict Alexander-Whitney algebras, due to the naturality of the Eilenberg-Zilber equivalence $\nabla$.
\end{rmk}

\begin{ex} If $C$ is a connected, coaugmented, cocommutative coalgebra with comultiplication $\Delta$, then $(C,\Om \Delta)$ is a strict, balanced Alexander-Whitney coalgebra. Dually, if $A$ is an augmented, commutative algebra with multiplication $m$, then $(A, \Bar m)$ is a strict, balanced Alexander-Whitney algebra.
\end{ex}

\begin{ex}\label{ex:hpst} \cite{hpst} For any reduced simplicial set $K$, there is a natural choice of chain algebra map $\omega _{K}:\Om C_{*}K\to \Om (C_{*}K\otimes C_{*}K)$ such that $(C_{*}K, \omega _{K})$ is an Alexander-Whitney coalgebra. 

In general $C_{*}K$ is not  a strict Alexander-Whitney coalgebra. On the other hand, as explained in Example 3.8 of \cite{hps-cohoch}, $C_{*}K$ is always a quasistrict Alexander-Whitney coalgebra.  

Along the same lines, Theorem 4.9 in \cite {hps2}  implies that if $L$ is a pointed simplicial set such that $C_{*}L$ is a cocommutative coalgebra, then there is a natural choice of chain algebra map $\omega _{\E L}:\Om C_{*}\E L\to \Om (C_{*}\E L\otimes C_{*}\E L)$ such that $(C_{*}\E L, \omega_{\E L})$ is a balanced Alexander-Whitney coalgebra, where $\E$ denotes the simplicial suspension functor \cite{may}.  More precisely, $\Om C_{*}\E L\cong T( \overline C_{*}L)$, the tensor algebra generated by coaugmentation coideal of $C_{*}L$, endowed with the strictly linear differential induced by the differential on $C_{*}L$.  Moreover, the comultiplication $\psi_{\E L}=q\omega_{\E L}$ satisfies 
$$\psi _{\E L}(x)=x_{i}\otimes x^{i}\in T(\overline C_{*}L)\otimes T(\overline C_{*}L),$$
for all $x\in \overline C_{*}L$, where $\Delta (x)=x_{i}\otimes x^{i}$, and $\Delta$ is the usual comultiplication on $C_{*}L$.
In particular, if $\S$ is the simplicial double suspension functor, then  $(C_{*}\S M, \omega_{\S K})$ is a  balanced Alexander-Whitney coalgebra for all simplicial sets $M$, since all positive-degree elements of $C_{*}\Eu M$ are primitive (cf. Appendix \ref{app:susp}).
\end{ex}

Inspired by Example \ref{ex:hpst}, we formulate the following definition.

\begin{defn}\label{defn:symmetric}   A reduced simplicial set $K$ is \emph{symmetric} if $C_{*}K\otimes \mathbb F_{2}$ is a strictly cocommutative coalgebra, where $\mathbb F_{2}$ denotes the field of 2 elements.
\end{defn}

In these terms, the result from \cite {hps2} cited above in Example \ref{ex:hpst}  says that if $K$ is a symmetric simplicial set, then $C_{*}\E K\otimes \mathbb F_{2}$ is a balanced Alexander-Whitney coalgebra.

\begin{rmk}\label{rmk:rpinfty}  The class of symmetric simplicial sets includes all simplicial suspensions, both reduced and unreduced, since the natural comultiplication on the normalized chain complex of a simplicial suspension is trivial in positive degrees.  Moreover an easy calculation shows that the nerve of $\mathbb Z/2\mathbb Z$, which is a simplicial model of $\mathbb RP^\infty$, is also symmetric (cf.~Example \ref{ex:rpinfty}).  

From these examples of symmetric simplicial sets, one can construct many others, including all wedges of truncated real projective spaces of arbitrary dimension.  It is clear that all subsimplicial sets and all quotients of symmetric simplicial sets are also symmetric, as is any wedge sum of symmetric simplicial sets.  We intend to study symmetric simplicial sets in greater detail in future work. 
\end{rmk} 

The next theorem provides us with another important class of Alexander-Whitney (co)algebras. 

\begin{thm}\label{thm:bar-hopf} If $H$ is a chain Hopf algebra, then $\Bar H$ is naturally an Alexander-Whitney coalgebra.  If $H$ is connected and of finite type, then $\Om H$ is naturally an Alexander-Whitney algebra.  Moreover, the associated loop comultiplication on $\Om\Bar H$ and associated bar multiplication on $\Bar\Om H$ are such that the natural quasi-isomorphisms $\ve_{H}:\Om\Bar H\to H$  and $\eta_{H}:H\to \Bar\Om H$ are both maps of Hopf algebras.
\end{thm}

We refer the reader to Appendix \ref{sec:bar-hopf} for the proof of this theorem.

When we define power maps on Hochschild complexes in section \ref{sec:powermaps}, we can actually relax slightly the conditions on the (co)algebras we consider and study the Hirsch (co)algebras of Kadeishvili \cite{kadeishvili}.

\begin{defn}  A \emph{Hirsch coalgebra} consists of a connected, coaugmented chain coalgebra $C$, together with a coassociative comultiplication $\psi:\Om C\to \Om C\otimes \Om C$, called its \emph{loop comultiplication}, that is a morphism of chain algebras.  A Hirsch coalgebra $(C,\psi)$ is \emph{balanced} if $\psi $ is cocommutative.

\emph{Hirsch algebras} are defined dually.
\end{defn}

\begin{notn}\label{notn:hirsch} Let $\cat {Hirsch}_{R}$ denote the category of which the objects are Hirsch coalgebras $(C, \psi)$ and where
$$\cat {Hirsch}_{R}\big( (C, \psi), (C', \psi')\big)=\{ f\in \cat {Coalg}_{R}(C,C')\mid (\Om f\otimes \Om f) \psi=\psi' \Om f\}.$$
\end{notn}

\begin{rmk}\label{rmk:aw-to-gerst} If $(C,\omega)$ is an Alexander-Whitney coalgebra, then   $(C, q\omega)$ is a Hirsch coalgebra.
\end{rmk}

\begin{rmk}\label{rmk:gerstenhaber} As explained in \cite{kadeishvili}, the homology of a Hirsch algebra is naturally a \emph {Gerstenhaber algebra}, i.e.,  a commutative graded algebra endowed with a Lie bracket of degree $-1$ that is a biderivation with respect to the multiplication.  In particular, the homology of any Alexander-Whitney algebra is a Gerstenhaber algebra.
\end{rmk}

\begin{ex}\cite{kadeishvili} The Hochschild cochain complex of a chain algebra $A$, usually denoted $C^*(A,A)$, is a Hirsch algebra.
\end{ex}

\begin{rmk} In \cite{baues}  Baues provided combinatorial formulas for a natural Hirsch coalgebra structure on $C_{*}K$, for all reduced simplicial sets $K$.  The Alexander-Whitney structure defined in \cite{hpst} lifts Baues' Hirsch structure.
\end{rmk}
 
\begin{ex} \label{ex:nonrealAW} Not all Hirsch coalgebras are induced from Alexander-Whitney coalgebras, as in Remark \ref{rmk:aw-to-gerst}. Let $C$ denote the free graded abelian group with five generators:  $u$ in degree 0 (which plays the role of 1), $x$ in degree 3, $y$ and $y'$ in degree 4 and $z$ in degree 7.  Endow $C$ with a differential $d$ specified by

	$$dy=2x\quad\text{ and }\quad dy'=3x,$$
while $u$, $x$, and $z$ are cycles for degree reasons.   Define a comultiplication $\Delta$ on $C$ by setting $x$, $y$ and $y'$ to be primitive, while

	$$\Delta(z)=u\otimes  z + 3 x\otimes y -2 x\otimes y' + z\otimes u .$$

It is easy to check that $\Delta$ is a chain map.  Moreover, $\Delta$ is cocommutative up to chain homotopy, where the chain homotopy $F$ is given by $F(x)=F(y)=F(y')=0$ and

	$$F(z)= y'\otimes y - y\otimes y'.$$

There is a cocommutative, coassociative comultiplication $\psi : \Om C\to \Om C\otimes \Om C$, defined to be primitive on all generators, except $\si z$, where
$$\psi (\si z)= \si z\otimes 1 +\si y' \otimes \si y -\si y\otimes \si y' + 1\otimes \si z.$$
Thus, $(C,\psi)$ is a balanced Hirsch coalgebra.  Easy computations show that there is no algebra map $\omega: \Om C\to \Om (C\otimes C)$ such that $q\omega =\psi$.  It follows that $(C,\psi)$ is not realizable as the Hirsch coalgebra of a simplicial set.
\end{ex}

\begin{rmk}\label{rmk:rectify}  Any Hirsch coalgebra $(C,\psi)$ is weakly equivalent to an Alexander-Whitney coalgebra, since $\Bar (\Om C,\psi)$ is an Alexander-Whitney coalgebra,  by Theorem  \ref{thm:bar-hopf}, and $\eta_{C}:C\to \Bar\Om C$ is a quasi-isomorphism of chain coalgebras.
\end{rmk}

When we construct power maps on the Hochschild complex of a twisting cochain later in this paper, we are led to consider variants of the category $\cat {Twist}$ (cf.~Notation \ref{notn:twist}) that involve Hirsch coalgebra structure.

\begin{notn}\label{notn:twist-hirsch}  Let $\cat{Twist}_{\text{Hirsch}}$ denote the category with
\begin{itemize}
\item $\ob\cat{Twist}_{\text{Hirsch}} =\big\{\big((C,\psi), C\xrightarrow t A\big)\mid (C,\psi) \in \ob \cat {Hirsch}_{R}, t\in  \ob \cat {Twist} \big\}$, and
\item if $\big((C,\psi), t\big)$ and $\big((C',\psi'), t'\big)$ are objects in $\cat{Twist}_{\text{Hirsch}}$, then
\begin{multline*}\cat{Twist}_{\text{Hirsch}}\Big(\big((C,\psi), t\big),\big((C',\psi'), t'\big)\Big)\\=\big\{(f,g)\in \cat {Twist}(t,t')\mid (\Om f\otimes \Om f)\psi=\psi'\Om f\big \}.
\end{multline*}
\end{itemize}
\end{notn}

We sometimes need an even more highly structured category. Recall Notation \ref{notn:twist-hopf}.

\begin{notn}\label{notn:twist-hiho}  Let $\cat{Twist}_{\text{HH}}$ denote the category  the objects of which are triples$$ \big((C, \psi), C\xrightarrow t H, (H,\delta)\big),$$ where $(C, \psi)\in \ob \cat {Hirsch}_{R}$, $t\in \ob\cat{Twist}$, and $(H,\delta)\in\ob  \cat {Hopf}_{R}$.  If $ \big((C, \psi), t, (H,\delta)\big)$ and $ \big((C', \psi'), t', (H',\delta')\big)$ are objects in $\cat{Twist}_{\text{HH}}$, then
\begin{multline*}\cat{Twist}_{\text{HH}}\Big(\big((C, \psi), t, (H,\delta)\big),  \big((C', \psi'), t', (H',\delta')\big)\Big) \\
=\cat{Twist}_{\text{Hirsch}}\Big(\big((C, \psi), t\big),  \big((C', \psi'), t'\big)\Big) \cap \cat{Twist}_{\text{Hopf}}\Big(\big(t, (H,\delta)\big),  \big( t', (H',\delta')\big)\Big) 
\end{multline*}

\end{notn}

\subsection {Existence of (co)multiplication on the Hochschild complex}

We are now ready to generalize Corollary \ref{cor:hoch-comult}, as well as Theorem 3.9 in \cite{hps-cohoch}, which says that the coHochschild complex of an Alexander-Whitney coalgebra admits a natural comultiplication.

\begin{thm}\label{thm:hoch-comult} Let $(C,\omega)$ be an Alexander-Whitney coalgebra with underlying comultiplication $\Delta:C\to C\otimes C$.  Let $H$ be a chain Hopf algebra, with comultiplication $\delta:H\to H\otimes H$. Let $t:C\to H$ be a twisting cochain. 

If $(\omega, \delta)$ is a morphism in $\cat{Twist^{sh}}$ from $t$ to $t*t$, then the Hochschild complex of $t$ admits a comultiplication $\widehat\delta:\hoch(t)\to \hoch(t)\otimes \hoch(t)$ such that 
 $$\xymatrix{H\ar [d]_{\delta}\cof &\hoch(t)\ar [d]_{\widehat \delta}\fib & C\ar [d]_{\Delta}\\
 		    H\otimes H\cof &\hoch (t)\otimes \hoch (t) \fib&C\otimes C}$$
commutes.  Moreover, $\widehat \delta$ is coassociative (respectively, coassociative up to chain homotopy) if $(C,\omega)$ is a strict (respectively, quasistrict) Alexander-Whitney coalgebra.
\end{thm}

Note that, according to Remark \ref{rmk:cartesian-twist},  asking for $(\omega, \delta)$ to be a morphism in $\cat{Twist^{sh}}$ from $t$ to $t*t$ is equivalent to requiring the diagram
$$\xymatrix{\Om C\ar [dd]_{\alpha _{t}}\ar [r]^(0.4)\omega&\Om (C\otimes C)\ar [d]^q\\
&\Om C\otimes \Om C\ar [d]^{\alpha_{t}\otimes \alpha_{t}}\\
H\ar [r]^(0.4)\delta&H\otimes H}$$
to commute, i.e., to requiring $\alpha_{t}$ to be a map of chain coalgebras with respect to the associated loop comultiplication
$$\Om C\xrightarrow \omega \Om (C\otimes C) \xrightarrow q \Om C\otimes \Om C.$$

\begin{proof}  Apply Theorem \ref{thm:hoch-extnatl} to $(\omega, \delta):t \to t*t$, obtaining a chain map
$$\hoch^{sh}(\omega, \delta): \hoch (t)\to \hoch (t*t).$$
The desired comultiplication $\widehat\delta$ is then given by the composite
$$\hoch(t)\xrightarrow{\hoch^{sh}(\omega, \delta)} \hoch(t*t)\xrightarrow \cong \hoch(t)\otimes \hoch (t).$$ 
It follows easily from the formulas in the proof of Theorem \ref{thm:hoch-extnatl} that $\widehat \delta$ is coassociative (respectively, coassociative up to chain homotopy) if $(\Id_{C}\wedge\omega) \omega =(\omega\wedge \Id_{C})\omega$ (respectively, if there is a derivation homotopy from  $(\Id_{C}\wedge\omega) \omega$ to $(\omega\wedge \Id_{C})\omega$). 
\end{proof}

\begin{ex} Let $C$ be a quasistrict Alexander-Whitney coalgebra. Applying Theorem \ref{thm:hoch-comult} to the universal twisting cochain $t_{\Om}:C\to \Om C$, we obtain a homotopy coassociative comultiplication on $\cohoch(C)=\hoch (t_{\Om})$.  Theorem 3.9 in \cite{hps-cohoch} is therefore a special case of our Theorem \ref{thm:hoch-comult}.
\end{ex}

\begin{ex} Since the bar construction on a chain Hopf algebra is an Alexander-Whitney coalgebra, the Hochschild complex of $H$, which is equal to $\hoch(t_{\Bar})$, admits a comultiplication.  We conjecture that if $H$ is cocommutative, then $\Bar H$ is quasistrict and therefore the comultiplication on $\hoch(t_{\Bar})$ is coassociative up to chain homotopy.
\end{ex}

Dualizing both the statement and the proof of the theorem above, we obtain the following result.

\begin{thm}\label{thm:hoch-mult}  Let $(A,\nu)$ be an Alexander-Whitney algebra of finite type, with underlying multiplication $m:A\otimes A\to A$.  Let $H$ be a chain Hopf algebra of finite type, with multiplication $\mu:H\otimes H\to H$.  Let $t:H\to A$ be a twisting cochain.  

If $(\mu, \nu)$ is a morphism in $\cat{Twist_{sh}}$ from $t*t$ to $t$, then the Hochschild complex of $t$ admits a multiplication $\widehat\mu:\hoch(t)\otimes \hoch(t)\to \hoch(t)$ such that 
 $$\xymatrix{A\otimes A\ar [d]_{m}\cof &\hoch(t)\otimes \hoch (t)\ar [d]_{\widehat \mu}\fib & H\otimes H\ar [d]_{\mu}\\
 		    A\cof &\hoch (t) \fib&H}$$
commutes.  Moreover, $\widehat \mu$ is associative (respectively, associative up to chain homotopy) if $(A,\nu)$ is a strict (respectively, quasistrict) Alexander-Whitney algebra.
\end{thm}

Note that, according to Remark \ref{rmk:cartesian-twist}, asking for $(\mu, \nu)$ to be a morphism in $\cat{Twist_{sh}}$ from $t*t$ to $t$ is equivalent to requiring the diagram
$$\xymatrix{H\otimes H\ar [d]_{\beta _{t}\otimes \beta _{t}}\ar [r]^\mu&H\ar [dd]^{\beta_{t}}\\
\Bar A\otimes \Bar A\ar [d]_{\nabla}\\
\Bar (A\otimes A)\ar [r]^\nu&\Bar A}$$
to commute, i.e., to requiring  $\beta_{t}$ to be a map of chain algebras with respect to the multiplication
$$\Bar A\otimes \Bar A \xrightarrow\nabla \Bar (A\otimes A) \xrightarrow \nu \Bar A.$$

\begin{ex}  Since the cobar construction on a finite-type, connected chain Hopf algebra is an Alexander-Whitney algebra, the coHochschild complex of $H$, which is equal to $\hoch(t_{\Om})$, admits a multiplication.
\end{ex}


\section{Power maps on the Hochschild complex of a twisting cochain}\label{sec:powermaps}

Let $t:C\to H$ be a twisting cochain, where $H$ is a chain Hopf algebra.  The goal of this section is to prove the existence, under certain cocommutativity conditions, of an \rth-power map $\widetilde \lambda_{r}$ on the Hochschild complex of $t$, extending the usual \rth-power map $\lambda_{r}$ on $H$ (cf.~equation (\ref{eqn:convolution})).  In the next section we show that if $C$ is the chain complex on a simplicial double supension $K$, and $H=\Om C$, then the algebraic \rth-power map $\widetilde\lambda_{r}$ is topologically meaningful, in the sense that it models the topological \rth-power map on $\op L |K|$.

\subsection{The existence theorem for power maps}

\begin{thm}\label{thm:exists-nthpower}  Let $C$ be a Hirsch coalgebra, with loop comultiplication $\psi: \Om C\to \Om C\otimes \Om C$.  Let $H$ be a chain Hopf algebra, with comultiplication $\delta: H\to H\otimes H$.  Let $t:C\to H$ be a twisting cochain.

 If 
 \begin{enumerate}
\item  the induced chain algebra map $\alpha _{t}:\Om C\to H$ is also a map of coalgebras, and
 \item $\tau\delta t=\delta t$, where $\tau:H\otimes H\xrightarrow\cong H\otimes H$ is the symmetry isomorphism,
 \end{enumerate} 
 then for any positive integer $r$, there is an endomorphism of chain complexes 
 $$\widetilde\lambda_{r}:\hoch( t)\to \hoch( t),$$
 natural with respect to morphisms in $\cat{Twist}_{\mathrm{HH}}$ (cf.~Notation \ref{notn:twist-hiho}), 
 such that 
 $$\xymatrix{H \ar[d]_{\lambda _{r}}\cof & \hoch (t)\ar[d]_{\widetilde\lambda_{r}}\fib & C\ar@{=}[d]\\
 		    H\cof &\hoch (t) \fib& C}$$
commutes, where $\lambda_{r}$ denotes the \rth-power map on $H$.  In particular, if $(\Om C, \psi)$ is primitively generated, then $\widetilde\lambda_{r}=\Id_{C}\otimes \lambda_{r}$.
\end {thm}

There are two special cases of Theorem \ref{thm:exists-nthpower} that are particularly worthy of note.  

\begin {cor}\label{cor:nth-cohoch}  If $(C,\psi)$ is a balanced Hirsch coalgebra,  then the coHochschild complex $\cohoch(C)$ of $C$ admits an \rth-power map $\widetilde\lambda_{r}$, for all positive integers $r$, that is natural with respect to morphisms in $\cat {Hirsch}_{R}$ (cf.~Notation \ref{notn:hirsch}).  In particular,
$$\xymatrix{\Om C \ar[d]_{\lambda _{r}}\cof & \cohoch (C)\ar[d]_{\widetilde\lambda_{r}}\fib & C\ar@{=}[d]\\
 		    \Om C\cof &\cohoch (C) \fib& C}$$
commutes, where $\lambda_{r}$ denotes the \rth-power map on $\Om C$.
\end{cor}

\begin{proof}  Apply Theorem \ref{thm:exists-nthpower} to the twisting cochain $t_{\Om }:C\to \Om C$. Hypotheses (1) and (2) are satisfied because $\alpha_{t_{\Om}}=\Id_{\Om C}$ and because $(C,\psi)$ is balanced. 

With respect to the naturality of $\widetilde \lambda _{r}$, note that if $f\in \cat {Hirsch}\big( (C, \psi), (C', \psi')\big)$, then $(f, \Om f)\in \cat {Twist}_{\text{HH}}(t_{\Om}, t_{\Om})$.
\end{proof}

\begin{cor}\label{cor:nth-hoch}  If $H$ is a cocommutative chain Hopf algebra, then the Hochschild complex $\hoch (H)$ of $H$ admits an \rth-power map $\widetilde\lambda_{r}$, for all positive integers $r$, that is natural with respect to chain Hopf algebra maps.  In particular,
$$\xymatrix{H \ar[d]_{\lambda _{r}}\cof & \hoch (H)\ar[d]_{\widetilde\lambda_{r}}\fib & \Bar H\ar@{=}[d]\\
 		    H\cof &\hoch (H) \fib& \Bar H}$$
commutes, where $\lambda_{r}$ denotes the \rth-power map on $H$.
\end{cor}

\begin{proof} Apply Theorem \ref{thm:exists-nthpower} to the twisting cochain $t_{\Bar}:\Bar H\to H$.  Since $\ve_{H}:\Om \Bar H\to H$ is a map of coalgebras (Theorem \ref{thm:bar-hopf}), hypothesis (1) holds, while hypothesis (2) follows from the cocommutativity of $H$.

With respect to the naturality of $\widetilde \lambda _{r}$, note that Theorem \ref{thm:bar-hopf} implies that  the Alexander-Whitney coalgebra structure on $\Bar H$ is natural in $H$.  Any chain Hopf algebra map $g:H\to H'$ therefore induces a morphism $(\Bar g, g):t_{\Bar }\to t_{\Bar}$ in $\cat {Twist}_{\text{HH}}.$ 
\end{proof}

The naturality of the power map enables us to compare the constructions of the two corollaries above, via a twisting cochain.

\begin{cor}\label{cor:compare-nth} Let $(C,\psi)$ be a balanced Hirsch coalgebra, and let $H$ be a cocommutative chain Hopf algebra with comultiplication $\delta$.  Let $\psi_{H}$ denote the natural comultiplication on $\Om \Bar H$ with respect to which $\ve_{H}:(\Om \Bar H,\psi_{H}) \to (H,\delta)$ is a morphism of chain Hopf algebras (cf.~Theorem \ref{thm:bar-hopf}).

If $t:C\to H$ is a twisting cochain such that  $\Om \beta_{t}: (\Om C, \psi)\to (\Om \Bar H, \psi_{H})$ is a morphism of chain Hopf algebras, then 
$$\xymatrix {\cohoch (C) \ar [d]_{\widetilde\lambda _{r}}\ar [rr]^{\hoch(\beta_{t},\alpha_{t})}_{}&&\hoch (H) \ar [d]_{\widetilde\lambda _{r}}\\
\cohoch (C)\ar [rr]^{\hoch(\beta_{t},\alpha_{t})}_{}&&\hoch (H)}$$
commutes.
\end{cor}

\begin{proof}  Observe that 
$$\xymatrix{ C\ar [d]_{t_{\Om}}\ar [r]^{\beta_{t}}&\Bar H\ar [d]^{t_{\Bar}}\\ \Om C \ar [r]^{\alpha_{t}}& H}$$
always commutes, i.e., that $(\beta_{t},\alpha_{t}):t_{\Om }\to t_{\Bar}$ is always a morphism in $\cat {Twist}$.
Since $\Om \beta_{t}: (\Om C,\psi)\to (\Om \Bar H, \psi_{H})$ is a morphisms of chain Hopf algebras by hypothesis, and $\alpha_{t}=\ve_{H}\circ \Om \beta _{t}$ (Remark \ref{rmk:twisting-induced}), 
$\alpha_{t}:(\Om C,\psi) \to (H,\delta)$ is also a morphism of chain Hopf algebras, whence $(\beta_{t},\alpha_{t})$ 
is actually a morphism in $\cat{Twist}_{\mathrm{HH}}$.
\end{proof}

\begin{ex}  Recall from Example \ref{ex:hpst} that if $L$ is a pointed simplicial set such that $C_{*}L$ is cocommutative, e.g., if $L$ is a simplicial suspension (reduced or unreduced), then $C_{*}\E L$ admits a natural, balanced Alexander-Whitney coalgebra structure.  Corollary \ref{cor:nth-cohoch} therefore implies that if $C_{*}L$ is cocommutative, then $\cohoch(C_*\E L)$ admits an \rth-power map, for all positive integers $r$.   If $L$ is itself a simplicial suspension, then $(\Om C_{*}\E L,\psi_{\E L})$ is primitively generated, and thus $\widetilde \lambda_{r}=\Id_{C_{*}\E L}\otimes \lambda_{r}$.

Moreover, if $C_{*}L$ is cocommutative, then $C_{*}\G\E L$ is a cocommutative chain Hopf algebra, as easily follows from an examination of the formulas for the simplicial suspension functor $\E$ and for the Kan loop group functor $\G$ (cf.~e.g., sections 2.1 (a) and (b) in \cite {hps2}), which imply the existence of a simplicial map 
$$L\to \G\E L: x \mapsto \overline{(1,x)}.$$
It therefore follows from Corollary \ref{cor:nth-hoch} that  if $C_{*}L$ is cocommutative, then $\hoch(C_*\G\E L)$ also admits an \rth-power map, for all positive integers $r$.   

Let $t_{\E L}:C_{*}\E L\to C_{*}\G\E L$ denote the Szczarba twisting cochain for $\E L$ (Example \ref{ex:szczarba}), with associated chain coalgebra map $\beta_{\E L}: C_{*}\E L \to\Bar C_{*} \G\E L$ and chain algebra map $\alpha_{\E L}: \Om C_{*}\E L  \to C_{*}\G\E L$, which together induce a chain map
$$ \hoch(\beta_{\E L},\alpha_{\E L}): \cohoch(C_*\E L)\to \hoch(C_*\G\E L).$$
It is natural to wonder under what conditions this map commutes with the \rth-power maps.  Recall from Example \ref{ex:szczarba} that $\alpha_{\E L}$, and thus $\beta_{\E L}$ and  $\hoch(\beta_{\E L},\alpha_{\E L})$, are quasi-isomorphisms if $L$ is actually reduced, since $\E L$ is then $1$-reduced.

It follows from the proof of Theorem 4.11 in \cite{hps2}  that for any pointed simplicial set $L$, 
$$\alpha_{\E L}:(\Om C_{*}\E L, \psi_{\E L}) \to (C_{*}\G\E L, \Delta)$$ 
is a chain Hopf algebra map, where $\Delta$ is the usual comultiplication on $C_{*}\G \E L$.  On the other hand, 
since $C_{*}\E L$ is a trivial coalgebra, 
$$\beta _{\E L}\big( \overline {(1,x)}\big) =s \big( t_{\E L}\overline {(1,x)}\big),$$
which implies that 
$$\Om \beta _{\E L}\big(s^{-1} \overline {(1,x)}\big) = s^{-1}\Big(s \big( t_{\E L}\overline {(1,x)}\big)\Big).$$
The formulas in the proof of Theorem \ref {thm:aw-bar-dcsh}  for the DCSH structure of the Alexander-Whitney map $f:\Bar (H\otimes H)\to \Bar H \otimes \Bar H$ imply that 
$$\psi_{H}\big(s^{-1}(sa)\big)=s^{-1}(sa)\otimes 1+ 1\otimes s^{-1}(sa)$$
for all $a\in H$ and for all connected chain Hopf algebras $H$.  In particular, therefore,
$$\psi_{C_{*}\G\E L}\circ \Om \beta _{\E L}\big(s^{-1} \overline {(1,x)}\big)=s^{-1}\Big(s \big( t_{\E L}\overline {(1,x)}\big)\Big)\otimes 1+ 1\otimes s^{-1}\Big(s \big( t_{\E L}\overline {(1,x)}\big)\Big)$$
for all $x\in C_{*}L$.

On the other hand, for all $x\in C_{*}L$
\begin{align*}
(\Om \beta_{\E L}\otimes  \Om \beta_{\E L})\circ \psi_{\E L}\big(s^{-1}\overline {(1,x)}\big) &=\Om \beta_{\E L}(\big(s^{-1}\overline {(1,x_{i})}\big) \otimes  \Om \beta_{\E L})\big(s^{-1}\overline {(1,x^{i})}\big) \\
&=s^{-1}\Big(s \big( t_{\E L}\overline {(1,x_{i})}\big)\Big)\otimes s^{-1}\Big(s \big( t_{\E L}\overline {(1,x^{i})}\big)\Big),
\end{align*}
where $\Delta(x)=x_{i}\otimes x^{i}$.  Since $t_{\E L} \overline {(1,y)}\not= 0$ for all $y\in C_{*}L\smallsetminus \{0\}$ (cf.~the explicit formula for $t_{\E L}$ given in \cite{hps2} just before Theorem 4.11), we conclude that $\Om \beta _{\E L}$ is a chain Hopf algebra map if and only if $C_{*}L$ is a trivial coalgebra.  In particular, if $L$ itself is a simplicial suspension (reduced or unreduced), then $\Om \beta _{\E L}$ is a chain Hopf algebra map, and Corollary \ref{cor:compare-nth} therefore implies that
$$\xymatrix {\cohoch (C_{*}\E L) \ar [d]_{\widetilde\lambda _{r}}\ar [rr]^{\hoch(\beta_{\E L},\alpha_{\E L})}&&\hoch (C_{*}\G\E L) \ar [d]_{\widetilde\lambda _{r}}\\
\cohoch (C_{*}\E L)\ar [rr]^{\hoch(\beta_{\E L},\alpha_{\E L})}&&\hoch (C_{*}\G\E L)}$$
commutes.
\end{ex}

\subsection{Proof of the existence of power maps via ``loop concatenation''}

The key to the proof of Theorem \ref{thm:exists-nthpower} is the following result, which is analogous to the existence of topological loop concatenation (cf.~equation (\ref{eqn:top-nth})).  In the statement below, for a Hopf algebra $H$ with multiplication $\mu$ and comultiplication $\delta$, we let $\mu ^{(r)}:H^{\otimes r}\to H$ and $\delta ^{(r)}:H\to H^{\otimes r}$ denote the iterated multiplication and comultiplication maps. 

\begin{thm}\label{thm:exists-concat}  Let $C$ be a Hirsch coalgebra, with loop comultiplication $\psi: \Om C\to \Om C\otimes \Om C$.   Let $H$ be a chain Hopf algebra, with multiplication $\mu$ and comultiplication $\delta: H\to H\otimes H$.  Let $t:C\to H$ be a twisting cochain.

 If 
 \begin{enumerate}
\item  the induced chain algebra map $\alpha _{t}:\Om C\to H$ is also a map of coalgebras, and
 \item $\tau\delta t=\delta t$, where $\tau:H\otimes H\xrightarrow\cong H\otimes H$ is the symmetry isomorphism,
 \end{enumerate} 
 then there is a chain map $\widetilde\mu_{r}:\hoch(\delta ^{(r)} t)\to \hoch (t)$, natural with respect to morphisms in $\cat {Twist}_{\mathrm{Hirsch}}$ (cf.~Notation \ref{notn:twist-hirsch}), such that 
 $$\xymatrix{H^{\otimes r} \ar[d]_{\mu^{(r)}}\cof & \hoch (\delta^{(r)}t)\ar[d]_{\widetilde\mu_{r}}\fib & C\ar@{=}[d]\\
 		    H\cof &\hoch (t) \fib& C}$$
commutes.  In particular, when the Hopf algebra $(\Om C, \psi)$ is primitively generated, $\widetilde\mu _{r}=\Id_{C}\otimes \mu^{(r)}$. 
\end {thm}

\begin{rmk} Existence of the map $\widetilde \mu_{r}$ does not follow immediately from the naturality---even extended---of the Hochschild construction, since $\mu ^{(r)}$ is in general not a map of algebras.
\end{rmk}

\begin{proof} For all $c\in \overline C$, we suppress one summation and write
$$\psi^{(r)}(\si c)=u_{1}(c)\otimes u_{2}(c)\otimes\cdots \otimes u_{r}(c)$$
and
$$u_{1}(c)=\si c_{1}|\cdots|\si c_{k}.$$
Moreover, we write $x_{\theta}\otimes x^{\theta}$ for the result of evaluating the comultiplication of $C$ or of $H$ on an element $x$.

Define $\widetilde\mu _{r}:\hoch (\delta ^{(r)}t)\to \hoch (t)$ by
{\smaller{\begin{align*}
&\widetilde\mu_{r}\big(c\otimes(w_{1}\otimes \cdots \otimes w_{r})\big)\\
&=\sum_{1\leq j\leq k}\pm c_{j}\otimes t(c_{j+1})\cdot\,\ldots\,\cdot t(c_{k})\cdot w_{1}\cdot \alpha_{t}\big(u_{2}(c)\big)\cdot\, \ldots\, \cdot \alpha_{t}\big(u_{r}(c)\big)\cdot w_{r}\cdot t(c_{1})\cdot\,\ldots\,\cdot t(c_{j-1}),
\end{align*}}}

\noindent where the signs are determined by the Koszul rule.  Note that if $(\Om C,\psi)$ is primitively generated, then this formula reduces to $\widetilde\mu_{r}=\Id_{C}\otimes \mu ^{(r)}$. Note further that naturality with respect to morphisms in $\cat{Twist}_{\mathrm{Hirsch}}$ follows immediately from the explicit formula for $\widetilde \mu_{r}$.

Before proving that $\widetilde\mu_{r}$ is a chain map, we make the following two observations, in which all signs are determined by the Koszul rule.
\begin{enumerate}
\item The loop comultiplication $\psi$ is a chain algebra map.  It follows that for all $c\in \overline C$,
\begin{align*}
&u_{1}(dc)\otimes \cdots \otimes u_{r}(dc)\pm u_{1}(c_{\theta})u_{1}(c^{\theta})\otimes \cdots \otimes u_{r}(c_{\theta})u_{r}(c^{\theta})\\
&=\sum _{1\leq q\leq r}\pm u_{1}(c)\otimes \cdots \otimes du_{q}(c)\otimes \cdots \otimes u_{r}(c)\\
&=\sum_{1\leq j\leq k}\pm \si c_{1}|\cdots|\si (dc_{j})|\cdots |\si c_{k} \otimes u_{2}(c)\otimes \cdots \otimes u_{r}(c)\\
&\qquad\qquad \pm \si c_{1}|\cdots|\si (c_{j})_{\theta}|\si (c_{j})^{\theta}|\cdots |\si c_{k} \otimes u_{2}(c)\otimes \cdots \otimes u_{r}(c)\\
&\quad+ \sum _{2\leq q\leq r}\pm u_{1}(c)\otimes \cdots \otimes du_{q}(c)\otimes \cdots \otimes u_{r}(c).
\end{align*}

\item The hypothesis that $\tau \delta t= \delta t$ implies that $\tau (\alpha_{t}\otimes \alpha_{t})\psi=(\alpha_{t}\otimes \alpha_{t})\psi$, since $\alpha_{t}$ is a coalgebra map.  Consequently, for all permutations $\sigma$ of the set $\{1,...,r\}$,
$$\alpha_{t}\big(u_{1}(c)\big)\otimes \cdots \otimes \alpha_{t}\big(u_{r}(c)\big)=\pm \alpha_{t}\big(u_{\sigma(1)}(c)\big)\otimes \cdots \otimes \alpha_{t}\big(u_{\sigma(r)}(c)\big).$$
\end{enumerate}

To simplify our calculations, without true loss of generality, we suppose henceforth that $dw_{j}=0$ for all $j$, and that the differential on $C$ is identically 0, from which it follows that 
\begin{equation}\label{eqn:dt}
dt(c)=t(c_{\theta})t(c^{\theta})
\end{equation}
for all $c\in \overline C$.  The general case proceeds in essentially the same manner.  We also drop all signs, as they are all determined by the Koszul rule.  

To simplify notation a bit, we let
$$h(c\otimes w_{1}\otimes \cdots \otimes w_{r})= w_{1}\cdot \alpha_{t}\big(u_{2}(c)\big)\cdot\, \ldots\, \cdot \alpha_{t}\big(u_{r}(c)\big)\cdot w_{r}.$$
Under our hypothesis that $dw_{j}=0$ for all $j$, 
\begin{equation}\label{eqn:dh}
\begin{split}
dh&(c\otimes w_{1}\otimes \cdots \otimes w_{r})\\
&=\sum _{2\leq q \leq r}w_{1}\cdot \alpha_{t}\big(u_{2}(c)\big)\cdot\, \ldots\, \cdot \alpha_{t}\big(du_{q}(c)\big)\cdot\, \ldots\, \cdot \alpha_{t}\big(u_{r}(c)\big)\cdot w_{r}.
\end{split}
\end{equation}

Note that when the differential on $C$ is identically zero, and we are ignoring signs, the equation in observation (1) above reduces to 
{\smaller\begin{equation}\label{eqn:key-eqn}
\begin{split}
 \si (c_{\theta})_{1}&|\cdots|\si (c_{\theta})_{k'}|\si (c^{\theta})_{1}|\cdots |\si (c^{\theta})_{k''}\otimes u_{2}(c_{\theta})\cdot u_{2}(c^{\theta})\otimes \cdots \otimes u_{r}(c_{\theta})\cdot u_{r}(c^{\theta})\\
&=\sum_{1\leq j\leq k} \si c_{1}|\cdots|\si (c_{j})_{\theta}|\si (c_{j})^{\theta}|\cdots |\si c_{k} \otimes u_{2}(c)\otimes \cdots \otimes u_{r}(c)\\
&\quad+ \sum _{2\leq q\leq r} \si c_{1}|\cdots|\si c_{k}\otimes u_{2}(c)\otimes \cdots \otimes du_{q}(c)\otimes \cdots \otimes u_{r}(c).
\end{split}
\end{equation}}

To see that $\widetilde \mu_{r}$ is a chain map, note first that
{\smaller \begin{align*}
&d_{t}\widetilde \mu _{r}\big(c\otimes(w_{1}\otimes \cdots \otimes w_{r})\big)\\
&=\sum _{1\leq j\leq k} (c_{j})_{\theta}\otimes t\big((c_{j})^{\theta}\big)\cdot t(c_{j+1})\cdot\,\ldots\,\cdot t(c_{k})\cdot h(c\otimes w_{1}\otimes \cdots \otimes w_{r})\cdot t(c_{1})\cdot\,\ldots\,\cdot t(c_{j-1})\\
&\qquad\quad + (c_{j})^{\theta}\otimes t(c_{j+1})\cdot\,\ldots\,\cdot t(c_{k})\cdot h(c\otimes w_{1}\otimes \cdots \otimes w_{r})\cdot t(c_{1})\cdot\,\ldots\,\cdot t(c_{j-1})\cdot t\big((c_{j})_{\theta}\big)\\
&\qquad\quad + c_{j}\otimes \sum _{j+1\leq l\leq k} t(c_{j+1})\cdot\,\ldots\,\cdot t\big((c_{l})_{\theta}\big)\cdot t\big((c_{l})^{\theta}\big)\cdot\,\ldots\,\cdot  t(c_{k})\\
&\hphantom{\qquad\quad + c_{j}\otimes \sum _{j+1\leq l\leq k} t(c_{j+1})\cdot\,\ldots\,\cdot t\big((c_{l})_{\theta}\big)}\qquad\cdot h(c\otimes w_{1}\otimes \cdots \otimes w_{r})\cdot t(c_{1})\cdot\,\ldots\,\cdot t(c_{j-1})\\
&\qquad\quad + c_{j}\otimes t(c_{j+1})\cdot\,\ldots\,\cdot t(c_{k})\cdot dh(c\otimes w_{1}\otimes \cdots \otimes w_{r})\cdot t(c_{1})\cdot\,\ldots\,\cdot t(c_{j-1})\\
&\qquad\quad + c_{j}\otimes \sum _{1\leq l\leq j-1} t(c_{j+1})\cdot\,\ldots\,\cdot   t(c_{k})\cdot h(c\otimes w_{1}\otimes \cdots \otimes w_{r})\\
&\hphantom{\qquad\quad + c_{j}\otimes \sum _{1\leq l\leq j-1} t(c_{j+1})\cdot\,\ldots\,\cdot t\big((c_{l})_{\theta}\big)}\qquad\cdot t(c_{1})\cdot\,\ldots\,\cdot t\big((c_{l})_{\theta}\big)\cdot t\big((c_{l})^{\theta}\big)\cdot\,\ldots\,\cdot t(c_{j-1}),
\end{align*}}

\noindent where we have applied equation (\ref{eqn:dt}) and the formula for $d_{t}$ from Definition \ref{defn:Hochschild-cx}.

Next observe that
{\smaller
\begin{align*}
&\widetilde\mu_{r}d_{\delta^{(r)}t}\big(c\otimes(w_{1}\otimes \cdots \otimes w_{r})\big)=\widetilde\mu _{r}\bigg(c_{\theta}\otimes \Big(\alpha_{t}\big(u_{1}(c^{\theta})\big)\cdot w_{1}\otimes \cdots \otimes\alpha_{t}\big(u_{r}(c^{\theta})\big)\cdot w_{r}\Big)\\
&\hphantom{\widetilde\mu_{r}d_{\delta^{(r)}t}\big(c\otimes(w_{1}\otimes \cdots \otimes w_{r})\big)=\widetilde\mu _{n}\Big(}+ c^{\theta}\otimes \Big(w_{1}\cdot\alpha_{t}\big(u_{1}(c_\theta)\big)\otimes \cdots \otimes w_{r}\cdot\alpha_{t}\big(u_{r}(c_{\theta})\big)\Big)\bigg)\\
&=\sum_{1\leq j\leq k'}(c_{\theta})_{j}\otimes t\big((c_{\theta})_{j+1}\big)\cdot\,\ldots\,\cdot t\big((c_{\theta})_{k'}\big)\cdot h\Big(c_{\theta}\otimes \alpha_{t}\big(u_{1}(c^{\theta})\big)\cdot w_{1}\otimes\cdots\otimes \alpha_{t}\big(u_{r}(c^{\theta})\big)\cdot w_{r}\Big)\\
&\hphantom{=\sum_{1\leq j\leq k}(c_{\theta})_{j}\otimes t\big((c_{\theta})_{j+1}\big)\cdot\,\ldots\,\cdot t\big((c_{\theta})_{k}\big)\cdot h\Big(c_{\theta}\otimes \alpha_{t}\big(u_{1}(c^{\theta})\big)}\qquad\cdot t\big((c_{\theta})_{1}\big)\cdot\,\ldots\,\cdot t\big((c_{\theta})_{j-1}\big)\\
&\quad+\sum_{1\leq j\leq k''}(c^{\theta})_{j}\otimes t\big((c^{\theta})_{j+1}\big)\cdot\,\ldots\,\cdot t\big((c^{\theta})_{k''}\big)\cdot h\Big(c^{\theta}\otimes w_{1}\cdot\alpha_{t}\big(u_{1}(c_\theta)\big)\otimes \cdots\otimes w_{r}\cdot\alpha_{t}\big(u_{r}(c_\theta)\big)\Big)\\
&\quad\hphantom{=\sum_{1\leq j\leq k}(c_{\theta})_{j}\otimes t\big((c_{\theta})_{j+1}\big)\cdot\,\ldots\,\cdot t\big((c_{\theta})_{k}\big)\cdot h\Big(c_{\theta}\otimes \alpha_{t}\big(u_{1}(c^{\theta})\big)}\qquad\cdot t\big((c^{\theta})_{1}\big)\cdot\,\ldots\,\cdot t\big((c^{\theta})_{j-1}\big).
\end{align*}
}

\noindent Implicit in this calculation is the fact that 
$$\delta ^{(r)}t=\alpha_{t}^{\otimes r}\psi ^{(r)}t_{\Om}:C\to H^{\otimes r},$$
since $\alpha_{t}$ is a coalgebra map.

Finally,
{\smaller\begin{equation}\label{eqn:h1}
\begin{split}
&h\Big(c_{\theta}\otimes \alpha_{t}\big(u_{1}(c^{\theta})\big)\cdot w_{1}\otimes\ldots\otimes \alpha_{t}\big(u_{r}(c^{\theta})\big)\cdot w_{r}\Big)\\
&=\alpha_{t}\big(u_{1}(c^{\theta})\big)\cdot w_{1}\cdot \alpha_{t}\big(u_{2}(c_\theta)\big)\cdot \alpha_{t}\big(u_{2}(c^{\theta})\big)\cdot\,\ldots\, \cdot \alpha_{t}\big(u_{r}(c_\theta)\big)\cdot\alpha_{t}\big(u_{r}(c^{\theta})\big)\cdot w_{r}
\end{split}
\end{equation}}

\noindent and
{\smaller\begin{equation}\label{eqn:h2}
\begin{split}
&h\Big(c^{\theta}\otimes w_{1}\cdot\alpha_{t}\big(u_{1}(c_\theta)\big)\otimes \cdots\otimes w_{r}\cdot\alpha_{t}\big(u_{r}(c_\theta)\big)\Big)\\
&= w_{1}\cdot \alpha_{t}\big(u_{1}(c_\theta)\big)\cdot \alpha_{t}\big(u_{2}(c^{\theta})\big)\cdot\,\ldots\, \cdot \alpha_{t}\big(u_{r-1}(c_\theta)\big)\cdot\alpha_{t}\big(u_{r}(c^{\theta})\big)\cdot w_{r}\cdot \alpha_{t}\big(u_{r}(c_\theta)\big)\\
&= w_{1}\cdot \alpha_{t}\big(u_{2}(c_\theta)\big)\cdot \alpha_{t}\big(u_{2}(c^{\theta})\big)\cdot\,\ldots\, \cdot \alpha_{t}\big(u_{r}(c_\theta)\big)\cdot\alpha_{t}\big(u_{r}(c^{\theta})\big)\cdot w_{r}\cdot \alpha_{t}\big(u_{1}(c_\theta)\big),
\end{split}
\end{equation}}

\noindent where we use observation (2) above in the final step of this calculation.  Recall that we are suppressing both a summation and signs from the notation.

We now explain why the terms of $d_{t}\widetilde \mu _{r}\big(c\otimes(w_{1}\otimes \cdots \otimes w_{r})\big)$ cancel with those of $\widetilde\mu_{r}d_{\delta^{(r)}t}\big(c\otimes(w_{1}\otimes \cdots \otimes w_{r})\big)$. Define a morphism of graded $R$-modules
$$\Gamma: (\Om C)^{\otimes r} \to C\otimes H$$
by
\begin{align*}
\Gamma& (v_{1}\otimes \cdots \otimes v_{r})\\
&=\sum_{1\leq i\leq k}\pm a_{i}\otimes t(a_{i+1})\cdot \ldots\cdot t(a_{k})\cdot w_{1}\cdot \alpha_{t}(v_{2})\cdot \ldots\cdot \alpha_{t}(v_{r})\cdot w_{r}\cdot t(a_{1})\cdot \ldots\cdot t(a_{i-1}),
\end{align*}
where $w_{1},...,w_{r}$ are the elements of $H$ fixed above, $v_{1}=\si a_{1}|\cdots |\si a_{k}$, and the signs are given by the Koszul rule, as usual.

Note that $d_{t}\widetilde\mu _{r}\big(c\otimes (w_{1}\otimes\cdots \otimes w_{r})\big)$ is equal to the result of applying $\Gamma$ to the right-hand side of equation (\ref{eqn:key-eqn}).  On the other hand, it follows from equations (\ref{eqn:h1}) and (\ref{eqn:h2}) that $\widetilde\mu _{r}d_{\delta ^{(r)}t}\big(c\otimes (w_{1}\otimes\cdots \otimes  w_{r})\big)$ is equal to the result of applying $\Gamma $ to the left-hand side of equation (\ref{eqn:key-eqn}).  

The naturality of $\widetilde \mu_{r}$ follows easily from inspection of its definition.
\end{proof}

Motivated by the definition of topological \rth-power map on a free loop space (\ref{eqn:top-nth}), we now complete the proof of the existence of the algebraic \rth-power map.

\begin{proof}[Proof of Theorem \ref{thm:exists-nthpower}] 
Since the diagram
$$\xymatrix{C\ar [d]_{t}\ar @{=}[r]&C\ar[d]^{\delta ^{(r)}t}\\ H\ar [r]^{\delta ^{(r)}} &H^{\otimes r}}$$
commutes, the pair $(\Id_{C},\delta^{(r)})$ is a morphism in $\cat {Twist}$ from $t$ to $\delta ^{(r)}t$ and therefore induces a chain map $\widetilde\delta_{r}:=\hoch (\Id_{C}, \delta^{(r)}): \hoch (t)\to \hoch (\delta^{(r)}t)$ such that 
 $$\xymatrix{H \ar[d]_{\delta^{(r)}}\cof & \hoch (t)\ar[d]_{\widetilde\delta_{r}}\fib & C\ar@{=}[d]\\
 		    H^{\otimes r}\cof &\hoch (\delta^{(r)}t) \fib& C}$$
commutes.  Note that $\widetilde \delta_{r}$ is natural with respect to morphisms in $\cat {Twist}_{\mathrm {Hopf}}$, as chain Hopf algebra maps commute with iterated comultiplications.

Let $\widetilde\lambda _{r}=\widetilde \mu _{r}\circ \widetilde \delta_{r}$, which is natural with respect to morphisms in $\cat {Twist}_{\mathrm{HH}}$, since $\widetilde\delta_{r}$  and $\widetilde \mu _{r}$ are natural with respect to morphisms in $\cat {Twist}_{\mathrm{Hopf}}$ and $\cat {Twist}_{\mathrm{Hirsch}}$, respectively.
\end{proof}

\begin{rmk}\label{rmk:formula} Combining the formula developed in the proof of Theorem \ref{thm:exists-concat} for $\widetilde\mu _{r}$ with the identity $\widetilde\delta _{r}=\Id_{C}\otimes \delta ^{(r)}$, we obtain the following formula for $\widetilde \lambda_{r}$.  If $c\in C$ and $w\in H$, then
{\smaller{\begin{align*}
&\widetilde\lambda_{r}(c\otimes w)\\
&=\sum_{1\leq j\leq k}\pm c_{j}\otimes t(c_{j+1})\cdot\,\ldots\,\cdot t(c_{k})\cdot w_{1}\cdot \alpha_{t}\big(u_{2}(c)\big)\cdot\, \ldots\, \cdot \alpha_{t}\big(u_{r}(c)\big)\cdot w_{r}\cdot t(c_{1})\cdot\,\ldots\,\cdot t(c_{j-1}),
\end{align*}}}

\noindent where (suppressing obvious summations)
$$\delta^{(r)}(w)=w_{1}\otimes \cdots \otimes w_{r},$$
$$\psi^{(r)}(\si c)=u_{1}(c)\otimes u_{2}(c)\otimes\cdots \otimes u_{r}(c),$$
and
$$u_{1}(c)=\si c_{1}|\cdots|\si c_{k}.$$
\end{rmk}

\begin{ex}  Recall Example \ref{ex:hpst}.  Let $K$ be a simplicial double suspension, either $\E^2 L$ for some pointed simplicial set $L$ or $\S M$ for some simplicial set $M$.  If $C=C_{*}K$, then  every element of $C$ is primitive and $\Om C$ is primitively generated, which implies that
$$\cohoch (C)=\big(C\otimes T(s^{-1}\overline C), d_{\cohoch}\big),$$
where, if $[-,-]$ denotes the graded commutator, then
$$d_{\cohoch}(c\otimes w)= dc\otimes w +(-1)^{|c|} c\otimes d_{\Om}w- 1\otimes [s^{-1} c, w],$$ 
and
$$\widetilde \lambda_{r}(x\otimes   w)=x\otimes \lambda_{r} (w)$$
for all $x\in  C$ and for all $w\in \Om C$.
\end{ex}

\begin{ex}\label{ex:rpinfty} Let $K$ be the nerve of the cyclic group of order two, which is a reduced simplicial model of $\mathbb RP^\infty$ (cf.~Remark \ref{rmk:rpinfty}).  An easy calculation shows that $C_{k}K$ is free abelian on one generator $z_{k}$ for each $k$.  Moreover, 
$$\Delta (z_{k})=\sum _{i=0}^k z_{i}\otimes z_{k-i}$$
for all $k$, which implies that $C_{*}K\otimes \mathbb F_{2}$ is cocommutative, i.e.,  that $K$ is a symmetric simplicial set.  Note that the differential in $C_{*}K\otimes \mathbb F_{2}$ is exactly $0$.

Let $C=C_{*}\E K\otimes \mathbb F_{2}$, and let $y_{k}$ denote the suspension of $z_{k}$. Consider $\cohoch (C)=(C\otimes T(C_{+}K), d_{\cohoch})$, where
$$
d_{\cohoch}(y_{l}\otimes z_{k_{1}}|\cdots |z_{k_{m}})= -1\otimes z_{l}|z_{k_{1}}|\cdots |z_{k_{m}} +(-1)^{lk} 1\otimes z_{k_{1}}|\cdots |z_{k_{m}}|z_{l},
$$
where $k=k_{1}+\cdots + k_{m}$.
Moreover, for all $k$,
$$\psi_{\E K}(z_{k})=\sum _{i=0}^k z_{i}\otimes z_{k-i}\in TC_{+}K\otimes TC_{+}K$$
and so 
$$\psi_{\E K}^{(r)}(z_{k})=\sum _{k_{1}+\cdots +k_{r}=k} z_{k_{1}}\otimes \cdots \otimes z_{k_{r}}\in (TC_{+}K)^{\otimes r}$$
for all $r$.
The formula in Remark \ref{rmk:formula} therefore implies that
\begin{align*}
\widetilde \lambda _{r}&(y_{l}\otimes z_{k_{1}}|\cdots |z_{k_{m}})\\
=&\sum  y_{l_{1}}\otimes z_{k_{1,1}}|\cdots |z_{k_{m,1}}|z_{l_{2}}|\cdots |z_{l_{r}}| z_{k_{1,r}}|\cdots |z_{k_{m,r}},
\end{align*}
where the sum is taken over 
\begin{itemize}
\item all $(l_{1},...,l_{r})$ such that $\sum _{j}l_{j}=l$, and
\item all $(k_{i,1},...,k_{i,r})$ such that $\sum _{j}k_{i,j}=k_{i}$, for all $1\leq i\leq m$.
\end{itemize}
\end{ex}

\begin{rmk}  All of the results in this section can be dualized, at least in the finite-type case.  We leave the straightforward task of dualizing the statements to the interested reader.
\end{rmk}


\section{The topological power map and its models}\label{sec:geometry}

Let $X$ be a connected topological space that has the homotopy type of the realization of a reduced simplicial set $K$, and let $r\in \mathbb N$. In this section we show that, under certain conditions on $X$, the algebraic \rth-power map 
$$\widetilde \lambda_{r}:\cohoch (C_{*}K)\to \cohoch (C_{*}K)$$  
of Corollary \ref{cor:nth-cohoch} is a model of the topological \rth-power map 
$$\widetilde \lambda ^{(r)}_{top}:\op LX\to \op LX$$ 
(cf.~(\ref{eqn:top-nth})).  More precisely, we prove the following theorem, which is essential to the construction of the chain model of topological cyclic homology in \cite {hess-rognes}. 

\begin{thm}\label{thm:topo-alg} If $X$ is a topological space such that $X\simeq |K|$ for some reduced simplicial set $|K|$, then there is a chain map 
$$\zeta:  \cohoch (C_{*}K)\xrightarrow {} S_{*} \op LX,$$
which is a quasi-isomorphism if $K$ is $1$-reduced.
Moreover,  if $K=\S M$, where $M$ is any simplicial set, then for all $r\in \mathbb N$,
$$\xymatrix{\cohoch(C_{*}K) \ar[d]_{\widetilde \lambda _{r}}\ar [r]^(0.55){\zeta}_(0.55){}&S_{*} \op LX\ar[d] ^{S_{*}\widetilde\lambda _{top}^{(r)}}\\
\cohoch (C_{*}K) \ar [r]^(0.55){\zeta}_(0.55){}&S_{*} \op L X }$$
commutes up to chain homotopy.
\end{thm}

\begin{rmk}  In terms of the notation used in the introduction, the algebraic model of $\op LX$ that we propose to use in studying power maps is thus $fls_{*}(X):=\cohoch (C_{*}K)$, when $K$ is a 1-reduced simplicial set, and $|K|\simeq X$.
\end{rmk}

\begin{rmk} Recall the definition of symmetric simplicial sets (Definition \ref{defn:symmetric}).  In future work we will show that the result above also holds, over $\mathbb F_{2}$, for topological spaces that realize symmetric simplicial sets.  In particular, if $X$ is a suspension of a wedge of truncated real projective spaces, then the algebraic \rth-power map is a model of the topological \rth-power map, over $\mathbb F_{2}$.
\end{rmk}

Before proving Theorem \ref{thm:topo-alg}, we illustrate its utility by applying it to carrying out explicit calculations of power maps in the homology of $\op L S^n$, for $n\geq 2$.

\begin{ex} Let $n\geq 2$, and let $S^n$ denote the simplicial model of the $n$-sphere with exactly one nondegenerate simplex of positive degree: $y$ in degree $n$.  An easy calculation shows that 
$$\cohoch(C_{*}S^n)\cong \big ((\zz \cdot 1 \oplus \zz \cdot y)\otimes Tx, D_{\cohoch}\big ),$$
where $Tx$ denotes the tensor algebra on a generator $x$ of degree $n-1$ and the differential $D$ is specified by
$$D(1\otimes x^k)=0$$
and
$$D(y\otimes x^k)=\begin{cases} 0&: \text{if }n \text{ odd or } k \text{ even}\\ -2\otimes x^{k+1}&:\text{if } n\text{ even and } k \text{ odd.}\end{cases}$$
Thus, as is well known, $H_{*}(\op LS^{2n+1})\cong(\zz \cdot 1 \oplus \zz \cdot y)\otimes Tx$ as graded $\zz$-modules, while 
$$H_{*}(\op LS^{2n})\cong \zz \oplus \bigoplus _{k\geq 0} \zz \cdot (1\otimes x^{2k+1}) \oplus\bigoplus_{k\geq 1}\zz/2\zz \cdot (1\otimes x^{2k})\oplus\bigoplus _{k\geq 0}\zz\cdot (y\otimes x^{2k}).$$

Since $S^n$ is a double suspension for all $n\geq 2$ and therefore $(\Om C_{*}S^n, \psi_{S^n})$ is primitively generated, the algebraic \rth-power map $\widetilde\lambda _{r}$ on $\cohoch(C_{*}S^n)$ is simply $\Id\otimes \lambda_{r}$, i.e.,
$$\widetilde\lambda _{r}(y\otimes x^k)=y\otimes \lambda_{r}(x^k)=y\otimes r^k x^k,$$
since  $\lambda_{r}$ is the usual \rth-power map on the primitively generated tensor algebra $Tx$.
\end{ex}

All of the simplicial constructions (the ordinary and cyclic bar constructions, Artin-Mazur totalization, twisted cartesian products, etc.) used in this section are recalled in Appendix \ref{sec:simpl-constr}, for the convenience of the reader and to fix our notation.  

\subsection{The B\"okstedt-Hsiang-Madsen model of the power map}

The starting point for our proof of Theorem \ref{thm:topo-alg} is the following construction due to B\"okstedt, Hsiang and Madsen, which can be found in sections 1 and 2 of  \cite{bhm}.  

Let $G$ be a topological group.  Let $\Z G$ and $\B G$ denote the cyclic and ordinary simplicial bar constructions on $G$ (see Appendix \ref{sec:simpl-constr}), and let $\pi:\Z G \to \B G$ be the obvious projection map.  Since $\Z G$ is a cyclic set, its geometric realization admits a natural circle action, which, when composed with $|\pi|$, gives rise to a continuous map
$$S^1\times |\Z G| \to |\Z G|\xrightarrow {|\pi|} |\B G|.$$
Taking the adjoint
$$h:|\Z G|\to \op L |\B G|$$
one obtains a homotopy equivalence, which B\"okstedt, Hsiang and Madsen proved to be compatible up to homotopy with power maps, as follows.

Let $\boldsymbol \Delta$ denote the usual simplex category, with objects $[n]$ for every non-negative integer $n$. Let $sd_{r}:\boldsymbol \Delta \to \boldsymbol \Delta$ denote the $r$-fold edgewise subdivision functor that takes $[n-1]$ to its $r$-fold concatenation $[rn-1]\cong [n-1]\coprod \cdots \coprod [n-1]$ and $\varphi \in \boldsymbol \Delta \big ([n-1], [m-1]\big)$ to $\varphi \coprod\cdots\coprod \varphi$.  The $r$-fold edgewise subdivision of a simplicial object $K$ in a category $\cat C$, denoted $sd_{r}K$, is then given by the composite $K\circ sd_{r}$, where we consider $K$ as a functor from $\boldsymbol \Delta^{op}$ to $\cat C$.  

For all simplicial sets (or spaces) $K$ and for all $r\in \mathbb N$, B\"okstedt, Hsiang and Madsen constructed a natural homeomorphism 
\begin{equation}\label{eqn:homeo}
D_{r}:|sd_{r}K| \xrightarrow \cong |K|.
\end{equation}
They observed moreover that, for any topological group $G$,  $\Z G$ embeds into $sd_{r}\Z G$ as the fixed points of the natural action of $C_{r}$, the cyclic group of order $r$, on $sd_{r}\Z G$, i.e., there is a simplicial injection given in dimension $n$ by
$$\iota_{r}:(\Z G)_{n}\to (sd_{r}\Z G)_{n}: (a_{1},...,a_{n},b)\mapsto (a_{1},...,a_{n},b, a_{1},...,a_{n},b,... ,a_{1},...,a_{n},b),$$
where the sequence $a_{1},...,a_{n},b$ appears $r$ times.

In the course of the proof of Proposition 2.6 in \cite{bhm}, the authors establish the following equivariance result.

\begin{prop}\label{prop:bhm}\cite{bhm} For every topological group $G$ and for every  $r\in \mathbb N$, the diagram
$$\xymatrix{|\Z G|\ar [r]^h_{\simeq}\ar [d] _{D_{r}\circ |\iota _{r}|}&\op L |\B G|\ar [d] ^{\widetilde \lambda_{top}^{(r)}}\\
|\Z G|\ar [r]^h_{\simeq}&\op L |\B G|}$$
commutes up to homotopy.
\end{prop}

\subsection{Various simplicial models of the power map}

We begin by defining, for all $r\in \mathbb N$ and for all simplicial groups $G$, a simplicial map $\widetilde \lambda _{simp}^{(r)}:\Z G\to \Z G$ extending the usual simplicial $r^{\text{th}}$-power map $\lambda^{(r)}$ on $\C G$, the constant simplicial group.

\begin{rmk}  For all $r\in \mathbb N$ there is a natural transformation $e_{r}:\Id_{\boldsymbol \Delta}\to sd_{r}$, which is defined as follows.  For any object $[n-1]$ in $\boldsymbol \Delta$, $e_{r}([n-1]):[n-1]\to [rn-1]$ is given by inclusion onto the final (right-hand) copy of $[n-1]$.  The natural transformation $e_{r}$ defines in turn a natural map $e_{\bullet}:sd_{r}K \to K$, for all simplicial sets $K$.
\end{rmk} 

\begin{prop}\label{prop:john}  For all $r\in \mathbb N$ and for all simplicial sets $K$, the homeomorphism $D_{r}:|sd_{r}K|\to |K|$ (\ref{eqn:homeo}) is naturally homotopic to the geometric realization of $e_{\bullet}:sd_{r}K\to K$.
\end{prop}

\begin{proof}  Both $D_{r}$ and $|e_{\bullet}|$ embed the geometric $(n-1)$-simplex in $|sd_{r}K|$ associated to an element in $sd_{r}K_{n-1}=K_{rn-1}$ linearly into the corresponding geometric $(rn-1)$-simplex in $|K|$.  The convex linear combination of these two maps, formed within the latter simplex, thus defines a homotopy between $D_{r}$ and $|e_{\bullet}|$.
\end{proof}

Taken together, Propositions \ref{prop:bhm} and \ref{prop:john} imply the existence of the following purely simplicial model of the power map.

\begin{cor}\label{cor:simpl-rth}  For all $r\in \mathbb N$ and for all topological groups $G$, the map of graded sets $\widetilde \lambda _{simp}^{(r)}:\Z G\to \Z G$, given in degree $n$ by
$$\widetilde \lambda _{simp}^{(r)} (a_{1},..., a_{n},b)=(a_{1},..., a_{n},b(ab)^{r-1}),$$
where $a=a_{1}\cdots a_{n}$, is a simplicial map.  Moreover 
$$\xymatrix{\C G \ar [d] ^{\lambda^{(r)} }\cof &\Z G\ar [d]^{\widetilde\lambda_{simp}^{(r)} }\fib &\B G\ar @{=}[d]\\
	\C G\cof &\Z G\fib &\B G}$$
commutes, and 
$$\xymatrix{|\Z G|\ar [rr]^h_{\simeq}\ar [d] ^{|\widetilde \lambda _{simp}^{(r)}|}&&\op L |\B G|\ar [d] ^{\widetilde \lambda_{top}^{(r)}}\\
|\Z G|\ar [rr]^h_{\simeq}&&\op L |\B G|}$$
commutes up to homotopy, where $h$ is the equivalence of Proposition \ref{prop:bhm}.
\end{cor}

\begin{proof}  Observe that for any simplicial set $e_{\bullet}:sd_{r}K\to K$ is equal in degree $n-1$ to the composite
$$(sd_{r}K)_{n-1}\xrightarrow= K_{rn-1}\xrightarrow {d_{0}^{(r-1)n}} K_{n-1}.$$
A straightforward simplicial computation then shows that $\widetilde \lambda _{simp}^{(r)} =e_{\bullet}\circ \iota _{r}$, which implies both that $\widetilde \lambda _{simp}^{(r)} $ is simplicial and that the second diagram commutes up to homotopy.  That the first diagram commutes is obvious from the definition of $\widetilde \lambda _{simp}^{(r)} $. 
\end{proof}

To simplify this simplicial model of the power map further, we next consider simplicial groups, instead of topological groups. 

Let $\tot$ denote the Artin-Mazur totalization functor.  As recalled in  Appendix \ref{sec:simpl-constr} (cf.~Remark \ref{rmk:totalize}), there is an isomorphism $\tot \B G\cong \overline \W G$ for every simplicial group $G$.  

\begin{defn} Let $G$ be a simplicial group. The \emph{simplicial Hochschild construction} on $G$, denoted $\H G$, is the simplicial set $\overline \W G \underset {\upsilon_{G}}\times {}^\gamma G$, where ${}^\gamma G$ denotes $G$ seen as a left $G$-space, with $G$ acting on itself by conjugation, and $\upsilon_{G}$ denotes the couniversal twisting function $\upsilon_{G}: \overline \W  G\to G$ (cf.~Remark \ref{rmk:couniv-twist}).  
\end{defn}

\begin{thm}\label{thm:tot-zg} If $G$ is a simplicial group, then applying the totalization functor to the sequence of bisimplicial maps
$$\xymatrix@1{\C G\;\cof &\; \Z G\; \fib& \;\B G}$$
gives rise to a Kan fibration
$$\xymatrix@1{G\;\cof &\; \H G\; \fib& \;\overline \W G.}$$
\end{thm}

\begin{proof} As recalled in Definition \ref{defn:artin-mazur},
$$\tot(\Z G)_n=\{(x_{0},...,x_{n})\in \prod_{i=0}^n \Z G_{i,n-i}\mid d_{0}^v x_{i}=d^h_{i+1}x_{i+1} \quad \forall\; 0\leq i<n\}.$$ 
Given $(x_{0},...,x_{n})\in\tot(\Z G)_n$, write for each $i$,
$$x_{i}=(a_{i,1},...,a_{i,n-i},b_{i}),$$
where $a_{i,j}, b_{i}\in G_{i}$ for all $j$. 
The condition $d_{0}^v x_{i}=d^h_{i+1}x_{i+1}$ then says that $d_{i+1}b_{i+1}=b_{i}a_{i,1}$ and that $d_{i+1}a_{i+1, j}=a_{i ,j+1}$ for all $1\leq j<n-i$.  It follows that
\begin{align*}
b_{n-1}&=d_{n}b_{n}\cdot a_{n-1, 1}^{-1}\\
b_{n-2}&=d_{n-1}b_{n-1}\cdot a_{n-2, 1}^{-1}=d_{n-1}d_{n}b_{n}\cdot  a_{n-2 ,2}^{-1}a_{n-2, 1}^{-1}\\
&\vdots\\
b_{0}&=d_{1}b_{1}\cdot a_{0,1}^{-1}=\cdots =d_{1}\cdots d_{n}b_{n}\cdot a_{0,n}^{-1}\cdots a_{0,1}^{-1}
\end{align*}
and that 
\begin{align*}
a_{0,n}&= d_{1}a_{1, n-1}=d_{1}d_{2}a_{2 ,n-2}=\cdots = d_{1}\cdots d_{n-1}a_{n-1, 1}\\
a_{0,n-1}&= d_{1}a_{1, n-2}=d_{1}d_{2}a_{2, n-3}=\cdots = d_{1}\cdots d_{n-2}a_{n-2, 1}\\
&\vdots\\
a_{0,2}&=d_{1}a_{1,1}.
\end{align*}

Define a map of graded sets $\kappa: \tot (\Z G)\to \H G$ by 
$$\kappa_{n}: \tot (\Z G)_{n} \to \H G_{n}: (x_{0},...,x_{n})\mapsto (a_{0,1}, a_{1,1},...,a_{n-1,1},b_{n}),$$
where we use the same notation as above.  It is a straightforward calculation to show that $\kappa$ is indeed a simplicial map.  Moreover, $\kappa$ admits an inverse $\kappa^{-1}$, defined as follows.

Given $(a_{0},...,a_{n-1},b_{n})\in \H G_{n}$, let $a_{i,1}=a_{i}$ and for all $j>1$,
$$a_{i,j}=d_{i+1}\cdots d_{i+j-1}a_{i+j-1}.$$
Let furthermore
$$b_{i}=d_{i+1}\cdots d_{n}b_{n}\cdot a_{i,n-i}^{-1} a_{i,n-i-1}^{-1} \cdots a_{i, 1}^{-1}.$$
Finally, let
$$y_{i}=(a_{i,1}, a_{i,2},...,a_{i,n-i}, b_{i}).$$
We can then set $\kappa^{-1}(a_{0},...,a_{n-1},b_{n})=(y_{0},...,y_{n})$.

The isomorphism $\kappa$ lifts the well-known isomorphism $\tot (\B G)\cong \overline \W G$.
\end{proof} 

\begin{rmk}\label{rmk:simpl-rth}  Let $\lambda_{hoch}^{(r)}=\kappa\tot (\widetilde \lambda _{simp}^{(r)})\kappa^{-1}:\H G\to \H G$, where $\kappa$ is the isomorphism constructed in the proof of Theorem \ref{thm:tot-zg}.  A direct calculation shows that
$$\lambda_{hoch}^{(r)}(a_{0},...,a_{n-1},b)=(a_{0},...,a_{n-1},b^r),$$
 for all $a_{0}\in G_{0},...,a_{n-1}\in G_{n-1}, b\in G_{n}$.  In particular,
$$\xymatrix{G \ar [d] ^{\lambda^{(r)} }\cof &\H G\ar [d]^{\lambda_{hoch}^{(r)} }\fib &\overline\W G\ar @{=}[d]\\
	G\cof &\H G\fib &\overline \W G}$$
commutes, where $\lambda^{(r)}$ is the usual $r^{\text{th}}$-power map on a simplicial group.
\end{rmk}

We can further simplify the simplicial model of the power map, when the simplicial group we consider is of the form $\G K$. 

\begin{defn} Let $K$ be a reduced simplicial set. The \emph{simplicial coHochschild construction} on $K$, denoted $\coH K$, is the simplicial set $K\underset {\tau_{K}}\times {}^\gamma {\G K}$, where ${}^\gamma \G K$ denotes $\G K$ seen as a left $\G K$-space, with $\G K$ acting on itself by conjugation, and $\tau_{K}$ denotes the universal twisting function $\tau _{K}:K\to \G K$ of Remark \ref{rmk:twisting-fcn}. 
\end{defn}

\begin{rmk}\label{rmk:rth-simpl-cohoch}  Let $r\in \mathbb N$, and let $K$ be a reduced simplicial set. It follows from the naturality of the twisted cartesian product construction that the map of  graded sets $\lambda_{cohoch}^{(r)}: \coH K\to \coH K$ given in degree $n$ by 
$$\lambda_{cohoch}^{(r)}(x, a)=(x, a^r)$$
is simplicial, since the $r^{\text{th}}$-power map on ${}^\gamma {\G K}$ is a morphism of left $\G K$-spaces.  Furthermore, the diagram
$$\xymatrix{\G K \ar [d] ^{\lambda^{(r)} }\cof &\coH K\ar [d]^{\lambda_{cohoch}^{(r)} }\fib &K\ar @{=}[d]\\
	\G K\cof &\coH K\fib &K}$$
commutes, where $\lambda^{(r)}$ is the usual $r^{\text{th}}$-power map on a simplicial group.
\end{rmk}

Recall the weak equivalence $\eta_{K}$ of Remark \ref{rmk:eta-equiv}.

\begin{thm}\label{thm:simpl-cohoch-to-hoch} If $K$ is a reduced simplicial set, then there is a commuting diagram of Kan fibrations
$$\xymatrix{\G K \ar @{=} [d] \cof &\coH K\ar [d]^{\eta_{K}\times \Id_{\G K}}\fib &K\ar [d]^{\eta_{K}}\\
	\G K\cof &\H \G K\fib &\overline \W \G K}$$
where are vertical arrows are weak equivalences.  Furthermore, 
$$\xymatrix {\coH K\ar [d]^{\eta_{K}\times \Id_{\G K}}\ar[rr]^{\lambda_{cohoch}^{(r)}}&&\coH K\ar [d]^{\eta_{K}\times \Id_{\G K}}\\ \H\G K \ar [rr]^{\lambda _{hoch}^{(r)}}&&\H\G K}$$
commutes, i.e., the equivalence $\coH K\xrightarrow\simeq \H \G K$ commutes with the natural $r^{\text{th}}$-power maps.
\end{thm}

\begin{proof}  Since both of the horizontal sequences in the first diagram are Kan fibrations, if the middle vertical map is indeed a simplicial map, then it is a weak equivalence.  It is a matter of simple calculation to show that $\eta_{K}\times \Id_{\G K}$ commutes with all faces and degeneracies and is therefore a simplicial map.  That the second diagram commutes is obvious after inspection of the definitions.
\end{proof}

\subsection{Chain models of the power map}

We can now show that, at least on the chain level, $\lambda_{cohoch}^{(r)}$ is an acceptable model of $\widetilde \lambda_{top}^{(r)}$.

\begin{thm} \label{thm:small-to-big}For any reduced simplicial set $K$, there is a natural commutative diagram of chain complexes 
\begin{equation}\label{eqn:small-to-big}
\xymatrix{
C_{*}\G K\;\ar@{>->}[r]^(0.4)\iota\ar [d]&C_*\coH K \ar @{->>}[r]^(0.6){\pi}\ar [d] &C_*K\ar [d]\\
S_{*}G\;\ar@{>->}[r]^(0.4)\iota &S_*(\op L|\B G|)\ar @{->>}[r]^(0.6){\pi} &S_{*}|\B G|,}
\end{equation}
where $G=|\G K|$.  Furthermore,  
\begin{equation}\label{eqn:rth-small-to-big}
\xymatrix {C_{*}\coH K\ar [d]^{\simeq}\ar[rr]^{C_{*}\lambda_{cohoch}^{(r)}}&&C_{*}\coH K\ar [d]^{\simeq}\\ S_*(\op L|\B G|)\ \ar [rr]^{S_{*}\widetilde \lambda_{top}^{(r)}}&& S_*(\op L|\B G|)}
\end{equation}
commutes up to homotopy.
\end{thm}

\begin{proof} The diagram (\ref{eqn:small-to-big}) is obtained by composing vertically in the following dia\-gram, where $\operatorname{diag}$ denotes the usual diagonal functor from bisimplicial sets to simplicial sets.

\begin{equation}\label{eqn:bigdiag}\xymatrix{C_{*}\G K\;\ar@{>->}[r]^(0.5)\iota\ar @{=}[d]&C_*\coH K \ar @{->>}[r]^(0.55){\pi}\ar [d]^\simeq &C_*K\ar [d]^\simeq\\
C_{*}\G K\;\ar@{>->}[r]^(0.5)\iota\ar [d]^\cong &C_*\H\G K\ar @{->>}[r]^(0.5){\pi}\ar [d]^\cong &C_*\overline\W\G K\ar [d]^\cong\\
\tot C_{**}(\C \G K)\;\ar@{>->}[r]^(0.5)\iota\ar [d]^\simeq &\tot C_{**}(\Z\G K)\ar @{->>}[r]^(0.5){\pi}\ar [d]^\simeq &\tot C_{**}\B\G K\ar [d]^\simeq\\
C_{*}(\operatorname{diag}\C \G K)\;\ar@{>->}[r]^(0.5)\iota\ar [d]^\simeq &C_{*}(\operatorname{diag}\Z\G K)\ar @{->>}[r]^(0.5){\pi}\ar [d]^\simeq &C_{*}\operatorname{diag}\B\G K\ar [d]^\simeq\\
S_{*}(|\operatorname{diag}\C \G K|)\;\ar@{>->}[r]^(0.5)\iota\ar [d]^\cong &S_{*}(|\operatorname{diag}\Z\G K|)\ar @{->>}[r]^(0.4){\pi}\ar [d]^\cong &S_{*}(|\operatorname{diag}\B\G K|)\ar [d]^\cong\\
S_{*}(|\C G|)\;\ar@{>->}[r]^(0.5)\iota\ar [d]^\simeq &S_{*}(|\Z G|)\ar @{->>}[r]^(0.5){\pi}\ar [d]^\simeq &S_{*}(|\B G|)\ar@{=} [d]\\
S_{*}(G)\;\ar@{>->}[r]^(0.5)\iota &S_{*}(\op L|\B G|)\ar @{->>}[r]^(0.5){\pi}&S_{*}(|\B G|)}
\end{equation}

The commutativity of the first row of squares follows from Theorem \ref{thm:simpl-cohoch-to-hoch}, while that of the second row follows by a slight generalization of the isomorphism (25) in \cite{cegarra-remedios}.  The third row of vertical arrows are Eilenberg-Zilber equivalences, so that the third row commutes by naturality of the Eilenberg-Zilber maps. Note that the Tot in the third line denotes the usual total complex of a bicomplex, rather than Artin-Mazur totalization.  Naturality again implies the commutativity of the fourth row, since the vertical maps here are the natural equivalences $C_{*}(-)\xrightarrow \sim S_{*}|-|$.

To understand the isomorphisms in the fifth row, recall (cf., e.g., section 1 in \cite{cegarra-remedios}) that if $K_{\bullet \bullet}$ is a bisimplicial set, the space obtained by first realizing levelwise and then taking the Segal realization of the resulting simplicial space is naturally homeomorphic to the realization of $\operatorname{diag} K_{\bullet \bullet}$.

Finally, the last row of the diagram commutes since 
$$\xymatrix{
|\C G|\ar[r]^(0.5)\iota\ar [d]^\cong &|\Z G|\ar @{->>}[r]^(0.5){\pi}\ar [d]^\simeq_{h} &|\B G|\ar@{=} [d]\\
G\ar[r]^(0.5)\iota &\op L|\B G|\ar @{->>}[r]^(0.5){\pi}&|\B G|}$$
commutes, by construction of $h$.

To verify that diagram (\ref{eqn:rth-small-to-big}) commutes, apply the second part of Theorem \ref{thm:simpl-cohoch-to-hoch} in the first row of diagram (\ref{eqn:bigdiag}), naturality in the  four succeeding rows and Corollary \ref{cor:simpl-rth} in the last row. 
\end{proof}

It remains to link $\cohoch (C_{*}K)$ to the other chain models for the free loop space and then to establish compatibility with the power maps.  The first of these goals was already attained  in \cite{hps-cohoch} (Theorems  3.15 and 3.16), where  $\op L K$ was used to denote the simplicial coHochschild construction $\coH K$.  

Recall the Szczarba twisting cochain $t_{K}$ and its associated chain algebra morphism $\alpha_{K}$ (Example \ref{ex:szczarba}).

\begin{thm}\label{thm:cohoch-ez}\cite{hps-cohoch} For any  reduced simplicial set $K$, there is a commutative diagram of chain complexes
\begin{equation}\label{eqn:cohoch-ez}
\xymatrix{
\Om C_{*}K\;\ar [d]_{\alpha_{K}}\ar@{>->}[r]^(0.4)\iota &\cohoch(C_{*}K)\ar [d]_{\theta_{K} } \ar @{->>}  [r]^(0.6){\pi_{\cohoch}} &C_{*}K\ar @{=}[d]\\
C_{*}\G K\;\ar@{>->}[r]^(0.4)\iota &C_*\coH K\ar @{->>}[r]^(0.6){\pi_{\op T}} &C_{*}K,}
\end{equation}
which is natural in $K$.  Furthermore, the vertical arrows are quasi-isomorphisms if $K$ is 1-reduced.

Here,  $\pi_{\cohoch}= \Id_{C_{*}K}\otimes \ve$, where $\ve:\Om C_{*}K\to R$ is the natural coaugmentation, and $\pi_{\op T}=C_{*}p_{K}$, where $p_{K}:\coH K \to K$ is the Kan fibration of Theorem \ref {thm:simpl-cohoch-to-hoch}.
\end{thm}

\begin{rmk} We correct here a mistake in the statement of Theorem 3.16 in \cite {hps-cohoch}, where it was asserted that the vertical maps above were quasi-isomorphisms for all reduced simplicial sets.  The proof of the existence of the map $\theta_{K}$ is independent of whether or not it is a quasi-isomorphism.

We remark moreover that, if we restrict to simplicial double suspensions, i.e., $K=\S M$ for some simplicial set $M$, then the acyclic models proof of Theorem \ref{thm:cohoch-ez} can be modified to resemble that of Theorem \ref{thm:bigthm} in Appendix \ref{app:acyclic}, so that we obtain  that $\theta_{\S M}$ is natural in $M$.
\end{rmk}

The proof that the equivalence $\theta_{K}:\cohoch(C_{*}K)\to C_{*}\coH K$ is compatible, at least up to homotopy, with the power maps defined on the source and target requires a bit of work, which we carry out in the next section.

\subsection{Compatibility with power maps}

Theorem \ref{thm:topo-alg} is an immediate consequence of Theorems \ref{thm:small-to-big} and \ref{thm:cohoch-ez}, together with the following theorem.

\begin{thm}\label{thm:compat}   If $K=\S M$, where $M$ is any simplicial set, then
$$\xymatrix{\cohoch(C_{*}K) \ar[d]_{\widetilde \lambda _{r}}\ar [r]^(0.55){\theta_{K}}_(0.55){\simeq}&C_{*}\coH K\ar[d] ^{C_{*}\lambda_{cohoch}^{(r)}}\\
\cohoch (C_{*}K) \ar [r]^(0.55){\theta_{K}}_(0.55){\simeq}&C_{*}\coH K}$$
commutes up to a chain homotopy that is natural in $M$, where $\theta_{K}$ is the natural equivalence of Theorem \ref{thm:cohoch-ez}.
\end{thm}

We actually establish a more general result, which allows us to apply acyclic models methods in the proof. Consider a simplicial map $f:K\to K'$ between reduced simplicial sets and the induced twisting cochain
$$C_{*}K\xrightarrow {C_{*}f} C_{*}K'\xrightarrow {t_{\Om}} \Om C_{*}K'.$$
Let $t_{\Om,f}$ denote this composite.

If the natural comultiplication on either $\Om C_{*}K$ or $\Om C_{*}K'$ (cf.~Example \ref{ex:hpst}) is cocommutative, then Theorem  \ref{thm:exists-nthpower} implies that 
for any  $r\in \mathbb N$, there is an endomorphism of chain complexes 
 $$\widetilde\lambda_{r}:\hoch( t_{\Om, f})\to \hoch( t_{\Om ,f}),$$
 natural with respect to morphisms in the category $\mathrm {Arr}(\cat{sSet_{0}})$ of arrows in $\cat {sSet}_{0}$, the category of reduced simplicial sets, 
 such that 
 $$\xymatrix{\Om C_{*}K' \ar[d]_{\lambda _{r}}\cof & \hoch (t_{\Om ,f})\ar[d]_{\widetilde\lambda_{r}}\fib & C_{*}K\ar@{=}[d]\\
 		    \Om C_{*}K'\cof &\hoch (t_{\Om,f}) \fib& C_{*}K}$$
commutes, where $\lambda_{r}$ denotes the usual \rth-power map on $\Om C_{*}K'$.

On the other hand, from the simplicial map $f:K\to K'$ we can construct a twisting function
$$K\xrightarrow f K'\xrightarrow {\tau _{K'}} \G K'$$
and thus a twisted cartesian product
$$\widehat\H (\tau_{K'}f):=K\underset {\tau_{K'}f} \times {}^\gamma\G K',$$
where ${}^\gamma\G K'$ denotes $\G K'$ endowed with the  left $\G K'$-action given by conjugation.
As in Remark \ref{rmk:rth-simpl-cohoch},  for any  $r\in \mathbb N$  the naturality of the twisted cartesian product construction implies that the map of  graded sets $\lambda_{cohoch}^{(r)}:\widehat\H (\tau_{K'}f)\to \widehat\H (\tau_{K'}f)$ given in degree $n$ by 
$$\lambda_{cohoch}^{(r)}(x, a)=(x, a^r)$$
is simplicial, since the $r^{\text{th}}$-power map on ${}^\gamma {\G K'}$ is a morphism of left $\G K'$-spaces.  Furthermore, the diagram
\begin{equation}\label{eqn:restrict-power}\xymatrix{\G K' \ar [d] ^{\lambda^{(r)} }\cof &\widehat\H (\tau_{K'}f)\ar [d]^{\lambda_{cohoch}^{(r)} }\fib &K\ar @{=}[d]\\
	\G K'\cof &\widehat\H (\tau_{K'}f)\fib &K}
\end{equation}
commutes, where $\lambda^{(r)}$ is the usual $r^{\text{th}}$-power map on a simplicial group.

Theorems 3.15 and 3.16 in \cite{hps-cohoch} imply that $\hoch(t_{\Om, f})$ is a chain model of $\widehat\H (\tau_{K'}f)$, in the following sense.

\begin{thm}\label{thm:cohoch-ez-gen}\cite{hps-cohoch} For any simplicial map  $f:K\to K'$ between reduced simplicial sets, there is a commutative diagram of chain complexes
\begin{equation}\label{eqn:cohoch-ez}
\xymatrix{
\Om C_{*}K'\;\ar [d]_{\alpha_{K'}}\ar @{>->} [r]^(0.4)\iota &\hoch(t_{\Om, f})\ar [d]_{\theta_{f} } \ar @{->>}  [r]^(0.6){\pi_{\cohoch}} &C_{*}K\ar @{=}[d]\\
C_{*}\G K'\;\ar@{>->}[r]^(0.4)\iota &C_*\widehat\H(\tau _{K'}f)\ar @{->>}[r]^(0.6){\pi_{\op T}} &C_{*}K,}
\end{equation}
which is natural in $f$.  Furthermore, the vertical arrows are quasi-isomorphisms if $K'$ is 1-reduced. 

Here, $\pi_{\cohoch}=\Id_{C_{*}K}\otimes \ve$, where $\ve:\Om C_{*}K'\to R$ is the natural augmentation, and $\pi_{\op T}=C_{*}p_{f}$, where $p_{f}:\widehat\H(\tau _{K'}f)\to K$ is the natural projection.
\end{thm}

\begin{rmk}  Modifying the acyclic models argument in \cite{hps-cohoch} along the lines of Appendix \ref{app:acyc}, we can show that if $f=\S g$ for some simplicial map $g:M\to M'$, then $\theta_{f}$ can be chosen naturally in $g$.
\end{rmk}

Our goal now is to prove the following, generalized compatibility result.  Recall from Corollary \ref{cor:nth-cohoch} that for any Hirsch coalgebra $(C,\psi)$, the algebraic power map $\widetilde \lambda_{r}:\cohoch (C)\to \cohoch (C) $ is natural with respect to morphisms in $\cat {Hirsch}_{R}$.  Since $C_{*}f:(C_{*}K,\psi_{K})\to (C_{*}K', \psi_{K'})$ is a morphism in $\cat {Hirsch}_{R}$ for any simplicial map $f:K\to K'$,  the power map $\widetilde \lambda_{r}:\cohoch (C_{*}K)\to \cohoch (C_{*}K) $ is natural with respect to simplicial maps, for all reduced simplicial sets $K$ such that $(C_{*}K, \psi_{K})$ is a balanced Hirsch coalgebra.

\begin{thm}\label{thm:compat-gen}   Let $g:M \to M'$ be a  simplicial map. If $f=\S  g$ and $K'=\S M'$,  then
$$\xymatrix{\hoch (t_{\Om ,f})\ar[d]_{\widetilde \lambda _{r}}\ar [r]^(0.5){\theta_{f}}&C_{*}\widehat\H (\tau_{K'}f)\ar[d] ^{C_{*}\lambda_{cohoch}^{(r)}}\\
\hoch (t_{\Om, f}) \ar [r]^(0.5){\theta_{f}}&C_{*}\widehat\H (\tau_{K'}f)}$$
commutes up to a chain homotopy that is natural in $g$.
\end{thm}

Note that Theorem \ref{thm:compat} follows immediately from Theorem \ref{thm:compat-gen} applied to $g=\Id_{M}$.
Theorem \ref{thm:compat-gen} in turn is a consequence of the following technical result, which we prove in Appendix \ref{app:acyclic}.  Let $\cat S$ denote the category of morphisms in $\cat {sSet}$.

\begin{thm}\label{thm:bigthm} 
 Let $\widehat \xi$ be a natural transformation from the functor
$$\P:\cat S\to \cat {Ch}_{R}: (g:M\to M') \mapsto \hoch(t_{\Om,\S g})$$
to the functor
$$\Q:\cat S\to \cat {Ch}_{R}: (g:M\to M') \mapsto C_{*}\widehat\H (\tau_{\S M'}\S g)$$
extending a natural transformation 
$$\xi:\Om C_{*}\S \to C_{*}\G\S : \cat {sSet} \to \cat {Ch}_{R},$$ i.e., 
\begin{equation*}\label{eqn:htpyzero}
\xymatrix{\Om C_{*}\S M'\ar[d]_{\xi (M')}\cof^(0.55)\iota &\P(g)\ar[d]^{\widehat \xi(g)}\\ 
C_{*}\G \S M'\cof ^(0.55)\iota&\Q(g)}
\end{equation*}
commutes for all $g:M\to M'$.

Suppose that for all $g:M\to M'$, $\xi(M')$ is naturally chain homotopic to the zero map, and the composite
$$\P(g) \xrightarrow {\widehat \xi (g)} \Q (g) \xrightarrow {C_{*}p_{g}} C_{*}\S M$$
is zero, where $p_{g}:\widehat\H (\tau_{\S M'}\S g)\to \S M$ is the projection. If the corestriction $$\widehat \xi (\Id_{\emptyset}): \P (\Id_{\emptyset})\to \ker C_{*}p_{\Id_{\emptyset}}$$ 
is chain homotopic to the zero map, then $\widehat \xi (g)$ is naturally chain homotopic to the zero map for all $g$.
\end{thm}

The application of this theorem to the case of power maps proceeds as follows.  Let $g: M\to M'$ be any simplicial map. Recall first that Theorem \ref{thm:exists-nthpower} implies that  $\hoch(t_{\Om,\S g})$ admits a power map $\widetilde \lambda_{r}$,  since the chain Hopf algebra $(\Om C_{*}\S M', \psi_{\S M'})$  is primitively generated and thus cocommutative.   Moreover it follows from the remark following Theorem \ref{thm:cohoch-ez-gen}   that there is a commutative diagram of chain complexes
\begin{equation}\label{eqn:cohoch-ez}
\xymatrix{
\Om C_{*}\S M'\;\ar [d]_{\alpha_{\S M'}}\ar @{>->} [r]^(0.5)\iota &\hoch(t_{\Om, \S g})\ar [d]_{\theta (g) } \ar @{->>}  [r]^(0.6){\pi_{\cohoch}} &C_{*}M\ar @{=}[d]\\
C_{*}\G \S M'\;\ar@{>->}[r]^(0.5)\iota &C_*\widehat\H(\tau _{\S M'}\S g)\ar @{->>}[r]^(0.6){C_{*}p_{g}} &C_{*}M,}
\end{equation}
where $\theta(g)$ is natural in $g$.

Recall from diagram (\ref{eqn:restrict-power}) that the power map $\lambda^{(r)}_{cohoch}$ on $\widehat\H(\tau _{\S M'}\S g)$ restricts to the usual power map $\lambda ^{(r)}: \G \S M'\to \G\S M'$ on the fiber.   To apply Theorem \ref{thm:bigthm}, we need control over the relation between $C_{*}\lambda^{(r)}$ and $\lambda _{r}:\Om C_{*}\S M'\to \Om C_{*}\S M'$, as provided by the following lemma. 

\begin{lem} For any  simplicial set $M$, the diagram
$$\xymatrix {\Om C_{*}\S M\ar [d]_{\alpha_{\S M}}\ar [r]^{\lambda _{r}}&\Om C_{*}\S M\ar [d]^{\alpha_{\S M}}\\
C_{*}\G \S M\ar [r]^{C_{*}\lambda ^{(r)}}& C_{*}\G \S M}$$
commutes up to natural chain homotopy.
\end{lem}

\begin{notn}  Henceforth, we denote the natural chain homotopy of the lemma above by $\Theta(M)$.
\end{notn}

\begin {proof} Let $K=\S M$. Consider the diagram of chain maps
$$\xymatrix {\Om C_{*}K\ar [d]_{\alpha_{K}}\ar [rr]^{\psi_{K}^{(r)}}&& (\Om C_{*}K)^{\otimes r}\ar [rr] ^{\mu^{(r)}}\ar[d]_{\alpha_{K}^{\otimes r}}&&\Om C_{*}K\ar[d]^{\alpha_{K}}\\
C_{*}\G K \ar [d]_{=}\ar [rr]^{\Delta_{C_{*}\G K}^{(r)}}&& (C_{*}\G K)^{\otimes r}\ar [d]_{\nabla^{(r)}}\ar [rr]^{\mu^{(r)}}&& C_{*}\G K\ar [d]^=\\
C_{*}\G K \ar [rr]^{C_{*}\Delta_{\G K}^{(r)}}&& C_{*}(\G K^{\times r})\ar [rr]^{C_{*}m ^{(r)}}&& C_{*}\G K,}$$
where $\nabla ^{(r)}$ is defined by iterating the natural Eilenberg-Zilber equivalence.
The upper lefthand square commutes strictly by Theorem 4.11 in \cite{hps2}.  The upper righthand square commutes strictly, since $\alpha_{K}$ is an algebra morphism.  The lower lefthand square commutes up to  natural chain homotopy because $\Delta_{C_{*}\G K}= f \circ C_{*}\Delta_{\G K}$, where $f$ is the natural Alexander-Whitney equivalence, and $\nabla f$ is naturally chain homotopic to the identity.   Finally, the lower righthand square commutes strictly, by definition of the multiplication on $C_{*}\G K$.
 \end{proof}

\begin{proof}[Proof of Theorem  \ref{thm:compat-gen}] Apply Theorem \ref{thm:bigthm} to 
$$\xi(M')=\alpha_{\S M'}\circ \lambda _{r}-C_{*}\lambda ^{(r)}\circ \alpha_{\S M'}: \Om C_{*}\S M'\to C_{*}\G \S M'$$
and its extension
$$\widehat \xi (g)=\theta(g) \circ \widetilde \lambda _{r}-C_{*}\lambda_{cohoch}^{(r)}\circ \theta(g): \hoch(t_{\Om,g})\to C_*\widehat\H(\tau _{\S M'}g).$$
Note that $C_{*}p_{g}\circ \widehat \xi (g)=0$ for all $g$,  since 
$$\xymatrix {\hoch(t_{\Om,g})\ar [d]_{\widetilde\lambda_{r}}\ar [r]^{\pi_{\cohoch}}&C_{*}M\ar [d]^=\\
\hoch(t_{\Om,g})\ar [r]^{\pi_{\cohoch}}&C_{*}M}$$
and 
$$\xymatrix {C_*\widehat\H(\tau _{\S M'}g)\ar [d]_{C_{*}\lambda_{cohoch}^{(r)}}\ar [r]^(0.6){C_{*}p_{g}}&C_{*}M\ar [d]^=\\
C_*\widehat\H(\tau _{\S M'}g)\ar [r]^(0.6){C_{*}p_{g}}&C_{*}M}$$
commute.  

The last hypothesis of Theorem \ref{thm:bigthm} that we need to check concerns $\widehat \xi(\Id _{\emptyset})$, which much be chain homotopic to 0, via a chain homotopy whose image is contained in $\ker C_{*}p_{\Id_{\emptyset}}$.  We begin by observing that the proof of Theorem  \ref{thm:cohoch-ez-gen} for simplicial double suspensions, i.e., the proof of the existence of $\theta (g)$ in general, is a proof by acyclic models, where the model that initiates the construction is the identity $\emptyset \to \emptyset$.  We can therefore choose $\theta (\Id _{\emptyset})$  to be any chain map that makes the diagram (\ref{eqn:cohoch-ez}) commute for $g=\Id_{\emptyset}$.  

We now show that one possible choice of $\theta (\Id_{\emptyset})$ is $\nabla(\Id_{C_{*}\S \emptyset}\otimes \alpha_{\S \emptyset})$. With this choice of $\theta(\Id_{\emptyset})$, it is immediate that $\widehat\xi(\Id_{\emptyset})$ is chain homotopic to 0, via the chain homotopy $\nabla(\Id_{C_{*}\S\Id_{\emptyset}}\otimes \Theta(\S{\emptyset}))$, whose image is obviously contained in $\ker C_{*}p_{\Id_{\emptyset}}$.

Since $C_{1}\S \emptyset = R\cdot (1,c_{0})$, and $C_{>1}\S \emptyset =0$, a simple calculation shows that 
$$\P (Id_{\emptyset})\cong C_{*}\S \emptyset \otimes \Om C_{*}\S \emptyset,$$ 
the usual, untwisted tensor product of complexes. Furthermore, $\P (Id_{\emptyset})$ is concentrated in degrees $0$ and $1$.

Let $\mathsf F$ denote the free group functor. Straightforward simplicial calculations show that 
$$(\S\Id_{\emptyset})_{n}=\{s_{0}^na_{0}\}\cup \{s_{i_{1}}s_{i_{2}}\cdots s_{i_{n-1}}(c,1)\mid s_{i_{1}}>s_{i_{2}}>\cdots>s_{i_{n-1}}\}$$
and
$$(\G\S\Id_{\emptyset})_{n}= \mathsf F\big(\,\overline {s_{n}s_{n-1}\cdots s_{0}(c,1)}\,\big),$$
from which is follows that the $0^{\text{th}}$-face map of $\widehat\H(\tau _{\S \emptyset}\Id_{\emptyset})$ is untwisted, i.e., 
$$\widehat\H(\tau _{\S \emptyset}\Id_{\emptyset})\cong \S \Id_{\emptyset}\times \G \S \emptyset.$$
Moreover, since $\overline {s_{n}s_{n-1}\cdots s_{0}(c,1)}\in \im s_{n-1}\cap \cdots \cap \im s_{0}$, all simplices of degree greater than $1$ are degenerate, and so 
$$\Q (\Id_{\emptyset})=C_{*}\big(\widehat\H(\tau _{\S \emptyset}\Id_{\emptyset})\big)=C_{*}(\S\Id_{\emptyset}\times \G\S\Id_{\emptyset}) $$
is also concentrated in degrees 0 and 1.  It follows that we can set 
$$\theta (\Id_{\emptyset})=\nabla(\Id_{C_{*}\S\Id_{\emptyset}}\otimes \alpha_{\S{\emptyset}}).$$
\end{proof}

\appendix


\section{The proof of Theorem \ref{thm:bar-hopf}}\label{sec:bar-hopf}

We begin by recalling those elements of homological perturbation
theory that we need in order to prove that applying  the bar construction to  a chain Hopf algebra gives rise to an Alexander-Whitney coalgebra.

\begin{defn}Suppose that $\n :(X, \del ) \rightarrow (Y,d)$ and
$f:(Y,d)\rightarrow (X,\del )$ are morphisms in $\cat {Ch}_{R}$.  If $f\n = \Id_{X}$  and there exists a
chain homotopy
$h : (Y,d)\rightarrow (Y,d)$ such that
\begin{enumerate}
\item $dh +h d =\n f -\Id_{Y}$,
\item $h \n =0$,
\item $fh =0$, and
\item $h \sp 2=0$,
\end{enumerate}
then $(X,d) \sdr{\n}f (Y,d)\circlearrowleft h$ is a \emph{strong
deformation retract (SDR) of chain complexes.}

If, moreover, $(Y,d, \Delta _{Y})$ and
$(X,d, \Delta _{X})$ are chain coalgebras and $\n$ is  a morphism of
coalgebras, then the SDR $(X,d) \sdr{\n}f (Y,d)\circlearrowleft h$ is called \emph{Eilenberg-Zilber data} \cite {gugenheim-munkholm}.
\end{defn}

\begin{rmk} If $(X,d) \sdr{\n}f (Y,d)\circlearrowleft h$ is Eilenberg-Zilber data, then 
$$(d\otimes \Id_{X}+\Id_{X}\otimes d)\bigl((f\otimes f)\Delta _{Y} h\bigr)+\bigl((f\otimes
f)\Delta_{Y}h\bigr)
d=\Delta_{X} f-(f\otimes f)\Delta_{Y},$$
i.e., $f$ is a map
of coalgebras up to chain homotopy. In fact, as stated precisely in the next theorem (due to Gugenheim and Munkholm and slightly strengthened in section 2.3 of \cite{hps2}), $f$ is a DCSH map, under reasonable local finiteness conditions.
\end{rmk}

Recall that if $V$ is a non-negatively graded $R$-module with $V_{0}=R$, then $\overline V$ denotes $V_{>0}$.

\begin{thm}\label{thm:g-m} \cite{gugenheim-munkholm, hps2} Let $(X,d) \sdr{\n}f
(Y,d)\circlearrowleft h$ be Eilenberg-Zilber data such that $X$  and $Y$ are connected. Let $\overline \Delta_{Y}:\overline Y \to \overline Y^{\otimes 2}$ denote the reduced comultiplication on $Y$. Let $F_0=0$, and let $F_{1}$ be the composite
$$
\overline Y\xrightarrow f \overline X\xrightarrow{s^{-1}}s^{-1}\overline X.
$$
For $k\geq 2$, let  
$$
F_{k}=-\sum _{i+j=k}(F_{i}\otimes F_{j})\overline\Delta _{Y} h:\overline Y\rightarrow T^k(s^{-1}\overline X).
$$
If for all $y\in Y$, there exists $N(y)\in \mathbb N$ such that $F_{k}(y)=0$ for all $k>N(y)$, then
$$F=\prod _{k\geq 1}F_{k}=\bigoplus _{k\geq 1}F_{k}: Y \to \Om X$$
is a twisting cochain.
 In particular, $f:Y\to X$ is a DCSH map, and $\alpha_{F}:\Om Y\to \Om X$ realizes its strong homotopy structure.
 \end{thm}
 
\begin{rmk} \label{rmk:Fk} Given Eilenberg-Zilber data $(X,d) \sdr{\n}f (Y,d)\circlearrowleft h$, there is a closed formula for each of the $F_{k}$'s above.  For any $k\geq 2$, let 
$$h_{k}=\sum _{0\leq i\leq k-2} \Id_{\overline Y}^{\otimes i}\otimes \overline\Delta_{Y}h\otimes \Id_{\overline Y}^{\otimes k-i-2}:\overline Y^{\otimes k-1}\to \overline Y^{\otimes k}$$
and let
\begin{equation}\label{eqn:Hk}
H_{k}=h_{k}\circ h_{k-1}\circ \cdots \circ h_{2}:\overline Y\to \overline Y^{\otimes k}.
\end{equation}
Then 
$$F_{k}=\sn {k+1} (\si f)^{\otimes k} \circ H_{k}.$$
\end{rmk}

We prove Theorem \ref{thm:bar-hopf} in this section by applying Theorem \ref{thm:g-m} to  appropriately chosen Eilenberg-Zilber data.  We now set up the desired SDR. 

 In the development below, we use  the following helpful notation for simplicial expressions. 
 
 \begin{notn} If $J$ is any  set of non-negative integers
$j_1<j_2<\cdots <j_r$, let  
$$s_J=s_{j_r}\cdots
s_{j_1},$$ 
and let $| J|=r$.   

 For non-negative integers $m\leq n$, let $[m,n]=\{j\in \mathbb Z\mid m\leq j\leq n\}$. Let $\mathbf \Delta$ denote the category with objects
$$\ob \mathbf  \Delta=\{[0,n]\mid n\geq 0\}$$
and 
$$\mathbf  \Delta \bigl ([0,m], [0,n]\bigr)=\{ f:[0,m]\to [0,n]\;| \;f\text{ order-preserving set map}\}.$$
Viewing a simplicial $R$-module $M_{\bullet}$ as a contravariant functor from $\mathbf \Delta $ to the category of $R$-modules, given $x\in M_n:=M([0,n])$ and $0\leq a_1<a_2<\cdots <a_m\leq n$, let
$$x_{a_1...a_m}:=M(\mathbf a)(x)\in M_m$$
where $\mathbf a:[0,m]\to [0,n]:j\mapsto a_j$.  Note that in particular
$$x_{0...r}=d_{r+1}\cdots d_{n}x,$$
while for all $m<r$,
$$x_{0...mr...n}=d_{m+1}\cdots d_{r-1}x.$$
\end{notn}
 
 \begin{ex}\label{ex:eilenberg-maclane} Let $\mathcal A$ denote the usual functor from simplicial $R$-modules to $\cat {Ch}_{R}$, i.e., for any simplicial $R$-module $M_{\bullet}$,  the graded $R$-module underlying $\mathcal A(M_{\bullet})$ is $\{M_{n}\}_{n\geq 0}$, and the differential in degree $n$ is given by the alternating sum of the face maps from $M_{n}$ to $M_{n-1}$.  Let $\mathcal A_{N}$ denote its normalized variant. 

In Theorem 2.1a) of \cite{eilenberg-maclane} Eilenberg and Mac Lane gave explicit formulas for a natural SDR of chain complexes
\begin{equation}\label{eqn:em-sdr}
\mathcal A_{N}(M_{\bullet})\otimes \mathcal A_{N}(M'_{\bullet})\sdr \nabla f \mathcal A_{N}(M_{\bullet}\boxtimes M'_{\bullet})\circlearrowleft h,
\end{equation}
where $\boxtimes$ denotes the levelwise tensor product of simplicial $R$-modules.  In particular, if $x\in M_{m}$ and $x'\in M'_{n}$, then 
\begin{equation}\label{eqn:f}
f(x\boxtimes y)=\sum _{0\leq \ell \leq n} x_{0...\ell}\otimes y_{\ell...n}
\end{equation}
and
\begin{equation}\label{eqn:nabla}
\n (x\otimes x')=\sum _{0\leq \ell\leq n}\sum _{A\cup B=[0,n-1]\atop |A|=n-\ell, |B|=\ell}\pm s_{A}x\boxtimes s_{B}x',
\end{equation}
where the sign of a summand is the sign of the shuffle permutation corresponding to the pair $(A,B)$.

If $R[K]$ denotes the free simplicial $R$-module generated by a simplicial set $K$, then $C_{*}K\otimes R\cong\mathcal A_{N}\big(R[K]\big)$. It follows that, when applied to $M_{\bullet}=R[K]$ and $M'_{\bullet}=R[K']$, for simplicial sets $K$ and $K'$, Eilenberg and Mac Lane's strong deformation retract becomes the usual Eilenberg-Zilber/Alexander-Whitney equivalence
$$C_{*}K\otimes C_{*}K'\sdr \nabla f  C_{*}(K\times K')\circlearrowleft h,$$
which is in fact Eilenberg-Zilber data.  In the case $K=K'$, these Eilenberg-Zilber data give rise to the Alexander-Whitney coalgebra structure on $C_{*}K$ \cite{hpst}.

In order to prove Theorem \ref{thm:bar-hopf}, we consider another special case of the Eilenberg-Mac Lane SDR. Recall that if $A$ is an augmented chain algebra, then $\Bar A$ is the normalized chain complex associated to the simplicial chain algebra $\Bar_{\bullet }A$, where $\Bar _{n}A=A^{\otimes n}$.  The degeneracy maps are given in terms of the unit map $R\to A$ by
$$s_{i}:A^{\otimes n}\to A^{\otimes n+1}: a_{1}\otimes \cdots \otimes a_{n}\mapsto a_{1}\otimes \cdots \otimes a_{i}\otimes 1\otimes a_{i+1}\otimes\cdots \otimes a_{n},$$
while the face maps are given in terms of the multiplication or the augmentation $\ve$ by
$$d_{i}:A^{\otimes n}\to A^{\otimes n-1}: a_{1}\otimes \cdots \otimes a_{n}\mapsto \begin{cases} 
\ve(a_{1})\cdot (a_{2}\otimes a_{3}\otimes\cdots \otimes a_{n})&:i=0\\
a_{1}\otimes \cdots \otimes a_{i}\cdot a_{i+1}\otimes\cdots \otimes a_{n}&:0<i<n\\
(a_{1}\otimes\cdots \otimes a_{n-1})\cdot \ve(a_{n})&:i=n.
\end{cases}$$

If $M_{\bullet}=\Bar_{\bullet}A$ and $M'_{\bullet}=\Bar_{\bullet}A'$, then Eilenberg and Mac Lane's strong deformation retract becomes
\begin{equation}\label{eqn:em-bar}
\Bar A\otimes \Bar A'\sdr \nabla f \Bar (A\otimes A')\circlearrowleft h,
\end{equation}
after identifying $\Bar (A\otimes A')$ with $ \mathcal A_{N}(M_{\bullet}\boxtimes M'_{\bullet})$ via a levelwise isomorphism
\begin{equation}\label{eqn:ident} (A\otimes A')^{\otimes n}\cong A^{\otimes n}\otimes A'^{\otimes n}.\end{equation}
The map $\nabla$
 is exactly the equivalence defined in Example \ref{ex:milgram} via the twisting cochain $t_{\Bar}*t_{\Bar}$.   In particular, $\nabla$ is a map of coalgebras, which implies that (\ref{eqn:em-bar}) is Eilenberg-Zilber data.
 
 Note that equation (\ref{eqn:nabla}) implies that  $\n(sa_{1}|\cdots |sa_{m}\otimes sa'_{1}|\cdots |sa'_{n})$ is equal to the signed sum of all possible $(m,n)$-shuffles of 
\begin{equation}\label{eqn:nabla-bar}
s(a_{1}\otimes 1)|\cdots |s(a_{m}\otimes 1)|s(1\otimes a'_{1})|\cdots |s(1\otimes a'_{n}).
\end{equation}
Moreover, equation (\ref{eqn:f}) implies that 
\begin{equation}\label{eqn:f-bar}
\begin{split}
f\big(& s(a_{1}\otimes a'_{1})|\cdots |s(a_{n}\otimes a'_{n})\big)\\
&=\sum_{0\leq \ell\leq n} \ve (a_{\ell + 1})\cdots \ve (a_{n})\ve(a'_{1})\cdots \ve(a'_{\ell }) \cdot sa_{1}|\cdots|sa_{\ell}\otimes sa'_{\ell+1}|\cdots |sa'_{n}
\end{split}
\end{equation}
Equation (\ref{eqn:f-bar}) implies that 
$$f\Big(\Bar_{n} (A\otimes A')\Big)\subset \bigoplus_{n'+n''=n} \Bar_{n'} A\otimes \Bar _{n''}A' .$$
 \end{ex}
 
 To prove Theorem \ref{thm:bar-hopf}, we apply Theorem \ref{thm:g-m} to the bar construction SDR of Eilenberg and Mac Lane (\ref{eqn:em-bar}).  We must therefore prove local finiteness of the associated $F_{k}$'s, which follows from a technical result proved in \cite {hpst} (Lemma 5.3), expressed below in terms of simplicial $R$-modules instead of simplicial sets.

\begin{lem} \cite{hpst}\label{lem:hpst-htpy}  Let $M_{\bullet}$ and $M'_{\bullet}$ be simplicial $R$-modules.
Let $m<r\leq n$ be non-negative integers, and let $A$ and $B$ be disjoint sets of non-negative integers such that $ A\cup B=[m+1,n]$ and $|B|=r-m$.

Let $h^{A,B}: (M\boxtimes M')_{n}\to (M\boxtimes M')_{n+1}$ be the $R$-linear map given by
$$h^{A,B}(x\boxtimes x')=s_{A\cup\{m\}}\,x_{0\ldots r}\boxtimes s_B\,x'_{0\ldots mr\ldots n}.$$
for all $x\in M_{n}$ and $x'\in M'_{n}$.
Then the Eilenberg--Mac Lane homotopy in level $n$ 
$$h:\mathcal A_{n}(M_{\bullet}\boxtimes M'_{\bullet})=M_{n}\otimes M'_{n}\to M_{n+1}\otimes M'_{n+1}=\mathcal A_{n+1}(M_{\bullet}\boxtimes M'_{\bullet})$$ 
is given by
$$h(x\boxtimes x') =
\sum_{ m<r,\; A\cup B=[m+1,n] \atop
    |A|=n-r,\;|B|=r-m} 
\pm h^{A,B}(x\boxtimes x'),$$
where the sign is that of the shuffle permutation associated to the couple $(A,B)$.
\end{lem}

\begin{ex} We are particularly interested in the case where $M_{\bullet}=\Bar _{\bullet}A$ and $M'_{\bullet}=\Bar _{\bullet}A'$, and we apply the identification (\ref{eqn:ident}) above.   We compute here one term of $h\big( s(a_{1}\otimes a_{1}')|s(a_{2}\otimes a_{2}')|s(a_{3}\otimes a_{3}')\big)$, to give some indication of the form of this homotopy, before providing general formulas below.

 Observe that if $x=a_{1}\otimes \cdots \otimes a_{n}\in \Bar _{n}A$, then 
$$x_{0...r}=\ve (a_{r+1}\cdot \ldots \cdot a_{n}) \cdot (a_{1}\otimes \cdots \otimes a_{r})\in \Bar _{r}A,$$
while if $x'=a'_{1}\otimes \cdots \otimes a'_{n}\in \Bar _{n}A'$, then 
$$x'_{0...mr...n}=a_{1}'\otimes \cdots \otimes a_{m}'\otimes a_{m+1}'\cdot\ldots\cdot a_{r}'\otimes a_{r+1}'\otimes \cdots \otimes a_{n}'.$$

When $n=3$, $r=2$, $m=1$, $A=\{3\}$ and $B=\{2\}$,
$$h^{A,B}\big( s(a_{1}\otimes a_{1}')|s(a_{2}\otimes a_{2}')|s(a_{3}\otimes a_{3}')\big)=\ve(a_{3})\cdot\big( s(a_{1}\otimes a_{1}')|s(1\otimes a_{2}')|s(a_{2}\otimes 1)|s(1\otimes a_{3}')\big),$$
because
$$s^{A\cup\{1\}} (a_{1}\otimes a_{2}\otimes a_{3})_{012}=\ve (a_{3})\cdot s_{3}s_{1} (a_{1}\otimes a_{2})=\ve(a_{3})\cdot (a_{1}\otimes 1\otimes a_{2}\otimes 1)$$
and
$$s^B (a'_{1}\otimes a'_{2}\otimes a'_{3})_{0123}=s_{2}(a'_{1}\otimes a'_{2}\otimes a'_{3})=a'_{1}\otimes a'_{2}\otimes 1\otimes  a'_{3}.$$
\end{ex}

In the case of the bar construction SDR, we obtain the following general, explicit formulas, where we use the notational shortcut
$$v=v_{1}|\cdots |v_{m} \quad\text{ and }\quad w=w_{1}|\cdots |w_{n}\;\Longrightarrow\; v|w:= v_{1}|\cdots |v_{m}|w_{1}|\cdots |w_{n}.$$

\begin{cor}\label{cor:bar-htpy} Let $A,A'\in \ob \cat {Alg}_{R}$. If $M_{\bullet}=\Bar_{\bullet}A$ and $M'_{\bullet}=\Bar_{\bullet}A'$, then the Eilenberg-Mac Lane homotopy $h:\Bar_{*} (A\otimes A')\to \Bar _{*+1}(A\otimes A')$ satisfies the following equations.
\begin{enumerate}
\item $h\big(s(1\otimes a'_{1})| \cdots |s(1\otimes a'_{n})\big)=0$ for all $a'_{1},...,a'_{n}\in A$.
\medskip
\item If $|a_{r}|\not=0$, then
{\smaller\begin{equation*}\begin{split}
h&\big(s(a_{1}\otimes a'_{1})| \cdots |s(a_{r}\otimes a'_{r})|s(1\otimes a'_{r+1})|\cdots |s(1\otimes a'_{n})\big)\\
=&\sum_{0\leq m<r}\pm s(a_{1}\otimes a'_{1})| \cdots |s(a_{m}\otimes a'_{m})|s(1\otimes a'_{m+1}\cdots a'_{r})|\n(sa_{m+1}|\cdots |sa_{r}\otimes sa'_{r+1}|\cdots |sa'_{n})
\end{split}\end{equation*}}

\noindent for all $a_{1},...,a_{r-1}\in A$ and $a'_{1},...,a'_{n}\in A'$.
\end{enumerate}
\end{cor}

\begin{rmk} Note that the formulas above imply that 
$$h\big(s(a_{1}\otimes 1)|\cdots |s(a_{r}\otimes 1)|s(1\otimes a'_{r+1})| \cdots |s(1\otimes a'_{n})\big)=0,$$
for all $a_{1},...,a_{r}\in A$ and $a'_{r+1},...,a'_{n}\in A'$ and for all $r\geq 0$, since in this case $a_{m+1}'=\cdots =a_{r}'=1$ for all $m<r$.
\end{rmk}

\begin{proof} Let 
$$w=s(a_{1}\otimes a'_{1})|\cdots |s(a_{M}\otimes a'_{M})|s(1\otimes a'_{M+1})|\cdots | s(1\otimes a'_{n})\in \Bar _{n}(A\otimes A'),$$
where $|a_{M}|>0$.  Note that  Lemma \ref{lem:hpst-htpy}  implies that $h(w)\in \Bar _{n+1}(A\otimes A')$.

It is clear that 
$$r<M \;\Longrightarrow\; h^{A,B}(w)=0$$
for all $0\leq m< r$ and $A$ and $B$ disjoint sets of non-negative integers such that $ A\cup B=[m+1,n]$ and $|B|=r-m$, since $\ve (a_{M})=0$.

We can also show that
$$r>M\; \Longrightarrow\; h^{A,B}(w)=0$$
for all $0\leq m< r$ and $A$ and $B$ disjoint sets of non-negative integers such that $ A\cup B=[m+1,n]$ and $|B|=r-m$.  To establish this implication, we consider the following two cases.  Suppressing summation, write 
$$h^{A,B}(w)=s(b_{1}\otimes b'_{1})|\cdots |s(b_{n+1}\otimes b'_{n+1}).$$
\begin{enumerate}
\item If $r>m\geq M$, then $|b_{m+1}|=\cdots =|b_{n}|=0$, while the list $b'_{m+1},...,b'_{n}$ includes at least $r-m$ elements of degree 0.  There exists therefore $k\in [m+1,n]$ such that both $b_{k}$ and $b'_{k}$ are of degree zero and therefore $s(b_{k}\otimes b'_{k})$ is degenerate in $\Bar_{\bullet}(A\otimes A')$, i.e., $s(b_{k}\otimes b'_{k})=0$ in the normalized complex.
\item If $r>M>m$, then the list $b_{m+1},..., b_{n}$ includes at most $M-m$ elements of positive degree, i.e., at least $n-M$ elements of degree 0.  On the other hand, the list $b'_{m+1},..., b'_{n}$ includes at least $r-m$ elements of degree 0.  Since 
$$(r-m)+ (n-M) = n -(m+M-r)> n-(m+1),$$
there exists $k\in [m+1,n]$ such that both $b_{k}$ and $b'_{k}$ are of degree zero and therefore $s(b_{k}\otimes b'_{k})$ is degenerate in $\Bar_{\bullet}(A\otimes A')$, i.e., $s(b_{k}\otimes b'_{k})=0$ in the normalized complex.
\end{enumerate} 

We conclude that the only nonzero summands of $h(w)$ are those for which $r=M$, in which case the formula given in the Corollary follows by straightforward application of the formula in Lemma \ref{lem:hpst-htpy}.
\end{proof}

\begin{thm}\label{thm:aw-bar-dcsh}  For all $A, A'\in \ob \cat{Alg}_{R}$, the Alexander-Whitney map
$$f:\Bar (A\otimes A')\to \Bar A\otimes \Bar A'$$
is a DCSH map.
\end{thm}

\begin{proof}  We prove this proposition by applying Theorem \ref{thm:g-m} to the Eilenberg-Zilber data (\ref{eqn:em-bar}). Note first that $\Bar A\otimes \Bar A'$ and $\Bar(A\otimes A')$ are both connected, by definition of the bar construction.

Given a nonzero element $w=s(a_{1}\otimes a_{1})|\cdots |s(a_{n}\otimes a'_{n})\in\Bar_{n}(A\otimes A')$, let 
$$\zeta(w)=\#\{ i\mid |a_{i}|=0\} + \#\{j\mid |a'_{j}|=0\}.$$
Let $\zeta (0)= +\infty$.

Let $w=s(a_{1}\otimes a_{1})|\cdots |s(a_{n}\otimes a'_{n})\in\Bar_{n}(A\otimes A')$. If $\zeta(w)>n$, then there exists $j\in [1,n]$ such that $|a_{j}|=0=|a'_{j}|$ and therefore $w$ corresponds to a degenerate element in $\Bar_{\bullet}(A\otimes A')$.  Since $\Bar (A\otimes A')$ is the normalized complex associated to $\Bar_{\bullet}(A\otimes A')$, it  follows that $w=0$.
We therefore conclude that
\begin{equation}\label{eqn:w-cond}
0\not=w\in\Bar_{n}(A\otimes A')\;\Longrightarrow\; \zeta(w)\leq n.
\end{equation}

Define a bifiltration of $\Bar (A\otimes A')$ by
$$\mathcal F^{p,n}\big (\Bar (A\otimes A')\big)=\{w\in \Bar_{\leq n} (A\otimes A')\mid \zeta (w)\geq p\},$$
and consider the induced bifiltration
$$\mathcal F^{p,n}\big (\Bar (A\otimes A')^{\otimes k}\big)=\bigoplus _{p_{1}+\cdots +p_{k}=p\atop n_{1}+...+n_{k}=n} \mathcal F^{p_{1},n_{1}}\big (\Bar(A\otimes A')\big)\otimes \cdots \otimes \mathcal F^{p_{k},n_{k}}\big (\Bar(A\otimes A')\big).$$
It is easy to check  that the comultiplication $\Delta :\Bar (A\otimes A')\to \Bar (A\otimes A')\otimes \Bar (A\otimes A')$ is a bifiltered map. Moreover, it follows from implication (\ref{eqn:w-cond}) that
\begin{equation}\label{eqn:pn-cond}
\mathcal F^{p,n}\big (\Bar (A\otimes A')^{\otimes k}\big)=0 \text{ for all }p>n\text { and } k\geq 1.
\end{equation}

To prove local finiteness of the $F_{k}$'s associated to the SDR (\ref{eqn:em-bar}), we show that 
\begin{equation}\label{eqn:zeta}
\zeta \big(h(w)\big)\geq \zeta (w)+2
\end{equation}
for all $w\in  \Bar(A\otimes A')$.  It follows that
$$h\Big(\mathcal F^{p,n}\big (\Bar (A\otimes A')\Big)\subset \mathcal F^{p+2, n+1}\big (\Bar (A\otimes A') \big)$$
for all $p$ and $n$, since the formulas in Corollary \ref{cor:bar-htpy} imply that
$$h\big(\Bar_{n}(A\otimes A)\big) \subset \Bar_{n+1}(A\otimes A').$$
Consequently, if $\overline\Delta$ denotes the reduced comultiplication on $\Bar (A\otimes A')$, then
$$\overline\Delta h\Big(\mathcal F^{p,n}\big (\Bar (A\otimes A')\Big)\subset \mathcal F^{p+2, n+1}\big (\Bar (A\otimes A') \otimes \Bar (A\otimes A')\big),$$
which is the base step in an easy recursive argument showing that
$$H_{k+1}\Big(\mathcal F^{p,n}\big (\Bar (A\otimes A')\Big)\subset \mathcal F^{p+2k, n+k}\big (\Bar (A\otimes A')^{\otimes k}\big),$$
for all $k\geq 1$, where the map $H_{k+1}$ is defined as in (\ref{eqn:Hk}).

Equation (\ref{eqn:pn-cond}) therefore implies that for all $w\in \mathcal F^{p,n}\big(\Bar (A\otimes A')\big)$ and for all $k> n-p+1$,
$$F_{k}(w)=(\si f)^{\otimes k}\circ H_{k} (w)=0.$$
Local finiteness of the $F_{k}$'s, which allows us to apply Theorem \ref{thm:g-m} and therefore conclude that $f:\Bar (A\otimes A')\to \Bar (A) \otimes \Bar (A')$ is a DCSH map, is thus a consequence of inequality (\ref{eqn:zeta}).

To complete the proof, we must prove that inequality (\ref{eqn:zeta}) holds.
It follows immediately from inspection of the formulas in Corollary \ref{cor:bar-htpy} and for $\n$ (\ref{eqn:nabla}) that, in comparison with $w=s(a_{1}\otimes a_{1})|\cdots |s(a_{n}\otimes a'_{n})$, each summand in the expression 
$$s(a_{1}\otimes a'_{1})| \cdots |s(a_{m}\otimes a'_{m})|s(1\otimes a'_{m+1}\cdots a'_{r})|\n(sa_{m+1}|\cdots |sa_{r}\otimes sa'_{r+1}|\cdots |sa'_{n}),$$ 
if nonzero,
\begin{itemize}
\item  contains $(n-r)+(r-m)+1$ new 1's, inserted in the last $n-m$ terms, and
\item has lost at most $(r-m-1)$ 1's, in the process of multiplying $a'_{m+1}\cdots a'_{r}$.
\end{itemize}
We see thus that
\begin{align*}
\zeta \big(s(a_{1}\otimes a'_{1})|& \cdots |s(a_{m}\otimes a'_{m})|s(1\otimes a'_{m+1}\cdots a'_{r})|\n(sa_{m+1}|\cdots |sa_{r}\otimes sa'_{r+1}|\cdots |sa'_{n})\big)\\
&\geq \zeta (w)+ (n-m+1)- (r-m-1)\\
&=\zeta(w)+n-r+2\\
&\geq \zeta(w)+2.
\end{align*}
\end{proof}

Before proving Theorem \ref{thm:bar-hopf}, we establish a technical lemma that plays an important role in showing that $\ve_{H}:\Om \Bar H\to H$ is a coalgebra map. Recall the cartesian product of twisting cochains from Definition \ref{defn:cartesian}.

\begin {lem}\label{lem:h-delta-f}  Let $A$ and $A'$ be augmented chain algebras.  For all $n>1$, the composite
{\smaller$$\xymatrix{\Bar_{n}(A\otimes A')\ar[r]^(0.45)h&\Bar _{n+1}(A\otimes A')\ar [r]^(0.35)\Delta &\bigoplus _{\ell+m=n+1}\Bar_{\ell}(A\otimes A')\otimes \Bar _{m}(A\otimes A')\ar [d]^{f\otimes f}\\
&&\bigoplus _{\ell+m=n+1\atop {\ell'+\ell ''=\ell\atop m'+m''=m}}\Big(\Bar_{\ell'}A\otimes\Bar _{\ell ''}A'\Big) \otimes \Big(\Bar_{m'}A\otimes \Bar _{m''}A'\Big)\ar [d]^{(t_{\Bar}*t_{\Bar})\otimes (t_{\Bar}*t_{\Bar})}\\
&&(A\otimes A')\otimes (A\otimes A')}$$}

\noindent is equal to zero.  Moreover,
$$(t_{\Bar}*t_{\Bar})^{\otimes 2}\circ f^{\otimes 2}\circ \Delta \circ h\big(s(a\otimes a')\big)=\begin{cases}\pm(1\otimes a')\otimes (a\otimes 1)&: |a|\cdot |a'|>0\\ 0&:\text{else.}\end{cases}$$
\end{lem}

\begin{proof} Let $\delta=(t_{\Bar}*t_{\Bar})^{\otimes 2}\circ f^{\otimes 2}\circ \Delta \circ h$.
Recall that $t_{\Bar}(sa)=a$ for all $a$ in $A$ or $A'$, while $t_{\Bar }(sa_{1}|\cdots |sa_{n})=0$ for all $n>1$. 

If $n>1$ and thus $\ell '+\ell ''+m'+m''=n+1>2$, then  
\begin{itemize}
\item at least one of $\ell'$, $\ell''$, $m'$ and $m''$ is greater than 1, or
\item $\ell'+\ell''=2$ and  $m'+m''\leq 2$, or
\item $\ell'+\ell''\leq 2$ and $m'+m''=2$.
\end{itemize}
In the first case, the corresponding summand of $\delta$ is zero, since $\Bar _{\geq 2} A\subset \ker t_{\Bar }$ and similarly for $A'$.  In the second and third cases, the corresponding summand of $\delta$ is also zero, since
$$\Bar_{1}A\otimes \Bar _{1}A'=s\overline A\otimes s\overline A'\subset \ker (t_{\Bar }* t_{\Bar}).$$

The case $n=1$ is established by a straightforward calculation.
\end{proof}  

\begin{cor}\label{cor:ker-hk} Let $A$ and $A'$ be augmented chain algebras.  Consider the composite
 $$\Om\Bar (A\otimes A')\xrightarrow {\alpha_{F}}\Om (\Bar A\otimes \Bar A')\xrightarrow q \Om \Bar A\otimes \Om \Bar A'\xrightarrow {\ve_{A}\otimes \ve_{A'}} A\otimes A',$$
where $F:\Bar (A\otimes A')\to \Om (\Bar A\otimes \Bar A')$ is the twisting cochain of Theorem \ref{thm:aw-bar-dcsh}. 
For all $n>1$, 
$$\si \Bar_{n}(A\otimes A')\subset \ker \Big( (\ve_{A}\otimes \ve_{A'})q\alpha_{F}\Big).$$
\end{cor}

\begin{proof} Remark \ref{rmk:cartesian-twist} implies that 
$$\alpha_{t_{\Bar}*t_{\Bar}}=(\ve_{A}\otimes \ve_{A'})q:\Om (\Bar A\otimes \Bar A')\to A\otimes A'.$$
Furthermore, for all $w_{1},...,w_{n}\in \Bar A$ and $w'_{1},...,w'_{n}\in \Bar A'$,
$$\alpha_{t_{\Bar}*t_{\Bar}}\big(\si (w_{1}\otimes w'_{1})|\cdots |\si (w_{n}\otimes w'_{n})\big)=\pm a_{1}\cdots a_{n}\otimes a'_{1}\cdots a'_{n},$$
where (suppressing summation) $\alpha_{t_{\Bar}*t_{\Bar}}\big(\si (w_{j}\otimes w'_{j})\big)=a_{j}\otimes a'_{j}$ for all $1\leq j\leq n$, and the sign is determined by the Koszul rule.

Recall equation (\ref{eqn:Hk}), the definition of the operators $H_{k}$ associated to Eilenberg-Zilber data. A straightforward inductive argument, of which Lemma \ref{lem:h-delta-f} is the base step, shows that for all $n>1$
$$f^{\otimes k}\circ H_{k}\Big(\Bar_{n} (A\otimes A')\Big)\subset\ker (t_{\Bar}*t_{\Bar})^{\otimes k}.$$
 It then follows from the second half of Remark \ref{rmk:Fk} that 
$$\alpha_{F}\Big(\Bar_{n} (A\otimes A')\Big)\subset\ker \alpha_{t_{\Bar}*t_{\Bar}},$$
and we can conclude.
\end{proof}

\begin{proof}[Proof of Theorem \ref{thm:bar-hopf}] Let $H$ be a chain Hopf algebra.  From Theorem \ref{thm:aw-bar-dcsh} it follows that $\Bar H$ is a weak Alexander-Whitney coalgebra, where the chain algebra map realizing the DCSH structure of $\delta :H\to H\otimes H$ is the composite
$$\Om \Bar H \xrightarrow{\Om \Bar \delta} \Om \Bar (H\otimes H)\xrightarrow {\alpha _{F}} \Om (\Bar H\otimes \Bar H).$$

It remains to show that $(\Bar H, \alpha_{F}\circ \Om \Bar\delta)$ is actually an Alexander-Whitney coalgebra, i.e., that the comultiplication
$$\Om \Bar H \xrightarrow {\alpha_{F}\circ \Om \Bar\delta} \Om (\Bar H\otimes \Bar H)\xrightarrow q \Om \Bar H\otimes \Om \Bar H$$
is coassociative.   Essentially the same argument as in the proof of coassociativity of the canonical  diagonal on $\Om C_{*}K$ (Theorem 4.2 in \cite{hpst}) works here, since the comultiplication on the cobar constructions comes in both cases from the Alexander-Whitney map in the original Eilenberg-Mac Lane SDR (\ref{eqn:em-sdr}).

Let $\psi=q\circ \alpha_{F}\circ \Om\Bar \delta: \Om\Bar H\to \Om \Bar H\otimes \Om \Bar H$. To prove that $\ve_{H}:\Om \Bar H\to H$ is a coalgebra map, we must verify the following two claims.
\begin{enumerate}
\item $(\ve_{H}\otimes \ve _{H})\circ \psi \big(\si (sa)\big)= \delta (a)$ for all $a\in H$.
\item   $(\ve_{H}\otimes \ve _{H})\circ \psi \big(\si (sa_{1}|\cdots |sa_{n})\big)= 0$ for all $n>1$.
\end {enumerate}
Claim (2) is an immediate consequence of Corollary \ref{cor:ker-hk}.

On the other hand, Lemma \ref{lem:h-delta-f} implies that if $\delta (a)=a\otimes 1 +1\otimes a +a_{i}\otimes a^{i}$, then  
\begin{align*}
(\ve_{H}\otimes \ve_{H})&\psi \big(\si (sa)\big)\\
&=\alpha_{t_{\Bar}*t_{\Bar}}\alpha_{F}\big (\si (s(a\otimes 1) + s(1\otimes a) +s(a_{i}\otimes a^{i})\big)\\
&=\alpha_{t_{\Bar}*t_{\Bar}} \big(\si(sa\otimes 1) +\si (1\otimes sa)+ \si (1\otimes sa^{i})|\si (sa_{i}\otimes 1)\big)\\
&=\delta (a),
\end{align*}
and so Claim (1) holds as well.

To prove the dual result, concerning the cobar construction on $H$, note that if $H$ is connected and of finite type, then $\hom_{R} (\Om H, R)$ is isomorphic to $\Bar \hom_{R} (H, R)$, which is an Alexander-Whitney coalgebra by the argument above.  It follows that $\Om H$ is an Alexander-Whitney algebra and that $\eta_{H}:H\to \Bar \Om H$ is an algebra map.
\end{proof}

\begin{rmk} In \cite{kadeishvili}, Kadeishvili showed that if $H$ is a chain Hopf algebra, then $\Om H$ is a Hirsch algebra.  In proving above that $\Om H$ is an Alexander-Whitney algebra, we have established a stronger result, one that is necessary to proving the existence of multiplicative structure on $\cohoch(H)$.
\end{rmk} 

\section{The proof of Theorem \ref{thm:bigthm}}\label{app:acyclic}

We  prove Theorem \ref{thm:bigthm} by methods of acyclic models, for which the notion of {representability} is key.

\subsection{Representability}

The notion of {representability} used here is essentially that of \cite[Definition 28.1]{may}.  Note that if $X$ is any set, then $R[X]$ denotes the free $R$-module generated by $X$.

\begin{defn} Let $\cat C$ be any category, and let $\cat {Mod}_R$ be the category of $R$-modules.  Let $\mathfrak M$ be a set of objects in $\cat C$.

Let $\mathsf A: \cat C \to \cat {Mod}_R$ be a functor.  The \emph{$\mathfrak M$-linearization} of $\A$ is the functor
$$\widetilde {\mathsf A}_{\mathfrak M}: \cat C \to \cat {Mod}_R$$
defined on objects by
$$\widetilde {\mathsf A}_{\mathfrak M}(C) = \bigoplus _{M\in \mathfrak M}R\big[ \cat C(M,C)\big]\otimes_{R} \mathsf A(M)$$
and on morphisms by postcompostion in the obvious way.  

The  \emph{$\mathfrak M$-assembly map} is the natural transformation
$$\lambda_{\A,\mathfrak M}:\widetilde {\mathsf A}_{\mathfrak M}Ê\to A$$
the components of which are specified by
$$\lambda _{\A,\mathfrak M}(C)(a\otimes w)= \A(a)(w)$$
for all $a\in \cat C(M,C)$, $w\in \A(M)$ and $M\in \mathfrak M$.

The functor $\A$ is \emph{$\mathfrak M$-representable} if the $\M$-assembly map admits a section, i.e., if there is a natural transformaion $\sigma: \A \to \widetilde \A _{\M}$ such that $\lambda_{\A, \M} \circ \sigma =\Id_{\A}$.
\end{defn}

It is easy to see that the $\M$-assembly map is natural in $\A$, in the following sense.

\begin{lem}\label{lem:lift-transf}  If $\A, \A': \cat C \to \cat {Mod}_R$ are functors, and $\M$ is a set of objects in $\cat C$, then for every natural transformation $\alpha: \A \to \A'$, there is a natural transformation $\widetilde \alpha_{\M} :\widetilde\A_{\M}\to \widetilde \A'_{\M}$ such that 
$$\xymatrix{\widetilde \A_{\M}\ar [d]_{\lambda_{\A,\M}}\ar [r]^{\widetilde \alpha_{\M}}&\widetilde \A'_{\M}\ar [d]^{\lambda_{\A',\M}}\\
\A\ar [r]^\alpha&\A'}$$
commutes.
\end{lem}

\begin{proof} The component of $\widetilde \alpha_{\M}$ at $C$ is defined by
$$\widetilde \alpha_{\M}(C)(a\otimes w)= a\otimes \alpha (M)(w)$$
for all $a\in \cat C(M,C)$, $w\in \A (M)$ and $M\in \M$.  The diagram above then commutes as desired, since $\A'(a)\circ \alpha (M)=\alpha(C)\circ \A (a)$ for all $a\in \cat C(M,C)$.
\end{proof}

In the proof below, we need the following easy lemma from \cite{may} (Lemma 29.5).

\begin{lem}\label{lem:may} Let $\A, \A':\cat C \to \cat {Mod}_R$, and let $\M$ be a set of objects in $\cat C$.  If $\A$ is a retract of $\A'$, and $\A'$ is $\M$-representable, then so is $\A$.
\end{lem}

\subsection{Models and their acyclicity}

We now define the set of models relevant for the proof of Theorem \ref{thm:bigthm} and prove their acyclicity.

\begin{notn} Let $n\geq 0$, and let $\P_{n}: \cat S \to \cat {Mod}_R$ denote the degree $n$ component of the functor $\P$. 
Set 
$$\op N=\big\{(n_{1},...,n_{l})\mid n_{i}\in \mathbb N\,\forall i, l\geq 1\big\},$$
$$\op K=\big\{(k;\vec n)\mid k\in \mathbb N; \vec n\in \op N\big\},$$
and 
$$\op K_{n}=\big\{ (k;n_{1},..., n_{l})\in \op K\mid n_{i}\geq 1\,\forall i;  n=k+ \sum_{1\leq i \leq l} (n_{i}-1)\big\}.$$

For all simplicial maps $g:M\to M'$,
$$\P_{n}(g)=(C_{n}\S M\otimes R) \oplus \bigoplus _{(k;\vec n)\in \op K_{n}}C_{k} \S M \otimes \si \overline C_{n_{1}}\S M'\otimes \cdots \otimes \si \overline C_{n_{l}}\S M',$$
where $R$ is the ground ring.
\end{notn}

The following simplicial maps serve as models for $\P _{n}$ in our acyclic models argument. To simplify notation, we let 
$$\Delta[\vec n]=\Delta[n_{1}]\coprod \cdots \coprod \Delta [n_{l}]$$
for all $\vec n=(n_{1},...,n_{l})\in \op N$.
We then let 
$$\vp_{\vec n}: \emptyset \to \Delta [\vec n]$$
 denote the unique simplicial map from $\emptyset $ to $\Delta [\vec n]$. For all $(k;\vec n)\in \op K$, let 
 $$\vp_{k;\vec n}: \Delta[k] \hookrightarrow \Delta [k]\coprod \Delta [\vec n]$$
 denote the inclusion. Finally, let 
 $$\M =\{\Id_{\emptyset}\}\cup\{ \Id_{\Delta [n]}\mid n\geq 0\}\cup \{\vp_{\vec n}\mid \vec n \in \op N\} \cup \{ \vp_{k;\vec n}\mid (k;\vec n)\in \op K_{n}\}.$$

The easy step of the proof of Theorem \ref{thm:bigthm} is to prove the following lemma.

\begin{lem}\label{lem:acyclic} For all $\vp \in \M\smallsetminus \{\Id_\emptyset\}$, 
$$\ker \big(C_{*}p_{\vp}: \Q (\vp) \to C_{*} M\big)$$
is acyclic, where $M$ denotes the domain of $\vp$.
\end{lem}

\begin{proof} Since there is a short exact sequence of chain complexes
$$0\to \ker C_{*}p_{\vp}\hookrightarrow \Q (\vp)\xrightarrow {C_{*}p_{\vp}}C_{*}M\to 0,$$
it suffices to prove that $H_{*}p_{\vp}$ is an isomorphism.

For any $n\geq 0$,  $|\S \Delta[n]|\sim *$, whence $|\G\S \Delta[n]|\sim \Om|\S\Delta[n]| \sim *$, and so
$$H_{s}\big(\S \Delta[n]\big)\cong H_{s}\big(\G\S \Delta[n]\big )\cong\begin{cases} R&: s=0\\ 0&: s>0.\end{cases}$$
An easy spectral sequence argument shows that $H_{*}p_{\vp}$ is an isomorphism, for all $\vp$ of the form $\Id _{\Delta [n]}$ or $\vp_{k;\vec n}$.  

If $\vp=\vp_{\vec n}$ for some $n\in \op N$, then $|\Q(\vp)|\sim |\S \emptyset|\sim S^1$, and $C_{*}p_{\vp}$ again clearly induces an isomorphism in homology.  
\end{proof}

\subsection{Representability of the models}
 
The next result is a key step in the proof of Theorem \ref{thm:bigthm}.  Let $\M$ be defined as in the previous section.

\begin{prop} For all $n\geq 0$, the functor $\P_{n}$ is $\M$-representable.
\end{prop}

\begin{proof} Let $C^{u}_{*}$ denote the unnormalized chains functor.  Recall that the normalized chains functor $C_{*}$ is a retract of $C^{u}_{*}$ (Corollaries 22.2 and 22.3 in \cite{may}).  It follows that for all $n\geq 0$,  $\P_{n}$ is a retract of 
$$\Pu_{n}:\cat S\to \cat {Mod}_R,$$ 
which is defined on $g:M\to M'$ by
$$\Pu_{n}(g)=(C^{u}_{n}\S M\otimes R) \oplus \bigoplus _{(k;\vec n)\in \op K_{n}}C^{u}_{k} \S M \otimes \si \overline {C^{u}_{n_{1}}}\S M'\otimes \cdots \otimes \si \overline {C^{u}_{n_{l}}}\S M',$$
We prove that $\Pu_{n}$ is $\M$-representable and conclude by Lemma \ref{lem:may} that $\P_{n}$ is as well.

We need to show that the $\M$-assembly map 
$$\lambda: (\widetilde {\Pu_{n}})_{\M}\to \Pu_{n}$$ 
(where we have dropped some decorations in the notation) admits a section $\sigma$.  In other words, for each $g:M\to M'$, we need to define a homomorphism of abelian groups
$$\sigma (g): \Pu_{n}(g)\to (\widetilde {\Pu_{n}})_{\M}(g),$$
natural in $g$, such that $\lambda (g)\circ \sigma (g) =\Id.$  Since $\Pu_{n}(g)$ is a free $R$-module, it suffices to define $\sigma (g)$ on generators, then to extend additively.

If $w\in \Pu_{n}(g)$ is a generator, then either 
\begin{enumerate}
\item there exists $x\in \S M_{n}$ such that $w=x\otimes 1$, or  
\item there exists $(k;n_{1},...,n_{l})\in \op K_{n}$, $x\in \S M_{k}$ and $y_{i}\in \S M'_{n_{i}}$, for $1\leq i \leq l$,  such that
$$w=x \otimes \si y_{1}|\cdots | \si y_{l}.$$
\end{enumerate}

\noindent \textbf{Case (1):}  Let $n'=||x||$ (cf. Notation  \ref{notn:absvalue}).  
\begin{itemize}
\item If $n'\geq 0$, then write $x=(p,p',\bar x)$. Let $(a_{w},b_{w})$ denote the following morphism in $\cat S$ from $\Id_{\Delta[n']}$ to $g$.
$$\xymatrix{\Delta[n']\ar[d]_{\Id_{\Delta[n']}}\ar [rrr] ^{\op Y(\bar x)}&&& M\ar [d]^g\\
\Delta[n']\ar [rrr] ^(0.6){g\circ \op Y (\bar x) }&&& M'.}$$
Define 
$$\sigma (g)(w)\in R\big[\cat S(\Id_{\Delta[n']}, g)\big] \otimes \Pu _{n}(\Id_{\Delta[n']})\subset (\widetilde {\Pu_{n}})_{\M}(g)$$
by 
$$\sigma (g)(w)
=(a_{w}, b_{w})\otimes \big((p,p',\epsilon_{n'})\otimes 1).
$$
\medskip
\item If $k'=-1$, then either $x=s_{0}^ka_{0}$ or $x=(j, s_{0}^{k-j}c_{0})$ for some $1\leq j\leq k$. Let $(a_{w}, b_{w})$ denote the following morphism in $\cat S$ from $\Id _{\emptyset}$ to $g$.
$$\xymatrix{\emptyset\ar[d]_{\Id_{\emptyset}}\ar [rr] ^{u_{M}}&& M\ar [d]^g\\
 \emptyset\ar [rr] ^(0.55){ u_{M'}}&& M'.}$$
 Define 
$$\sigma (g)(w)\in R\big[\cat S(\Id_{\emptyset}, g)\big] \otimes \Pu _{n}(\Id_{\emptyset})\subset (\widetilde {\Pu_{n}})_{\M}(g)$$
by 
$$\sigma (g)(w)
=(a_{w}, b_{w})\otimes \big(x\otimes 1\big).
$$
This definition makes sense because $x$ is equal to a simplex present in $\S M$ for all $M$.
\end{itemize}

\noindent \textbf{Case (2):}  Let $k'= ||x||$, $ n_{i}'=||y_{i}||$ for all $1\leq i\leq l$.  Without loss of generality, we may assume that there exists $l'\in [0,l]$ such that $n_{i}'\geq 0$ if $i\leq l'$ and $n_{i}'=-1$ if $i> l'$.  Let $\vec n'=(n_{1}',...,n_{l'}')$.

\begin{itemize}
\item If $k'\geq 0$, then write $x=(p,p',\bar x)$ and $y_{{j}}=(q_{j},q_{j}',\bar y_{{j}})$ for all $1\leq j\leq l'$.  Let $(a_{w},b_{w})$ denote the following morphism in $\cat S$ from $\vp_{k';\vec n'}$ to $g$.
$$\xymatrix{\Delta[k']\ar[d]_{\vp_{k';\vec m'}}\ar [rrr] ^{\op Y(\bar x)}&&& M\ar [d]^g\\
\Delta[k']\coprod \Delta [\vec n']\ar [rrr] ^(0.6){g\circ \op Y (\bar x) + \sum _{1\leq j\leq l'}\op Y(\bar y_{{j}})}&&& M'.}$$
Define 
$$\sigma (g)(w)\in R\big[\cat S(\vp_{k';\vec n'}, g)\big] \otimes \Pu _{n}(\vp_{k';\vec n'})\subset (\widetilde {\Pu_{n}})_{\M}(g)$$
by 
\begin{align*}
\sigma (g)&(w)\\
&=(a_{w}, b_{w})\otimes \big((p,p',\epsilon_{k'})\otimes \si (q_{1},q'_{1}, \epsilon_{n_{1}'})|\cdots | \si (q_{l'}, q_{l'}', \epsilon_{n'_{l'}})|\si y_{l'+1}|\cdots |\si y_{l}\big).
\end{align*}
This definition makes sense because $||y_{i}||=-1$ for all $i>l'$, i.e.,  for all $i>l'$, either $y_{i}=s_{0}^{n_{i}}a_{0}$ or $y_{i}=(k_{i}, s_{0}^{n_{i}-k_{i}}c_{0})$ for some $k_{i}$, which are simplices present in every double suspension.
\medskip

\item If $k'=-1$, then either $x=s_{0}^ka_{0}$ or $x=(j, s_{0}^{k-j}c_{0})$ for some $1\leq j\leq k$.   We separate this case into two further parts.
\begin{itemize}
\item  If $l'>0$, then write $y_{{j}}=(q_{j},q_{j}',\bar y_{{j}})$ for all $1\leq j\leq l'$, and let $(a_{w},b_{w})$ denote the following morphism in $\cat S$ from $\vp_{\vec m'}$ to $g$.
$$\xymatrix{\emptyset\ar[d]_{\vp_{\vec m'}}\ar [rrr] ^{u_{M}}&&& M\ar [d]^g\\
 \Delta [\vec n']\ar [rrr] ^(0.55){ \sum _{1\leq j\leq l'}\op Y(\bar y_{{j}})}&&& M'.}$$
Define 
$$\sigma (g)(w)\in R\big[\cat S(\vp_{\vec n'}, g)\big] \otimes \Pu _{n}(\vp_{\vec n'})\subset (\widetilde {\Pu_{n}})_{\M}(g)$$
by 
\begin{align*}
\sigma (g)&(w)\\
&=(a_{w}, b_{w})\otimes \big(x\otimes \si (q_{1},q'_{1}, \epsilon_{n_{1}'})|\cdots | \si (q_{l'}, q_{l'}', \epsilon_{n'_{l'}})|\si y_{l'+1}|\cdots |\si y_{l}\big).
\end{align*}
This definition makes sense because $||x||=-1=||y_{i}||$ for all $i>l'$ (cf. above).

\item If $l'=0$, then  for each $1\leq i\leq l$, either $y_{i}=s_{0}^ka_{0}$ or $y_{i}=(j, s_{0}^{k-j}c_{0})$ for some $1\leq j\leq k$.  Let $(a_{w}, b_{w})$ denote the following morphism in $\cat S$ from $\Id _{\emptyset}$ to $g$.
$$\xymatrix{\emptyset\ar[d]_{\Id_{\emptyset}}\ar [rr] ^{u_{M}}&& M\ar [d]^g\\
 \emptyset\ar [rr] ^(0.55){ u_{M'}}&& M'.}$$
 Define 
$$\sigma (g)(w)\in R\big[\cat S(\Id_{\emptyset}, g)\big] \otimes \Pu _{n}(\Id_{\emptyset})\subset (\widetilde {\Pu_{n}})_{\M}(g)$$
by 
$$\sigma (g)(w)=(a_{w}, b_{w})\otimes \big(x\otimes \si y_{1}|\cdots | \si y_{l}\big).
$$
This definition makes sense because $||x||=-1=||y_{i}||$ for all $i$ (cf. above).
\end{itemize}
\end{itemize}

It follows easily from the observations in Remarks \ref{rmk:trivial} and \ref{rmk:factor}  that we have defined $\sigma(g)$ exactly so that
$$\lambda(g)\circ \sigma(g)=\Id_{(\Pu)^n_{k,\vec m}(g)}. $$
The naturality of $\sigma$ is a consequence of Remark \ref{rmk:trivial} and of the observation that for any simplicial morphism $a:M\to M'$ and any $z\in M_{n}$,
$$\op Y\big(a(z)\big)=a\circ \op Y(z): \Delta [n] \to M'.$$
\end{proof}

\subsection{The proof of Theorem \ref{thm:bigthm}}

Theorem \ref{thm:bigthm} follows immediately from the next proposition, by induction on degree.  The base of the induction is trivial, since $\P_{<0}(g)=0$.

\begin{prop}  Under the hypotheses of Theorem \ref{thm:bigthm}, let $\Theta$ denote the natural chain contraction of $\xi$. Suppose that for some $n\geq 0$ and for all simplicial maps $g:M\to M'$, there is an $R$-linear map of degree $+1$
$$\widehat \Theta(g): \P _{<n}(g)\to \Q _{<n+1}(g)$$
that is natural in $g$, agrees with the given chain homotopy for $g=\Id_{\emptyset}$, extends $\Theta(M')_{<n}$, and satisfies
$$C_{*}p_{g}\circ \widehat \Theta (g)=0\quad \text{and}\quad d_{\Q}\widehat \Theta(g)+\widehat \Theta(g)d_{\P}=\widehat \xi (g)|_{\P_{<n}(g)}.$$
Then  there is a natural $R$-linear extension of $\widehat \Theta(g)$ to 
$$\widehat \Theta(g): \P_{\leq n}(g)\to \Q_{\leq n+1}(g),$$
agreeing with the given chain homotopy for $g=\Id_{\emptyset}$ and such that 
$$C_{*}p_{g}\circ \widehat \Theta (g)=0\quad \text{and} \quad d_{\Q}\widehat \Theta(g)+\widehat \Theta(g)d_{\P}=\widehat \xi (g)|_{\P_{\leq n}(g)}.$$
\end{prop}

\begin{proof}   We begin by applying Lemma \ref{lem:lift-transf} to $\widehat \Theta: \P_{k-1}\to \Q_{k}$ for all $k\leq n$ and to  $\widehat \xi: \P_{k}\to \Q_{k}$, $d_{\P}:\P_{k}\to \P_{k-1}$ and $d_{\Q}:\Q_{k}\to \Q_{k-1}$ for all $k$, obtaining the following commuting diagrams of natural transformations. 
$$\xymatrix{(\widetilde\P_{k-1})_{\M}\ar [d]_{\lambda_{\M}}\ar [r]^{\widetilde\Theta_{\M}}& (\widetilde\Q_{k})_{\M}\ar [d]^{\lambda_{\M}}&&(\widetilde\P_{k})_{\M}\ar [d]_{\lambda_{\M}}\ar [r]^{\widetilde\xi_{\M}}& (\widetilde\Q_{k})_{\M}\ar [d]^{\lambda_{\M}}\\
\P_{k-1}\ar [r]^{\widehat \Theta}&\Q_{k}&&\P_{k}\ar [r]^{\widehat \xi}&\Q_{k}}$$
$$\xymatrix{(\widetilde\P_{k})_{\M}\ar [d]_{\lambda_{\M}}\ar [r]^{(\widetilde d_{\P})_{\M}}& (\widetilde\P_{k-1})_{\M}\ar [d]^{\lambda_{\M}}&&(\widetilde\Q_{k})_{\M}\ar [d]_{\lambda_{\M}}\ar [r]^{(\widetilde d_{\Q})_{\M}}& (\widetilde\Q_{k-1})_{\M}\ar [d]^{\lambda_{\M}}\\
\P_{k}\ar [r]^{d_{\P}}&\P_{k-1}&&\Q_{n}\ar [r]^{d_{\Q}}&\Q_{k-1}}$$
An easy calculation using the explicit definition of the lifted natural transformation in Lemma \ref{lem:lift-transf} shows that 
$$(\widetilde d_{\Q})_{\M}\widetilde\Theta_{\M} + \widetilde\Theta_{\M}(\widetilde d_{\P})_{\M} =\widetilde \xi: (\widetilde\P_{n-1})_{\M} \to (\widetilde\Q_{n-1})_{\M}$$
and that 
$$\im \widetilde \Theta _{\M}\subseteq \bigoplus_{\vp\in \M}R\big[\cat S(\vp, -)\big]\otimes \ker C_{*}p_{\vp}.$$

We now extend $\widetilde \Theta_{\M}$ over $(\widetilde\P_{n})_{\M}$, then use the representability of $\P_{n}$ to obtain the desired extension of $\widehat \Theta_{\M}$ over $\P_{n}$. The existence of the extension is an easy consequence of the acyclicity of $\ker C_{*}p_{\vp}$ for all $\vp \in \M \smallsetminus \{\Id_{\emptyset}\}$, together with the hypothesis that $\widehat \xi (Id_{\emptyset})$ is chain homotopic to 0.

We can thus construct $\widetilde \Theta ': (\widetilde\P_{n})_{\M}\to (\widetilde\Q_{n+1})_{\M}$ such that
$$(\widetilde d_{\Q})_{\M}\widetilde\Theta'_{\M} + \widetilde\Theta_{\M}(\widetilde d_{\P})_{\M} =\widetilde \xi:(\widetilde\P_{n})_{\M} \to (\widetilde\Q_{n})_{\M}.$$
Now extend $\widehat \Theta$ to degree $n$ by
$$\widehat \Theta=\lambda_{\M}\widetilde \Theta'\sigma: \P_{n}\to \Q_{n+1},$$
where $\sigma: \P_{n}\to (\widetilde\P_{n})_{\M}$ is a section of $\lambda_{\M}$, the existence of which follows from the representability of $\P_{n}$.  Observe first that $\im \widehat \Theta \subseteq \ker C_{*}p$, by naturality of $p$. The following sequence of equalities of morphisms from $\P_{n}$ to $\Q_{n}$ then enables us to conclude.
\begin{align*}
d_{\Q}\widehat \Theta +\widehat \Theta d_{\P}&=d_{\Q}\lambda_{\M}\widetilde \Theta'\sigma + \widehat \Theta d_{\P}\\
&=\lambda_{\M}(\widetilde d_{\Q})_{\M}\widetilde\Theta'\sigma + \widehat \Theta d_{\P}\\
&=-\lambda_{\M}\widetilde \Theta _{\M}(\widetilde d_{\P})_{\M}\sigma + \lambda _{\M}\widetilde \xi \sigma + \widehat \Theta d_{\P}\\
&=-\widehat \Theta \lambda _{\M}(\widetilde d_{\P})_{\M}\sigma + \widehat \xi \lambda _{\M}\sigma + \widehat \Theta d_{\P}\\
&=-\widehat \Theta d_{\P}\lambda _{\M}\sigma + \widehat \xi + \widehat \Theta d_{\P}\\
&=\widehat \xi.
\end{align*}
\end{proof}

\section{Various useful simplicial constructions}\label{sec:simpl-constr}

We recall here the simplicial constructions that play a crucial role in section \ref{sec:geometry}.

\subsection{Twisting}

\begin{defn} Let $K$ be a simplicial set and $G$ a simplicial group, 
where the neutral element in any dimension is noted $e$. A 
degree $-1$ map of graded sets $\tau :K\to G$ is a \emph{twisting 
function} if 
\begin{align*}
d _{0}\tau (x)&=\bigl(\tau (d_{0}x)\bigr)^{-1}\tau (d
_{1}x)\\
d _{i}\tau (x)&=\tau (d
_{i+1}x)\quad i>0\\
s_{i}\tau (x)&=\tau (s_{i+1}x)\quad i\geq 0\\
\tau (s_{0}x)&=e
\end{align*}
for all $x\in K$.
\end{defn}

\begin{rmk}\label{rmk:twisting-fcn} Let $K$ be a reduced simplicial set, and let $\G K$ denote the Kan simplicial loop group on $K$ \cite{may}.  Let $\bar x\in (\G K)_{n-1}$ denote a free group generator, corresponding to $x\in K_{n}$. There is a universal, canonical twisting function $\tau_{K}:K\to \G K$, given by $\tau_{K}(x)=\bar x$.
\end{rmk}

Given a twisting function $\tau :K\to G$, where $G$ operates on the left on a simplicial set $L$, 
we can construct a \emph{twisted cartesian product} of $K$ 
and $L$, denoted $K\times _{\tau}L$, which is a simplicial set such 
that
$(K\underset{\tau}\times L)_{n}=K_{n}\times L_{n}$, with faces and 
degeneracies given by
\begin{align*} 
d_{0}(x,y)&=(d _{0}x,\tau (x)\cdot d _{0}y)\\
d _{i}(x,y)&=(d _{i}x,d_{i}y)\quad i>0\\
s_{i}(x,y)&=(s _{i}x,s_{i}y)\quad i\geq 0.
\end{align*}
If $L$ is a Kan complex, then the projection $K\underset{\tau}\times L\to K$ is 
a Kan fibration \cite {may}.

\subsection{Classifying spaces}

\begin{defn} Let $G$ be a simplicial group, where the neutral element in each dimension is denoted $e$.  The \emph{Kan classifying space} of $G$ is the simplicial set $\overline \W G$ such that $\overline \W G_{0}=\{(\;)\}$ is a singleton and for all $n>0$,
$$\overline \W G_{n}=G_{0}\times \cdots \times G_{n-1},$$
with face maps given by
$$d_{i}(a_{0},...,a_{n-1})=\begin{cases} (a_{0},...,a_{n-2})&: i=0\\ 
(a_{0},...,a_{n-i-2}, a_{n-i-1}\cdot d_{0}a_{n-i}, d_{1}a_{n-i+1},..., d_{i-1}a_{n-1})&: 0<i<n\\ 
(d_{1}a_{1},...,d_{n-1}a_{n-1})&: i=n\end{cases}$$
and degeneracies given by $s_{0}\big ( (\;)\big)=(e)$, while
$$s_{i}(a_{0},...,a_{n-1})=\begin{cases} (a_{0},...,a_{n-1},e)&:i=0\\ (a_{0},...,a_{n-i-1}, e,s_{0}a_{n-i},...,s_{i-1}a_{n-1})&: 0<i<n\\(e, s_{0} a_{0},...,s_{n-1}a_{n-1})&: i=n.\end{cases}
$$
\end{defn}

\begin{rmk} For the sake of coherence with the definition of the totalization functor below, our definition of the Kan classifying space differs from that in \cite{may} by a permutation of the factors.
\end{rmk}

\begin{rmk}\label{rmk:couniv-twist} The classifying space $\overline \W G$ deserves its name, as homotopy classes of simplicial maps into $\overline \W G$ classify twisted cartesian products with fiber $G$.  The universal $G$-bundle is a twisted cartesian product 
$$G\hookrightarrow \overline \W G\underset {\upsilon_{G}}\times G \twoheadrightarrow \overline \W G,$$
where $\upsilon_{G}: \overline \W  G\to G$ is the natural (couniversal) twisting function defined by
$$\upsilon_{G} (a_{0},...,a_{n-1})=a_{n-1},$$
and $G$ acts on itself by left multiplication.
\end{rmk}

\begin{rmk}\label{rmk:eta-equiv} The classifying space functor $\overline \W $ is right adjoint to the Kan loop group functor $\G $. Furthermore, the unit map 
$\eta_{K}: K\to \overline \W \G K$,
which is defined by
$$x\in K_{n}\Longrightarrow \eta_{K}(x)=(\overline{d_{0}^{n-1}x}, \overline{d_{0}^{n-2}x}, ..., \overline x),$$
is a weak equivalence for all reduced simplicial sets $K$.
\end{rmk}

\begin{defn}  Let $G$ be a simplicial group, where the neutral element in each dimension is denoted $e$. The \emph{bar construction} (or \emph{nerve}) on $G$ is the bisimplicial set $\B G$ such that 
$$\B G_{p,q}=(G_{p})^{\times q},$$ 
with horizontal and vertical face maps
$$d^h_{i}=(d_{i})^{\times q}:\B G_{p,q}\to \B G_{p-1,q}$$
and
$$d^v_{i}:\B G_{p,q}\to \B G_{p,q-1}: (a_{1},...,a_{q})\mapsto\begin{cases} (a_{2},...,a_{q}) &: i=0\\ (a_{1},...,a_{i}\cdot a_{i+1},...,a_{q})&: 0< i<q\\ (a_{1},...,a_{q-1}) &: i=q\end{cases}$$
and horizontal and vertical degeneracies
$$s^h_{i}=(s_{i})^{\times q}:\B G_{p,q}\to \B G_{p+1,q}$$
and
$$s^v_{i}:\B G_{p,q}\to \B G_{p,q+1}: (a_{1},...,a_{q})\mapsto (a_{1},...,a_{i},e, a_{i+1},...,a_{q})$$
for all $0\leq i\leq q$.
\end{defn}

\begin{defn} Let $G$ be a simplicial group, where the neutral element in each dimension is denoted $e$. The \emph{cyclic bar construction on $G$}  is the bisimplicial set $\Z G$ such that 
$$\Z G_{p,q}=(G_{p})^{\times q}\times G_{p},$$ 
with horizontal and vertical face maps
$$d^h_{i}=(d_{i})^{\times q+1}:\Z G_{p,q}\to \Z G_{p-1,q}$$
and
{\smaller{$$d^v_{i}:\Z G_{p,q}\to \Z G_{p,q-1}: (a_{1},...,a_{q},b)\mapsto\begin{cases} (a_{2},...,a_{q},b\cdot a_{1}) &: i=0\\ (a_{1},...,a_{i}\cdot a_{i+1},...,a_{q},b)&: 0< i<q\\ (a_{1},...,a_{q-1}, a_{q}\cdot b) &: i=q\end{cases}$$}}

\noindent and horizontal and vertical degeneracies
$$s^h_{i}=(s_{i})^{\times q+1}:\Z G_{p,q}\to \Z G_{p+1,q}$$
and
$$ s^v_{i}:\Z G_{p,q}\to \Z G_{p,q+1}: (a_{1},...,a_{q},b)\mapsto (a_{1},...,a_{i},e, a_{i+1},...,a_{q},b)$$
for all $0\leq i\leq q$.
\end{defn}

\begin{notn} Let $\pi: \Z G\to \B G$ denote the obvious projection map.  Note that the fiber of $\pi$ over the basepoint (neutral element) in $\B G$ is isomorphic to  $\C G$, the bisimplicial set that is constant in the vertical direction, with $\C G_{p,q}=G_{p}$ for all $p,q$.
\end{notn}

\begin{rmk} There are obvious analogous constructions of the ordinary and cyclic bar constructions for  topological groups.  We use the same notation for these constructions.
\end{rmk}

\begin{defn}\cite{artin-mazur}\label{defn:artin-mazur}  The (Artin-Mazur) \emph{totalization functor}, $\tot$, from bisimplicial sets to simplicial sets is defined as follows on objects.  If $K$ is a bisimplicial set, then for all $n\geq 0$,
$$\tot(K)_n=\{(x_{0},...,x_{n})\in \prod_{i=0}^n K_{i,n-i}\mid d_{0}^v x_{i}=d^h_{i+1}x_{i+1} \quad \forall\; 0\leq i<n\}.$$
Faces and degeneracies are given by
$$d_{i}(x_{0},...,x_{n})=(d_{i}^vx_{0},...,d_{1}^vx_{i-1},d_{i}^hx_{i+1},...,d_{i}^hx_{n})$$
and 
$$s_{i}(x_{0},...,x_{n})=(s_{i}^vx_{0},...,s_{0}^vx_{i},s_{i}^hx_{i},...,s_{i}^hx_{n})$$
for all $0\leq i\leq n$.
\end{defn}

\begin{rmk}\label{rmk:totalize} It is very easy to prove that $\tot \C K\cong K$ for every simplicial set $K$.  A somewhat more difficult computation shows that $\tot \B G\cong \overline \W G$ as well, which has been known for some time \cite{cegarra-remedios}.  
\end{rmk}

\subsection{Suspensions}\label{app:susp}

In this article we work with both reduced and unreduced simplicial suspension functors.  For the definition of the \emph{reduced simplicial suspension}
$$\E : \cat {sSet}_{*}\to \cat {sSet}_{0},$$
 we refer the reader to \cite[Definition 27.6]{may} and recall only that   if $L$ is a pointed simplicial set, with basepoint $x_{0}$, then $(\E L)_{0}=\{a_{0}\}$, and for $n>0$, 
$$(\E L)_{n}=\{s_{0}^na_{0}\}\amalg\coprod_{1\leq k\leq n} \{k\} \times L_{n-k}/\sim,$$
where $(k, s_{0}^{n-k}x_{0})\sim s_{0}^na_{0}$ for all $1\leq k\leq n$.

A useful formulation of the \emph{unreduced simplicial suspension} 
$$\Eu:\cat {sSet}\to \cat {sSet}_{*}$$
can be derived from the unreduced cone construction $M\hookrightarrow \widetilde C M$ in \cite[Example 2.16]{wu}, by setting $\Eu M= \widetilde CM/M$.  In particular, if $M$ is any simplicial set, then $(\Eu M)_{0}=\{b_{0},c_{0}\}$, where $b_{0}$ is the basepoint, and $c_{0}$ is the cone point, and for $n>0$,
$$(\Eu M)_{n}=\{s_{0}^nb_{0}, s_{0}^n c_{0}\}\amalg\coprod_{1\leq k\leq n} \{k\} \times M_{n-k}.$$ 

\begin{rmk}  Let $L$ be a pointed simplicial set, and let $M$ be any simplicial set. An easy calculation shows that every positive-degree element of  either $C_{*}\E L$ or $C_{*}\Eu M$ is primitive.
\end{rmk} 

The \emph{double suspension functor} for unpointed simplicial sets is 
 $$\S=\E \Eu:\cat {sSet}\to \cat{sSet}_{0}.$$
 The description above of $\E $ and $\Eu$  implies that $(\S M)_{0}=\{a_{0}\}$, $(\S M)_{1}=\{s_{0}a_{0}, (1, c_{0})\}$,  and for $n>1$, 
$$(\S M)_{n}= \{s_{0}^na_{0}, (k,s_{0}^{n-k}b_{0}),(k,s_{0}^{n-k}c_{0})\mid 1\leq k\leq n \}\amalg \coprod_{k+l+m=n} \{(k,l)\} \times M_{m}/\sim,$$
where $(k, s_{0}^{n-k}b_{0})\sim s_{0}^na_{0}$ for all $1\leq k\leq n$.  In particular, if $M=\emptyset$, the empty simplicial set, then $\S \emptyset$ is a simplicial model for the circle,  with
$$(\S \emptyset)_{n}=\{s_{0}^na_{0}, (k,s_{0}^{n-k}c_{0})\mid 1\leq k\leq n \}$$ for all $n\geq 1$.  

\begin{notn}\label{notn:absvalue}  For any simplicial set $M$, let $||-||: \S M \to \mathbb Z$ denote the function defined by 
$$x\in M_{m}\Longrightarrow ||(k,l, x)||=m, \forall k,l,$$
and
$$||s_{0}^na_{0}||=-1=||(k, s_{0}^{n-k}c_{0})||, \forall n,k.$$
\end{notn}

\begin{rmk}\label{rmk:trivial} It is trivial, but nonetheless useful, to observe that for {any} simplicial map $g:M\to M'$, the induced map
$\S g: \S M\to \S M'$ satisfies $\S g(s_{0}^na_{0})=s_{0}^na_{0}$, $\S g(k,s_{0}^{n-k}c_{0})=(k,s_{0}^{n-k}c_{0})$ and $\S g(k,l,x)=\big(k,l,g(x)\big)$ for all $n$, $k$, $l$ and $x$.
\end{rmk}

\begin{notn}  For any $n\geq 0$, let $\epsilon_{n}$ denote the unique nondegenerate $n$-simplex of $\Delta[n]$.

If $M$ is a simplicial set, and $x\in M_{n}$, let $ \op Y(x):\Delta[n]\to M$ denote the simplicial map corresponding to $x$ under the Yoneda isomorphism, i.e., $\op Y(x)(\epsilon_{n})=x$.
\end{notn}

\begin{rmk} \label{rmk:factor} Observe that if $M$ is a simplicial set, and $(k,l,x)\in (\S M)_{n}$, then 
$$\xymatrix{\Delta [n]\ar[rr]^{\op Y{(k,l,x)}}\ar [dr]_{\op Y{(k,l,\epsilon_{m})}}&&\S M\\  &\S \Delta [m] \ar [ur]_{\S \op Y( x)}}$$
commutes, where $m=n-k-l$.  There are, moreover, factorizations
$$\xymatrix{\Delta [n]\ar[rr]^{\op Y{(s_{0}^na_{0})}}\ar [dr]_{\op Y{(s_{0}^na_{0})}}&&\S M\\  &\S \emptyset \ar [ur]_{\S u_{M}}}$$
and 
$$\xymatrix{\Delta [n]\ar[rr]^{\op Y{(k,s_{0}^{n-k}c_{0})}}\ar [dr]_{\op Y{(k,s_{0}^{n-k}c_{0})}}&&\S M\\  &\S \emptyset \ar [ur]_{\S u_{M}}},$$
 for all $0\leq k\leq n$, where $u_{M}$ is the unique simplicial map from the empty simplicial set $\emptyset$ to $M$.
\end{rmk}

 \bibliographystyle{amsplain}
\bibliography{powermap}
\end{document}